\pdfoutput=1
\documentclass{amsart}
\usepackage[utf8]{inputenc}

\usepackage[margin=1in]{geometry}
\usepackage[expansion=false]{microtype}

\usepackage{amsmath, amsfonts, amssymb, amsthm, amsopn, color}
\usepackage{eucal}

\usepackage{tikz}
\usetikzlibrary{diagrams}

\usepackage[pdfusetitle,unicode,hidelinks]{hyperref}

\newcommand{\ie}{{\sfcode`\.1000 i.e.}}

\numberwithin{equation}{section}

\theoremstyle{plain}
\newtheorem*{theorem*}{Theorem}

\newtheorem{theorem}[equation]{Theorem}
\newtheorem{proposition}[equation]{Proposition}
\newtheorem{lemma}[equation]{Lemma}
\newtheorem{corollary}[equation]{Corollary}

\theoremstyle{definition}
\newtheorem{definition}[equation]{Definition}

\theoremstyle{remark}
\newtheorem{remark}[equation]{Remark}

\let\scr=\mathcal
\let\bb=\mathbb
\def\N{\bb N}
\def\Z{\bb Z}

\def\A{\bb A}

\def\1{\mathbf{1}}

\let\del=\partial

\let\into=\hookrightarrow

\let\tens=\otimes

\def\ph{\mathord-}

\def\id{\mathrm{id}}
\DeclareMathOperator{\Hom}{Hom}

\DeclareMathOperator{\Fun}{Fun}
\DeclareMathOperator{\End}{End}
\DeclareMathOperator{\Aut}{Aut}
\DeclareMathOperator{\Map}{Map{}}

\DeclareMathOperator{\Spec}{Spec}

\def\et{\mathrm{\acute et}}

\DeclareMathOperator{\Tr}{Tr}

\def\ch{\mathrm{ch}}

\def\op{\mathrm{op}}

\def\Cat{\mathrm{Cat}{}}
\def\CAT{\mathbf{Cat}{}}

\def\ev{\mathrm{ev}}
\def\coev{\mathrm{coev}}

\def\HH{\mathrm{HH}}

\def\CAlg{\mathrm{CAlg}{}}
\def\CMon{\mathrm{CMon}{}}

\def\Mod{\mathrm{Mod}{}}

\let\lim=\relax
\DeclareMathOperator*{\lim}{lim}
\DeclareMathOperator*{\colim}{colim}

\def\h{\mathrm h}
\def\St{\mathrm{St}}
\def\Pr{\scr P\mathrm r}
\def\PR{\mathbf{Pr}{}}
\def\L{\mathrm{L}}
\def\R{\mathrm{R}}
\def\ex{\mathrm{ex}}
\def\perf{\mathrm{perf}}

\def\Ind{\mathrm{Ind}}

\def\rig{\mathrm{rig}}

\def\lax{\mathrm{lax}}
\def\oplax{\mathrm{oplax}}

\DeclareMathOperator{\maps}{maps}
\def\Fr{\mathrm{Fr}}
\def\Fin{\mathrm{Fin}}
\def\Sp{\mathrm{Sp}{}}
\def\pre{\mathrm{pre}}
\def\QCoh{\mathrm{QCoh}}
\def\Perf{\mathrm{Perf}}
\def\Mot{\mathrm{Mot}}
\def\MOT{\mathbb M\mathrm{ot}}

\def\K{\mathbb{K}}
\def\HH{\mathrm{HH}}

\def\Aff{\mathrm{Aff}{}}
\def\PrSt{\mathrm{PrStk}{}}

\def\Mor{\mathrm{Mor}}
\def\sat{\mathrm{sat}}
\def\Tan{\mathrm{Tan}}
\def\loc{\mathrm{loc}}
\def\add{\mathrm{add}}

\def\Ar{\mathrm{Ar}}

\newcommand{\ccat}[2]{\mathop{\mbox{$ \mathrm{Cat^{#1}(#2)}$}}}

\newcommand{\cA}{\mathcal{A}}
\newcommand{\cB}{\mathcal{B}}
\newcommand{\cD}{\mathcal{D}}
\newcommand{\cE}{\mathcal{E}}
\newcommand{\cL}{\mathcal{L}}
\newcommand{\cO}{\mathcal{O}}
\newcommand{\cP}{\mathcal{P}}
\newcommand{\cC}{\mathcal{C}}

\newcommand{\cS}{\mathcal{S}}

\newcommand{\cU}{\mathcal{U}}
\newcommand{\cI}{\mathcal{I}}

\newcommand{\Pre}{\mathcal{P}}

\DeclareMathOperator{\Lax}{Lax}

\usepackage[all]{xy}

\begin{document}
\title{Higher traces, noncommutative motives, and the categorified Chern character}

\begin{abstract}
We propose a categorification of the Chern character 
that refines earlier work of To\"en and Vezzosi and of Ganter and Kapranov.  If $X$ is an algebraic stack, 
our categorified  Chern character is a symmetric monoidal functor from a category of mixed noncommutative motives over $X$, which we introduce,   
to $S^1$-equivariant perfect complexes on the derived free loop stack $\cL X$.  As an application of the theory, we
 show that To\"en and Vezzosi's secondary Chern character factors through secondary $K$-theory. Our techniques depend on a careful investigation of the functoriality of traces in symmetric monoidal $(\infty,n)$-categories, which is of independent interest.
\end{abstract}

\author{Marc Hoyois}
\address{Marc Hoyois, Massachusetts Institute of Technology, Cambridge, MA, USA}
\email{hoyois@mit.edu}

\author{Sarah Scherotzke}
\address{Sarah Scherotzke,
Mathematical Institute of the University of Bonn,
Endenicher Allee 60,
53115 Bonn,
Germany}
\email{sarah@math.uni-bonn.de}

\author[Sibilla]{Nicol\`o Sibilla}
\address{School of Mathematics, Statistics and Actuarial Sciences\\ 
Cornwallis Building\\ 
University of Kent\\ 
Canterbury, Kent CT2 7NF\\
UK}
\email{\href{N.Sibilla@kent.ac.uk}{N.Sibilla@kent.ac.uk}}
\maketitle

{\small \tableofcontents}

\section{Introduction}

In this paper,
we develop a formalism of higher categorical traces and use it
to refine in various ways the categorified character theory developed by Toën and Vezzosi in \cite{TV1,TV2}.
Along the way we introduce a theory of relative noncommutative motives and generalize work of 
 Cisinski and Tabuada \cite{CT}, and Blumberg, Gepner, and Tabuada \cite{BGT}. 
 We start by placing our results in the context of To\"en and Vezzosi's work on secondary $K$-theory, 
and of the categorified homological algebra that emerges from the work of Ben-Zvi, Francis, and Nadler \cite{BFN}.

\subsection{Secondary $K$-theory and the secondary Chern character}
\label{sub:K2intro}
The algebraic $K$-theory of a scheme or stack $X$ measures 
the geometry of algebraic vector bundles on $X$. 
The analogy with topological $K$-theory suggests that algebraic $K$-theory probes 
only the first chromatic layer of the geometry of $X$. In the stable homotopy category, homotopy theorists have developed a rich picture of 
the chromatic hierarchy of homology theories. 
In particular, homology theories of chromatic level two have been 
the focus of intense investigation ever since work of Landweber, Stong, and Ravenel on elliptic cohomology in the late 80's.  An important insight emerging from topology is that climbing up the chromatic ladder 
is related to studying invariants of spaces that are higher-categorical in nature. A vast literature is  devoted to investigating cohomology theories of chromatic level two that are related to elliptic cohomology and measure the geometry of higher-categorical analogues of vector bundles, see for instance \cite{BDR}.

Motivated by these ideas from homotopy theory, To\"en and Vezzosi introduced in \cite{TV1} the notion of secondary $K$-theory of schemes and stacks.
Their work hinges on a categorification of coherent sheaf theory where the role of the structure sheaf is taken up by the sheaf of symmetric monoidal  
$\infty$-categories $\Perf(-)$ on $X$, which maps affine open subschemes to their category of perfect complexes.  Categorical sheaves are sheaves of small stable $\infty$-categories on $X$ that are tensored over $\Perf(-)$. By results of Gaitsgory \cite{Ga1}, they can often be described in a global way as \emph{$\Perf(X)$-linear} $\infty$-categories: that is, as $\infty$-categories tensored over 
$\Perf(X)$  (this holds for instance when $X$ is a quasi-compact quasi-separated scheme, or a semi-separated Artin stack of finite type over a field of characteristic zero). The definition of the secondary $K$-theory spectrum of 
$X$, denoted by 
$K^{(2)}(X)$, is closely patterned after classical  algebraic $K$-theory. In particular, its group of connected components  
$K^{(2)}_0(X) = \pi_0 K^{(2)}(X)$  
is the free abelian group on the set of equivalence classes of dualizable categorical sheaves on $X$, modulo the relations 
\[
[\cB] = [\cA] + [\cC], 
\]
where $\cA \to \cB \to \cC$ is a Verdier localization. 

Secondary $K$-theory is an intricate invariant of $X$ and is related to several other fundamental invariants.
Derived Azumaya algebras over $X$ are examples of dualizable categorical sheaves, and this gives rise to a multiplicative map from the algebraic Brauer space $\mathrm{Br}_\mathrm{alg}(X)$ of Antieau--Gepner \cite{AG} to $K^{(2)}(X)$, analogous to the multiplicative map from the Picard space of line bundles to ordinary $K$-theory. When $X$ is a scheme, this induces a map from the cohomological Brauer group $H^2_\et(X,\mathbb G_m)$ to $K^{(2)}_0(X)$. Tabuada has shown that this map is nontrivial in many cases, and that it is even injective when $X=\Spec k$ for $k$ a field of characteristic zero \cite{TabuadaK2}.
In that case, Bondal, Larsen, and Lunts \cite{BLL} have constructed a motivic measure on $k$-varieties with value in the ring $K^{(2)}_0(k)$ (under a different name, see \cite{TabuadaK2}): they showed that the assignment $X\mapsto \Perf(X)$, for $X$ a smooth projective variety, induces a ring homomorphism $K_0(\mathrm{Var}_k)/(\mathbb L-1)\to K_0^{(2)}(k)$, which is again nontrivial (although not injective). Finally, secondary $K$-theory is closely related to \emph{iterated} $K$-theory (see Remark~\ref{rmk:KK} for a precise statement).
From this perspective, the analogy with elliptic cohomology is made precise by the work of Ausoni and Rognes \cite{AR,Ausoni}, who showed that $K(K(\mathbb C))$ is a spectrum of telescopic complexity $2$ at primes $\geq 5$. 

One of the main results of \cite{TV2} is the construction of a secondary Chern character for dualizable categorical sheaves on a derived $k$-stack $X$. 
To\"en and Vezzosi achieve this by developing a general formalism of Chern 
characters using $S^1$-invariant traces that applies also to the classical Chern character in ordinary $K$-theory. Let 
\[
\cL X = \mathrm{\mathbb{R} Hom } (S^1,X) \quad\text{and}\quad \cL^{2}X = \mathrm{\mathbb{R} Hom } (S^1 \times S^1,X)
\] 
be the free loop stack and the double free loop stack of $X$. 
When $k$ is a field of characteristic zero and $X$ is a smooth $k$-scheme, the $E_\infty$-ring spectrum $\cO(\cL X)^{hS^1}$ of $S^1$-invariant functions on $\cL X$ is closely related to the de Rham cohomology of $X$: there is an equivalence
\[
\cO(\cL X)^{hS^1}[u^{-1}] \simeq H^\mathrm{per}_\mathrm{dR}(X),
\]
where $u$ is a generator of the cohomology of $BS^1$ (in degree $-2$) and $H_\mathrm{dR}^\mathrm{per}(X)$ is the $2$-periodization of the de Rham complex of $X$ over $k$ (see \cite[\S3]{TV}).
From this perspective, the classical Chern character 
\[
K (X) \to H^\mathrm{per}_\mathrm{dR}(X) \simeq \cO(\cL X)^{hS^1}[u^{-1}]
\]
sends a vector bundle over $X$ to the trace of the canonical monodromy operator on its pullback to $\cL X$ (see \cite[Appendix B]{TV2}). The fact that this trace is a homotopy $S^1$-fixed point follows from a general $S^1$-invariance property of the trace, which is a consequence of the $1$-dimensional cobordism hypothesis \cite{Lurie}.

The construction of the secondary Chern character in  \cite{TV2} is in keeping with this picture of Chern characters as traces of monodromy operators. For $X$ a derived $k$-stack, To\"en and Vezzosi's secondary Chern character is a map
\begin{equation}\label{eqn:ch2intro}
\iota_0\Cat^\sat(X) \rightarrow \Omega^\infty\cO(\cL^{2} X)^{h(S^1 \times S^1)},
\end{equation}
where $\iota_0\Cat^\sat(X)$ is the $\infty$-groupoid of dualizable categorical sheaves on $X$ (called sheaves of saturated dg-categories in \cite{TV2}). This map sends a dualizable categorical sheaf on $X$ to the 2-fold trace of the pair of commuting monodromy operators on its pullback to $\cL^2 X$.

One of our main applications is  
that~\eqref{eqn:ch2intro} factors canonically through the secondary $K$-theory of $X$, and even through a nonconnective version of it, denoted by $\K^{(2)}(X)$:

\begin{theorem}[see Theorem~\ref{thrm:TV2}]
\label{thrm:mainsecond}
The Toën–Vezzosi secondary Chern character is refined by a morphism of $E_\infty$-ring spectra
\[
\mathrm{\mathbb{K}}^{(2)}(X) \to \cO(\cL^{2} X)^{h(S^1 \times S^1)}.
\] 
\end{theorem}  

This theorem is a consequence of a more fundamental result that we discuss in the next subsection.
When $X=BG$, $\cO(\cL^{2} X)$ is the set of conjugation-invariant functions on commuting pairs of elements of $G$. In that case, the secondary Chern character of Theorem~\ref{thrm:mainsecond} is reminiscent of the character constructed by Hopkins, Kuhn, and Ravenel \cite{hopkins1992morava} for Morava $E$-theory at height $2$, and it suggests that secondary $K$-theory can be viewed as an algebraic (and integral) analogue of the latter.

\subsection{Categorified Hochschild homology and the categorified Chern character}
\label{sub:chc}
In \cite{BFN}, Ben-Zvi, Francis and Nadler investigate categorified instances of Hochschild homology and Hochschild cohomology of  
commutative algebra objects in symmetric monoidal 
$\infty$-categories. 
Their work suggests a categorification of the Dennis trace map in which the role of topological Hochschild homology is played by the $\infty$-category of quasi-coherent sheaves on the free loop space.

Let $\cC$ be a symmetric monoidal 
$\infty$-category and let $\mathrm{Alg}(\cC)$ be the $\infty$-category of 
associative algebra objects in $\cC$.  Every  associative algebra $A \in \mathrm{Alg}(\cC)$ has a canonical structure of $A \otimes A^\op$-module.  
Following \cite{BFN}, we define 
the \emph{Hochschild homology} of $A$ by the formula
\[
\mathrm{HH}(A) = A \otimes_{A \otimes A^{\op}} A \in \cC.   
\]
If $\cC = \mathrm{\cP r^L}$ 
is the $\infty$-category of presentable $\infty$-categories and $A= \mathrm{QCoh}(X)$ is the $\infty$-category of quasi-coherent sheaves on a scheme $X$ (or more generally on a \emph{perfect stack}),  
their theory categorifies the picture of Hochschild (co)homology of Calabi--Yau categories emerging from two-dimensional TQFT. In particular, a key insight of \cite{BFN} is that there is an equivalence of presentable stable $\infty$-categories  
\[
\mathrm{HH}( \mathrm{QCoh}( X)) 
\simeq \mathrm{QCoh}(\cL X)
\] 
between the Hochschild homology of 
$\mathrm{QCoh}(X)$ and quasi-coherent sheaves on the free loop stack $\cL X$. 

In ordinary algebra, one of the salient properties of Hochschild homology is that it is the recipient of a 
trace map that takes perfect modules to Hochschild classes.  Namely, let $\Perf(A)$ be the $\infty$-category of perfect modules over an associative algebra $A$ (in spectra, say) and denote by $\iota_0 \Perf(A)$ its underlying space of objects. Then the Hochschild homology of $A$ is isomorphic to the Hochschild homology of $\Perf(A)$, and this  gives rise to a trace map 
\begin{equation}
\label{eqn:tr}
\iota_0\Perf(A) \to \mathrm{HH}(\Perf(A)) \simeq \mathrm{HH}(A).
\end{equation}
Note that in \eqref{eqn:tr}, 
contrary to the previous paragraph, the notation 
$\mathrm{HH}(\Perf(A))$ stands for the  ordinary Hochschild homology of small 
stable $\infty$-categories, which takes values in spectra.  
This trace map factors through the nonconnective $K$-theory of $A$ and lands in the homotopy fixed points of the $S^1$-action on $\mathrm{HH}(A)$, giving rise to the classical Chern character with value in the negative cyclic homology of $A$ (see for example \cite{Keller}):
\begin{tikzmath}
	\diagram{
	\iota_0\Perf(A) & \HH(A) \\
	\K(A) & \HH(A)^{hS^1}\rlap. \\
	};
	\arrows (11-) edge node[above]{\eqref{eqn:tr}} (-12) (11) edge (21) (21-) edge node[above]{$\ch$} (-22) (22) edge (12);
\end{tikzmath}

If $\mathrm{Cat}(\Perf(X))$ denotes  
the $(\infty,2)$-category of small $\Perf(X)$-linear stable $\infty$-categories,  
a categorified analogue of \eqref{eqn:tr}
would be an $(\infty,1)$-functor 
\begin{equation}
\label{eqn:cattr}
\iota_1\mathrm{Cat}(\Perf(X)) \to \mathrm{HH}(\mathrm{QCoh}(X)) \simeq \mathrm{QCoh}(\cL X),
\end{equation} 
where $\iota_1\mathrm{Cat}(\Perf(X))$ denotes the maximal sub-$(\infty,1)$-category of $\mathrm{Cat}(\Perf(X))$.
The formalism of Chern characters developed by Toën and Vezzosi in \cite{TV2} gives a partial construction of such a functor. Namely, they construct a morphism of $\infty$-groupoids $\iota_0\mathrm{Cat}(\Perf(X)) \to \iota_0\QCoh(\cL X)$, and they show that it factors through the homotopy $S^1$-fixed points $\iota_0\QCoh^{S^1}(\cL X)$. When $X$ is the classifying stack of an algebraic group, their construction is an enhancement of Ganter and Kapranov's 2-character theory \cite{Ganter20082268}.

One of our main results is that Toën and Vezzosi's categorified Chern character is the shadow of a much richer categorified character theory which is captured by an exact symmetric monoidal $(\infty,1)$-functor
\[
\ch\colon \MOT(\Perf(X)) \rightarrow \mathrm{QCoh}^{S^1}(\cL X)
\] 
between the stable $\infty$-category of \emph{localizing $\Perf(X)$-motives}  (see \S\ref{intro:nc} below), and  
$S^1$-equivariant  
quasi-coherent sheaves on $\cL X$. More precisely:

\begin{theorem}[see Corollary~\ref{cor:affine}]
 The Toën--Vezzosi categorified Chern character can be promoted to a symmetric monoidal $(\infty,1)$-functor \eqref{eqn:cattr}, which moreover fits in a commutative square
\begin{tikzmath}
	\diagram{
	\iota_1\Cat(\Perf(X)) & \QCoh(\cL X) \\
	\MOT(\Perf(X)) & \QCoh^{S^1}(\cL X)\rlap. \\
	};
	\arrows (11-) edge node[above]{\eqref{eqn:cattr}} (-12) (11) edge (21) (21-) edge node[above]{$\ch$} (-22) (22) edge (12);
\end{tikzmath}	
\end{theorem}

This construction categorifies three key features of the ordinary trace map \eqref{eqn:tr}, namely its multiplicativity, its $S^1$-invariance, and the fact that it splits exact sequences.
In particular, this clarifies that the stable $(\infty, 1)$-category of noncommutative motives $\MOT(\Perf(X))$ can be viewed as 
the ``nonconnective $K$-theory'' of the $(\infty, 2)$-category 
$\Cat(\Perf(X))$. 

\begin{remark}
	If $k$ is a field of characteristic zero and $X$ is a smooth $k$-scheme, the categorified Chern character $\ch\colon\MOT(\Perf(X))\to\QCoh^{S^1}(\cL X)$ is related to the Hodge realization of ordinary motives, see \cite[\S4.3]{TV2}. More precisely, there is a commutative diagram of stable presentable $\infty$-categories
	\begin{tikzmath}
		\diagram{
		& \MOT(\Perf(X)) & \QCoh^{S^1}(\cL X) \\
		\mathrm{SH}(X)^\vee & \MOT_{\mathbb A^1}(\Perf(X)) & \Mod_{\Z/2}(\scr D_X)\rlap, \\
		};
		\arrows (12-) edge node[above]{$\ch$} (-13) (12) edge (22) (13) edge (23) (21-) edge (-22) (22-) edge[dashed] (-23);
	\end{tikzmath}
	where $\mathrm{SH}(X)^\vee$ is the dual of Voevodsky's stable $\mathbb A^1$-homotopy category over $X$ \cite{VoevodskyA1}, $\MOT_{\mathbb A^1}(\Perf(X))\subset \MOT(\Perf(X))$ is the reflective subcategory of $\A^1$-local motives \cite{TabuadaA1}, and $\Mod_{\Z/2}(\scr D_X)$ is the $\infty$-category of $\mathbb Z/2$-graded $\scr D$-modules over $X$. The lower composition sends a smooth $X$-scheme to the corresponding variation of Hodge structure over $X$. 
	See also work of Robalo \cite{Robalo} for a different but closely related formalism of noncommutative motives, and a comparison with Voevodsky's stable $\mathbb A^1$-homotopy category.
\end{remark}

\subsection{The structure of the paper}
We explain next our main results and the structure of the paper.

\subsubsection{Higher traces}
It is well known that the topological Hochschild homology $\HH(\scr A,\scr M)$ of a small stable $\infty$-category $\scr A$ with coefficients in an $\scr A$-bimodule $\scr M$ can be identified with the \emph{trace} of $\scr M$ in a certain symmetric monoidal $(\infty,2)$-category, whose objects are stable $\infty$-categories and whose morphisms are bimodules (we will recall this identification in \S\ref{sub:hochschild}).

Consider the following features of topological Hochschild homology:
\begin{enumerate}
	\item It is functorial in the pair $(\scr A,\scr M)$ as follows: given an exact functor $\scr A\to \scr B$ and a morphism of $\scr A$-bimodules $\scr M\to\scr N$, where $\scr N$ is an $\scr B$-bimodule, there is an induced morphism $\HH(\scr A,\scr M)\to \HH(\scr B,\scr N)$.
	\item When $\scr M=\scr A$, $\HH(\scr A)=\HH(\scr A,\scr A)$ has a canonical action of the circle group $S^1$, which is moreover natural in $\scr A$.
	\item It is a \emph{localizing invariant} of stable categories: given a fully faithful inclusion $\scr A\into\scr B$ with Verdier quotient $\scr B/\scr A$, there is a cofiber sequence of spectra
	\[
	\HH(\scr A) \to \HH(\scr B) \to \HH(\scr B/\scr A).
	\]
\end{enumerate}
Our goal in Sections \ref{sec:traces} and \ref{sec:localization} is to show that (1)--(3) are general features of traces in symmetric monoidal $(\infty,n)$-categories. Given a symmetric monoidal $(\infty,n)$-category $\scr C$, endomorphisms of the unit object form a symmetric monoidal $(\infty,n-1)$-category $\Omega\scr C$. We define in \S\ref{sub:trace} a symmetric monoidal $(\infty,n-1)$-category $\End(\scr C)$, whose objects are pairs $(X,f)$ where $X\in\scr C$ is dualizable and $f$ is an endomorphism of $X$. A $1$-morphism $(X,f)\to (Y,g)$ in $\End(\scr C)$ is a pair $(\phi,\alpha)$ where $\phi\colon X\to Y$ is a right dualizable $1$-morphism and $\alpha\colon \phi f\to g\phi$ is a $2$-morphism.
The main construction of \S\ref{sub:trace} can be summarized as follows:

\begin{theorem}\label{thm:intro-trace}
	Let $\scr C$ be a symmetric monoidal $(\infty,n)$-category. Then the assignment $(X,f)\mapsto \Tr(f)$ can be promoted to a symmetric monoidal $(\infty,n-1)$-functor
	\[
	\Tr\colon \End(\scr C) \to \Omega \scr C,
	\]
	natural in $\scr C$.
\end{theorem}

This theorem generalizes the functoriality of Hochschild homology described in (1).
In \S\ref{sub:S1trace}, we consider a subcategory $\Aut(\scr C)\subset \End(\scr C)$, whose objects are pairs $(X,f)$ where $X$ is a dualizable object and $f$ is an \emph{automorphism} of $X$. There is a ``tautological'' action of the circle group $S^1$ on $\Aut(\scr C)$, and we show that the trace functor $\Tr$ is $S^1$-invariant:

\begin{theorem}[see Theorem \ref{thm:S1trace}]
	\label{thm:intro-S1trace}
	Let $\scr C$ be a symmetric monoidal $(\infty,n)$-category. Then the symmetric monoidal $(\infty,n-1)$-functor
	\[
	\Tr\colon \Aut(\scr C) \to \Omega\scr C
	\]
	admits a canonical $S^1$-invariant refinement which is natural in $\scr C$.
\end{theorem}

In particular, since $(X,\id_X)$ is an $S^1$-fixed point in $\Aut(\scr C)$, $\Tr(\id_X)$ has a canonical action of $S^1$ which is natural in $(X,\id_X)$. Specializing to the $(\infty,2)$-category of stable $\infty$-categories and bimodules, this theorem recovers property (2) above.
If $n=1$, then $\Aut(\scr C)$ and $\Omega\scr C$ are $\infty$-groupoids, and Theorem \ref{thm:intro-S1trace} was proved by Toën and Vezzosi in \cite[\S2.3]{TV2}. The generalization to $n\geq 2$ is not a mere formality, however, as we will explain in \S\ref{sub:tracesIntro}. Our proof crucially relies on the formalism of higher lax transfors developed by Johnson-Freyd and Scheimbauer in \cite{JFS}.

A symmetric monoidal $(\infty,2)$-category is called \emph{linearly symmetric monoidal} if its mapping $\infty$-categories are stable and if composition and tensor products of $1$-morphisms are exact in each variable.
The notion of Verdier localization sequence makes sense in any such $(\infty,2)$-category, see Definition~\ref{dfn:exact}.
We then have the following generalization of (3):

\begin{theorem}[see Theorem \ref{thm:localization}]
	\label{thm:intro-localization}
	Let $\scr C$ be a linearly symmetric monoidal $(\infty,2)$-category.
	Let $X\to Y \to Z$ be a localization sequence of dualizable objects in $\scr C$, and let
	\begin{tikzmath}
		\diagram{X & Y & Z \\ X & Y & Z \\};
		\arrows (11-) edge (-12) (12-) edge (-13)
		(21-) edge (-22) (22-) edge (-23)
		(11) edge[font=\normalsize,draw=none] node[rotate=45]{$\Longrightarrow$} (22)
		(11) edge[draw=none] node[pos=.3]{$\alpha$} (22)
		(12) edge[font=\normalsize,draw=none] node[rotate=45]{$\Longrightarrow$} (23)
		(12) edge[draw=none] node[pos=.3]{$\beta$} (23)
		(11) edge node[left]{$f$} (21) (12) edge node[left]{$g$} (22) (13) edge node[right]{$h$} (23);
	\end{tikzmath}
	be a commutative diagram where $\alpha$ and $\beta$ are right adjointable. Then $\Tr(f)\to \Tr(g)\to\Tr(h)$ is a cofiber sequence in $\Omega\scr C$.
\end{theorem}

We can regard Theorem~\ref{thm:intro-localization} as a categorification of May's theorem on the additivity of traces in symmetric monoidal stable $\infty$-categories \cite{May}. An interesting question is what form this additivity theorem should take for traces in symmetric monoidal $(\infty,n)$-categories, for $n\geq 3$.

\subsubsection{Noncommutative 
$\cE$-motives}
\label{intro:nc} 
A theory of noncommutative 
motives was first sketched by Kontsevich in the mid 2000's \cite{Kontsevich1}, in analogy with the 
theory of pure Chow motives of algebraic varieties.   
The objects of Kontsevich's category of noncommutative motives are smooth and proper
 triangulated dg-categories, and the mapping spaces are given by the $K$-theory of bimodules.  
 Tabuada's work \cite{tabuada2008} 
 shifts the perspective by foregrounding   
 a universal property of noncommutative motives that is reminiscent of the universality of ordinary motives within Weil cohomology theories. Roughly speaking, Tabuada defines the category of noncommutative motives as the recipient of the universal invariant of dg-categories that satisfies Waldhausen additivity. In Tabuada's approach, the fact that the mapping spaces are given by bivariant
 $K$-theory is a theorem,  rather than being part of the definition. 
As a corollary, Kontsevich's noncommutative motives sit inside Tabuada's as a full subcategory.

We will rely on the theory of noncommutative motives of stable $\infty$-categories that was  developed in  \cite{BGT}. After giving a short  recapitulation of the theory of tensored 
$\infty$-categories in section \ref{prel}, we devote section 
\ref{E-motives} to extend to the enriched setting the results of \cite{BGT}.  Let $\cE$ be a small symmetric monoidal $\infty$-category that is rigid and stable, and let  
$\ccat{perf}{\cE}$ be the $\infty$-category of small, stable, and idempotent complete $\infty$-categories that are tensored over $\cE$. As evinced in \cite{CT} and \cite{BGT}, there are two meaningful classes of invariants of stable $\cE$-linear $\infty$-categories (called \emph{additive} and \emph{localizing}, see Definition \ref{def:add} and \ref{def:loc}), and they pick out 
two different notions of 
noncommutative motives. The universal additive and localizing invariants of 
$\cE$-linear categories are canonical functors 
\[
\cU_{\add}\colon \ccat{perf}{\cE} \to \mathrm{Mot(\cE)}, 
\quad \cU_{\loc}\colon \ccat{perf}{\cE} \to 
 \mathrm{\mathbb{M}ot(\cE)}, 
\]
that map respectively to the category of \emph{additive}  and of \emph{localizing} $\cE$-motives. 
In section \ref{sub:add} and \ref{sub:loc} 
we construct the categories 
of additive and localizing $\cE$-motives 
$\mathrm{Mot(\cE)}$ and 
$\mathrm{\mathbb{M}ot(\cE)}$. 
We show that the $\infty$-categories 
$\mathrm{Mot(\cE)}$ and 
$\mathrm{\mathbb{M}ot(\cE)}$
 are stable and presentable, 
and $\cU_{\add}$ and 
$\cU_{\loc}$ are symmetric monoidal functors. 
In its main lines our treatment follows 
\cite{CT} and \cite{BGT}, and encompasses the theory of \cite{BGT} as the special case when  
$\cE$ is the $\infty$-category of finite spectra.  

In Section \ref{sec:corepresentability}, we prove that the $K$-theory of 
$\cE$-linear $\infty$-categories is corepresentable in noncommutative $\cE$-motives:

\begin{theorem}[see Theorems \ref{thm:representability} and \ref{thm:nc-representability}]
\label{thrm:corep}
Let $\cA$ be an $\infty$-category in $\mathrm{Cat^{perf}(\cE)}$. Then there are natural equivalences
\[
\mathrm{Mot(\cE)}(\cU_{\add}(\cE), \cU_{\add}(\cA)) \simeq K(\cA)
\quad\text{and}\quad \mathrm{\mathrm{\mathbb{M}}ot(\cE)}(\cU_{\loc}(\cE), \cU_{\loc}(\cA)) \simeq \mathrm{\mathbb{K}}(\cA),
\]
where $K(\cA)$ and $\mathrm{\mathbb{K}}(\cA)$ are respectively the connective and 
nonconnective $K$-theory spectra of $\cA$.  
\end{theorem} 

By the functoriality properties of 
$\mathrm{Mot(\cE)}$ and 
$\mathrm{\mathrm{\mathbb{M}}ot(\cE)}$,  
the proof of Theorem \ref{thrm:corep} reduces to the absolute case which was studied in \cite{BGT}. 
In the case of connective $K$-theory 
a substantially stronger corepresentability 
result is available, see Theorem \ref{thm:representability}.  
Functoriality arguments however are insufficient to establish Theorem \ref{thm:representability}, 
and the proof consists instead in an adaptation of 
the arguments of \cite{tabuada2008} and 
\cite{BGT} to the enriched setting.

\subsubsection{The categorified Chern character}

Let $X$ be a derived stack over a base commutative ring $k$. In \cite{TV1,TV2}, Toën and Vezzosi consider a generalization of the classical Chern character, which assigns to every perfect complex on $X$ a rotation-invariant function on the free loop stack $\scr LX$. More precisely, they construct an additive and multiplicative map
\begin{equation}\label{eqn:TVchern1}
\ch^\pre\colon \iota_0\Perf(X) \to \scr O(\scr LX)^{hS^1},
\end{equation}
where $\iota_0\Perf(X)$ is the maximal sub-$\infty$-groupoid of the $(\infty,1)$-category $\Perf(X)$.
They also introduce a categorified version of this construction, which is an additive and multiplicative map
\begin{equation}\label{eqn:TVchern2}
\ch^\pre\colon \iota_0\Cat^\perf(X) \to \iota_0\QCoh^{S^1}(\scr LX),
\end{equation}
where $\Cat^\perf(X)$ is the $(\infty,2)$-category of sheaves of dg-categories on $X$. Finally, they combine \eqref{eqn:TVchern1} and \eqref{eqn:TVchern2} to obtain the secondary Chern pre-character
\begin{equation}\label{eqn:TVchern3}
\ch^{\pre,(2)}\colon \iota_0\Cat^\sat(X) \to \scr O(\scr L^2X)^{h(S^1\times S^1)},
\end{equation}
where $\Cat^\sat(X)$ is the $(\infty,2)$-category of sheaves of saturated dg-categories on $X$.

In Section \ref{sec:chern}, we use the main results of the previous sections to refine each of these constructions as follows. First of all, we allow $k$ to be a (potentially nonconnective) $E_\infty$-ring spectrum and $X$ to be a spectral prestack over $k$.
\begin{enumerate}
	\item In Theorem~\ref{thm:additivity}, we prove that~\eqref{eqn:TVchern1} factors through the nonconnective deloopings of $K$-theory and induces a morphism of $E_\infty$-ring spectra
	\[
	\ch\colon \K(X) \to \scr O(\scr LX)^{hS^1}.
	\]
	\item In Theorem~\ref{thm:main}, we prove that~\eqref{eqn:TVchern2} lifts to a symmetric monoidal $(\infty,1)$-functor
	\[
	\ch^\pre\colon \iota_1\Cat^\perf(X) \to \QCoh^{S^1}(\scr LX),
	\]
	which sends localization sequences to cofiber sequences. As a consequence, we obtain the morphism of $E_\infty$-ring spectra $\K^{(2)}(X)\to \K^{S^1}(\scr LX)$ envisioned in \cite[Section 4]{TV1}.
	\item Finally, in \S\ref{sub:K2},
	we combine these results and deduce that~\eqref{eqn:TVchern3} descends to a morphism of $E_\infty$-ring spectra
	\[
	\ch^{(2)}\colon \K^{(2)}(X) \to \scr O(\scr L^2X)^{h(S^1\times S^1)}.
	\]
\end{enumerate}

\subsection{Acknowledgments} 
We are grateful to Clark Barwick, Kai Behrend, David Ben-Zvi, Andrew Blumberg, David Carchedi, Rune Haugseng, Claudia Scheimbauer, Chris Schommer-Pries, Jay Shah,
Gon\c{c}alo Tabuada, Bertrand To\"en, and Jesse Wolfson for useful conversations and for their interest in this project. We thank the referee for very helpful comments that have greatly improved the readability of the paper.  
 
\subsection{Terminology and notation}
We use the terms $\infty$-category and $(\infty,1)$-category interchangeably.
We denote by $\cS$ and $\Sp$ the $\infty$-categories of spaces (i.e., $\infty$-groupoids) and of spectra, respectively.
If $\scr A$ is an $\infty$-category and $a,b\in\scr A$, we write $\scr A(a,b)\in\cS$ for the space of maps from $a$ to $b$ in $\scr A$.
If $\scr A$ is stable, we will also write $\scr A(a,b)$ for the spectrum which is the canonical infinite delooping of that space.
If $\cA$ admits filtered colimits, recall that an object $a\in\cA$ is \emph{compact} if $\cA(a,\ph)\colon\cA^\op\to\cS$ preserves filtered colimits; we denote by $\scr A^\omega\subset\scr A$ the full subcategory of compact objects.

\section{Traces in symmetric monoidal \texorpdfstring{$(\infty,n)$}{(∞,n)}-categories}
\label{sec:traces}

\subsection{Introduction}
\label{sub:tracesIntro}

Let $(\scr C,\tens,\mathbf{1})$ be a symmetric monoidal category. Recall that an object $X\in\scr C$ is \emph{dualizable} if there exists an object $X^\vee$ and morphisms $\coev_X\colon \mathbf{1}\to X\tens X^\vee$ and $\ev_X\colon X^\vee\tens X\to \mathbf{1}$ satisfying the \emph{triangle identities}: the composites
\begin{gather*}
	X  \xrightarrow{\coev_X \tens \id}  X \tens X^\vee \tens X \xrightarrow{ \id \tens \ev_X}  X, \\
	X^\vee  \xrightarrow{  \id \tens \coev_X}   X^\vee \tens X  \tens X^\vee \xrightarrow{  \ev_X \tens \id}  X^\vee
\end{gather*}
are identity morphisms.

If $X$ is dualizable and $f\colon X\to X$ is any endomorphism, the \emph{trace} $\Tr(f)$ of $f$ is the endomorphism of $\mathbf{1}$ given by the composition
\[
\mathbf{1} \xrightarrow{\coev_X} X\tens X^\vee \xrightarrow{f\tens\id} X\tens X^\vee \simeq X^\vee\tens X \xrightarrow{\ev_X} \mathbf{1}.
\]
If $\phi\colon X\to Y$ is an isomorphism in $\scr C$, then $\Tr(f)=\Tr(\phi\circ f\circ \phi^{-1})$. If we write $\End({\scr C})$ for the groupoid of endomorphisms of dualizable objects in $\scr C$ and $\Omega\scr C$ for the set of endomorphisms of $\mathbf{1}$, the trace is thus a functor
\[
\Tr\colon \End({\scr C}) \to \Omega\scr C.
\]

Things become more interesting if $\scr C$ is a symmetric monoidal $2$-category.\footnote{In this paper, we use the term ``$2$-category'' for what is often called a weak $2$-category or a bicategory: the composition of $1$-morphisms is only required to be associative up to a (specified) $2$-isomorphism.} Then $\Omega\scr C$ is a category, and one can ask to what extent the trace can be upgraded to a functor with values in $\Omega\scr C$. Recall that a $1$-morphism $\phi\colon X\to Y$ in a $2$-category $\scr C$ is \emph{right dualizable} if there exists a $1$-morphism $\phi^r\colon Y\to X$ and $2$-morphisms $\eta\colon \id_X\to \phi^r\circ\phi$ and $\epsilon\colon \phi\circ \phi^r\to \id_Y$ satisfying the following triangle identities: the composites
\begin{gather*}
	\phi  \xrightarrow{\id\circ\eta}  \phi\circ\phi^r\circ\phi \xrightarrow{ \epsilon\circ\id} \phi, \\
	\phi^r \xrightarrow{\eta\circ \id}   \phi^r\circ\phi\circ\phi^r \xrightarrow{ \id\circ\epsilon}  \phi^r
\end{gather*}
are identity $2$-morphisms.
We then say that $\phi^r$ is \emph{right adjoint} to $\phi$. Note that this agrees with the usual notion of adjunction when $\scr C$ is the $2$-category of categories.

Define a $(2,1)$-category $\End(\scr C)$ as follows:
\begin{itemize}
	\item An object of $\End(\scr C)$ is a pair $(X,f)$ where $X$ is a dualizable object in $\scr C$ and $f$ is an endomorphism of $X$.
	\item A $1$-morphism $(X,f)\to(Y,g)$ in $\End(\scr C)$ is a pair $(\phi,\alpha)$ where $\phi\colon X\to Y$ is a \emph{right dualizable} $1$-morphism and $\alpha\colon \phi f\to g\phi$ is a $2$-morphism:
	\begin{tikzmath}
		\diagram{X & Y \\ X & Y\rlap. \\   };
		\arrows (11-) edge node[above]{$\phi$} (-12) (11) edge node[left]{$f$} (21) (12) edge node[right]{$g$} (22) (21-) edge node[below]{$\phi$} (-22)
		(11) edge[font=\normalsize,draw=none] node[rotate=45]{$\Longrightarrow$} (22)
		(11) edge[draw=none] node[pos=.3]{$\alpha$} (22);
	\end{tikzmath}
	\item A $2$-morphism $(\phi,\alpha)\to (\psi,\beta)$ in $\End(\scr C)$ is a $2$-isomorphism $\xi\colon \phi\stackrel\sim\to\psi$ such that $(g\xi)\alpha= \beta(\xi f)$.
\end{itemize}
It is easy to show that the trace can be upgraded to a functor $\Tr\colon \End(\scr C)\to\Omega\scr C$: the image of a $1$-morphism $(\phi,\alpha)\colon (X,f)\to (Y,g)$ is depicted by the diagram
\begin{tikzequation}\label{eqn:2trace}
	\def\rowsep{1em}
	\diagram{
	& X\tens X^\vee & & & X\tens X^\vee & \\
	\mathbf{1} & & & & & \mathbf{1} \rlap, \\
	& Y\tens Y^\vee & & & Y\tens Y^\vee & \\
	};
	\arrows (21) edge (12) edge (32) (12-) edge node[above]{$f\tens \id$} (-15) (32-) edge node[below]{$g\tens \id$} (-35) (15) edge (26) (35) edge (26) (12) edge node[right]{$\phi\tens\phi^{r\vee}$} (32) (15) edge node[left]{$\phi\tens\phi^{r\vee}$} (35)
	(21) edge[font=\normalsize,draw=none] node[rotate=45]{$\Longleftarrow$} node[pos=.13,rotate=45]{$\Longleftarrow$} node[pos=.87,rotate=45]{$\Longleftarrow$} (26)
	(21) edge[draw=none] node[above left]{$\alpha\tens \id$} (26)
	;
\end{tikzequation}
where the unlabeled $2$-morphisms are
\begin{gather*}
	(\phi\tens\phi^{r\vee})\coev_X = (\phi\phi^r\tens\id)\coev_Y \stackrel\epsilon\to \coev_Y,\\
	\ev_X \stackrel\eta\to \ev_X(\phi^r\phi\tens\id) = \ev_Y(\phi\tens\phi^{r\vee}).
\end{gather*}
Dually, the trace also has a functoriality with respect to \emph{left dualizable} $1$-morphisms. However, this is a special case of the above functoriality, applied to the $2$-category obtained from $\scr C$ by reversing the direction of the $2$-morphisms. As the symmetric monoidal $2$-categories that occur in practice seem to favor right dualizability, we only treat the right dualizable case explicitly.

Our goal in the section is to generalize this functoriality of the trace to the case where $\scr C$ is a symmetric monoidal $(\infty,n)$-category. In that case, $\Omega\scr C$ is an $(\infty,n-1)$-category, and we will define an $(\infty,n-1)$-category $\End(\scr C)$ and a functor $\Tr\colon \End(\scr C)\to\Omega\scr C$ with the expected values on objects and $1$-morphisms. As an example, we will see in \S\ref{sub:hochschild} that, if $\scr C$ is the symmetric monoidal $(\infty,2)$-category of compactly generated stable $\infty$-categories, then $\End(\scr C)$ is the $(\infty,1)$-category of pairs $(\scr A,\scr M)$ where $\scr A$ is a small idempotent complete stable $\infty$-category and $\scr M$ is an $\scr A$-bimodule, and $\Tr\colon\End(\scr C)\to\Sp$ sends a bimodule to its topological Hochschild homology. 

Going back to the $1$-categorical situation, one may observe that the functor $\scr C\mapsto \End({\scr C})$ from symmetric monoidal categories to groupoids is \emph{corepresentable}. That is, there exists a symmetric monoidal category $\scr E$ and an equivalence of categories
\[
\Fun^\tens(\scr E,\scr C)\simeq \End({\scr C}),
\]
natural in $\scr C$. The category $\scr E$ is the free rigid symmetric monoidal category on $B\bb N$, the category with one object and morphism set the monoid $\bb N$. The above equivalence sends a symmetric monoidal functor $f\colon \scr E\to\scr C$ to the image by $f$ of the ``walking endomorphism'' $1\in \N$.

By the Yoneda lemma, the trace functor $\Tr\colon \End({\scr C})\to \Omega\scr C$, being natural in the symmetric monoidal category $\scr C$, is completely determined by the trace of the walking endomorphism, which is an element of the set $\Omega\scr E$. In \cite{TV2}, Toën and Vezzosi used this observation to define a functorial enhancement of the trace on a symmetric monoidal $(\infty,1)$-category $\scr C$.
If $\Cat_{(\infty,1)}^\tens$ denotes the $(\infty,1)$-category of symmetric monoidal $(\infty,1)$-categories, there is an adjunction
\[
\Fr^\rig: \Cat_{(\infty,1)} \rightleftarrows \Cat_{(\infty,1)}^\tens:(\ph)^\rig,
\]
where $\scr C^\rig\subset\scr C$ is the full subcategory of dualizable objects. The existence of the left adjoint $\Fr^\rig$ follows from the adjoint functor theorem (a more explicit description of $\Fr^\rig$ is given by the $1$-dimensional cobordism hypothesis with singularities \cite[\S4.3]{Lurie}, but we do not need this description for the time being).
If $\End({\scr C})$ is the $\infty$-groupoid of endomorphisms of dualizable objects, it is then clear that the functor $\scr C\mapsto\End({\scr C})$ is corepresented by $\Fr^\rig (B\N)$. 
By the Yoneda lemma, the trace of the walking endomorphism in $\Omega \Fr^\rig(B\N)$ specifies a morphism of $\infty$-groupoids $\Tr\colon \End({\scr C})\to \Omega\scr C$, natural in $\scr C$.
Because symmetric monoidal functors commute with traces, this morphism sends an object $(X,f)\in\End(\scr C)$ to the trace of $f$, so it is indeed a functorial enhancement of the trace. 

One may hope to use a similar corepresentability trick to define the trace on a symmetric monoidal $(\infty,2)$-category. This hope is quickly squashed by the observation that \emph{any} corepresentable functor on the $(\infty,2)$-category of rigid symmetric monoidal $(\infty,2)$-categories takes values in $\infty$-groupoids: a standard argument shows that symmetric monoidal natural transformations between symmetric monoidal functors with rigid domain are always invertible. 
To work around this problem, one is therefore led to consider \emph{lax} natural transformations. The formalism of lax natural transformations was developed by Johnson-Freyd and Scheimbauer in \cite{JFS}. We will see that a natural definition of $\End(\scr C)$, for $\scr C$ a symmetric monoidal $(\infty,n)$-category, is the $(\infty,n-1)$-category of symmetric monoidal \emph{oplax transfors} $\Fr^\rig(B\N)\to\scr C$, where $\Fr^\rig (B\N)$ is the \emph{same} category that was used in the case $n=1$. One must be careful that categories of oplax transfors are not the mapping objects of any higher category, and so there is no sense in which the functor $\scr C\mapsto \End(\scr C)$ is corepresentable. Instead of invoking the Yoneda lemma, we will therefore argue ``by hand'' that an element in $\Omega \Fr^\rig(B\N)$ gives rise to a natural morphism of $(\infty,n-1)$-categories $\End(\scr C)\to\Omega\scr C$.

\subsection{The trace as a symmetric monoidal \texorpdfstring{$(\infty,n-1)$}{(∞,n-1)}-functor}
\label{sub:trace}

We borrow terminology and notation from \cite{JFS}. In particular, $(\infty,n)$-categories are modeled by Barwick's complete $n$-fold Segal spaces, where in this context ``space'' means Kan complex. A complete $n$-fold Segal space $\scr C$ is in particular a collection of spaces $\scr C_{\vec k}$ indexed by $\vec k\in (\Delta^\op)^n$. For $p\leq n$, we denote by $(p)$ the $n$-tuple $(1,\dotsc,1,0,\dotsc,0)$, with $p$ copies of $1$. Then $\scr C_{(p)}$ is the \emph{space of $p$-morphisms} in $\scr C$. We denote by $\maps^h$ and $\times^h$ strictly functorial models for derived mapping spaces and homotopy fiber products of presheaves of spaces (i.e., computed with respect to objectwise weak equivalences). If $\scr C$ is a complete $n$-fold Segal space and $m\leq n$, we denote by $\iota_m\scr C$ the complete $m$-fold Segal space defined by $(\iota_m\scr C)_{\vec k}=\scr C_{\vec k,\vec 0}$. Thus, $\iota_m\scr C$ models the maximal sub-$(\infty,m)$-category of $\scr C$. Note that $\iota_0\scr C=\scr C_{(0)}$ is the space of objects of $\scr C$. We will also use the computads $\Theta^{\vec k}$ and $\Theta^{\vec k;\vec l}$, for tuples of natural numbers $\vec k$ and $\vec l$, as defined in \cite{JFS}. The former are designed so that $\scr C_{\vec k}\simeq\maps^h(\Theta^{\vec k},\scr C)$. Symmetric monoidal $(\infty,n)$-categories are modeled by strict functors from the category of pointed finite sets to the category of complete $n$-fold Segal spaces satisfying the usual Segal condition; we commit the usual sacrilege of identifying such a functor with its value on $[1]$ (the set $\{0,1\}$ pointed at $0$).

Let $\scr C$ be an $(\infty,n)$-category. For $\vec k$ a $n$-tuple of natural numbers, let $\Lax_{\vec k}(\scr C)$ be the $n$-fold simplicial space defined by
\[
\Lax_{\vec k}(\scr C)_{\vec\bullet} = \maps^h(\Theta^{\vec k;\vec\bullet},\scr C).
\]
By \cite[Theorem 5.11]{JFS}, $\Lax_{\vec k}(\scr C)$ is a complete $n$-fold Segal space, and moreover $\vec k\mapsto\Lax_{\vec k}(\scr C)$ is a complete $n$-fold Segal object internal to complete $n$-fold Segal spaces. 
We have $\Lax_{(0)}(\scr C)\simeq \scr C$ and, in the notation of \cite{JFS}, $\Lax_{(1)}(\scr C)\simeq\scr C^\to$. We also have $\Lax_{\vec k}(\scr C)_{(0)}\simeq \scr C_{\vec k}$, since $\Theta^{\vec k;(0)}=\Theta^{\vec k}$. In particular, the space of objects of $\Lax_{(p)}(\scr C)$ is the space of $p$-morphisms in $\scr C$.
If $\scr C$ is symmetric monoidal, so is $\Lax_{\vec k}(\scr C)$ with $\Lax_{\vec k}(\scr C)[m]=\Lax_{\vec k}(\scr C[m])$.

If $\scr B$ and $\scr C$ are $(\infty,n)$-categories, the complete $n$-fold Segal space $\Fun^\oplax(\scr B,\scr C)$ of \emph{oplax transfors} is defined by:\footnote{The fact that these are oplax rather than lax is explained by the asymmetry of the Gray tensor product: if $\scr A$, $\scr B$, and $\scr C$ are $(\infty,n)$-categories, then $\maps^h(\scr A,\Fun^\oplax(\scr B,\scr C))\simeq \maps^h(\scr B,\Fun^\lax(\scr A,\scr C))$.
The definition of $\Fun^\oplax$ should be understood as the special case of this formula with $\scr A=\Theta^{\vec\bullet}$.
}
\[
\Fun^\oplax(\scr B,\scr C)_{\vec \bullet} = \maps^h(\scr B,\Lax_{\vec \bullet}(\scr C)).
\]
If $\scr B$ and $\scr C$ are moreover symmetric monoidal, the complete $n$-fold Segal space $\Fun^\oplax_\tens(\scr B,\scr C)$ of \emph{symmetric monoidal oplax transfors} is similarly defined by:
\[
\Fun^\oplax_\tens(\scr B,\scr C)_{\vec\bullet} = \maps^{h}_\tens(\scr B,\Lax_{\vec\bullet}(\scr C)).
\]

Recall that an object in a symmetric monoidal $(\infty,n)$-category $\scr C$ is called \emph{dualizable} if it is dualizable in the homotopy $1$-category $\h_1\iota_1\scr C$ (in the classical sense recalled in \S\ref{sub:tracesIntro}). Similarly, a $1$-morphism of $\scr C$ is called \emph{right dualizable} if it is so in the homotopy $2$-category $\h_2\iota_2\scr C$. We denote by $\scr C^\rig\subset\scr C$ the full subcategory spanned by the dualizable objects. The functor
\[
(\ph)^\rig\colon\Cat_{(\infty,n)}^\otimes \to \Cat_{(\infty,n)}
\]
preserves filtered colimits as well as limits \cite[Proposition 4.6.1.11]{HA}, and hence it admits a left adjoint
\[
\Fr^\rig\colon \Cat_{(\infty,n)} \to \Cat_{(\infty,n)}^\otimes,
\]
by the adjoint functor theorem \cite[Corollary 5.5.2.9]{HTT}. For $1\leq m\leq n$, $\Fr^\rig$ sends $(\infty,m)$-categories to $(\infty,m)$-categories, since $(\iota_m\scr C)^\rig=\iota_m(\scr C^\rig)$.

\begin{definition}
	If $\scr C$ is a symmetric monoidal $(\infty,n)$-category, the complete $(n-1)$-fold Segal space $\End(\scr C)$ is defined by:
	\[
	\End(\scr C)=\iota_{n-1}\Fun^\oplax_\tens(\Fr^\rig(B\N),\scr C).
	\]
\end{definition}

\begin{remark}
	One can show that the $(\infty,n)$-category $\Fun^\oplax_\tens(\Fr^\rig(B\N),\scr C)$ is in fact already an $(\infty,n-1)$-category: the rigidity of $\Fr^\rig(B\N)$ implies that the components of a symmetric monoidal oplax $k$-transfor $\Fr^\rig(B\N)\to\scr C$ satisfy a one-sided dualizability condition, which amounts to invertibility for $k=n$.
\end{remark}

If $p\leq n-1$, the space of $p$-morphisms in $\End(\scr C)$ is thus the space of symmetric monoidal functors $\Fr^\rig(B\N)\to\Lax_{(p)}(\scr C)$, i.e., the space of endomorphisms of dualizable objects in $\Lax_{(p)}(\scr C)$. Let us make this more explicit for $p=1$.
An object in $\Lax_{(1)}(\scr C)$ is a morphism $\phi\colon X\to Y$ in $\scr C$, and an endomorphism of such is a square
\begin{tikzmath}
	\diagram{X & Y \\ X & Y\rlap. \\ };
	\arrows (11-) edge node[above]{$\phi$} (-12) (11) edge node[left]{$f$} (21) (12) edge node[right]{$g$} (22) (21-) edge node[below]{$\phi$} (-22)
	(11) edge[font=\normalsize,draw=none] node[rotate=45]{$\Longrightarrow$} (22)
	(11) edge[draw=none] node[pos=.3]{$\alpha$} (22);
\end{tikzmath}
The following lemma shows that such a square is a $1$-morphism in $\End(\scr C)$, with source $(X,f)$ and target $(Y,g)$, if and only if $X$ and $Y$ are dualizable and $\phi$ is right dualizable:

\begin{lemma}\label{lem:dualizability}
	Let $\scr C$ be a symmetric monoidal $(\infty,n)$-category. An object $\phi\colon X\to Y$ in $\Lax_{(1)}(\scr C)$ is dualizable if and only if $X$ and $Y$ are dualizable in $\scr C$ and $\phi$ is a right dualizable $1$-morphism in $\scr C$. In that case, the trace of an endomorphism $(f,g,\alpha)$ of $\phi$ is the endomorphism of $\id_\mathbf{1}$ given by the diagram~\eqref{eqn:2trace}.
\end{lemma}

\begin{proof}
	Note that the assertion holds for $\scr C$ if and only if it holds for the homotopy $2$-category $\h_2\iota_2\scr C$. We may therefore assume that $\scr C$ is a $2$-category.
	Suppose that $X$ and $Y$ are dualizable and that $\phi\colon X\to Y$ is right dualizable, with right adjoint $\psi$ and unit and counit $\eta$ and $\epsilon$. We claim that $\psi^\vee\colon X^\vee\to Y^\vee$ is dual to $\phi$ in $\Lax_{(1)}(\scr C)$.
	Define $\coev_\phi\colon \id_\mathbf{1}\to \phi\tens\psi^\vee$ to be the following morphism in $\Lax_{(1)}(\scr C)$:
	\begin{tikzmath}
		\diagram{\mathbf{1} & \mathbf{1} \\ X\tens X^\vee & Y\tens Y^\vee\rlap, \\ };
		\arrows (11-) edge node[above]{$\id$} (-12) (11) edge node[left]{$\coev_X$} (21) (12) edge node[right]{$\coev_Y$} (22) (21-) edge node[below]{$\phi\tens\psi^\vee$} (-22)
		(11) edge[font=\normalsize,draw=none] node[rotate=45,pos=.6]{$\Longrightarrow$} (22);
	\end{tikzmath}
	where the $2$-morphism is
	\[
	(\phi\tens\psi^\vee)\coev_X = (\phi\psi\tens\id)\coev_Y\stackrel\epsilon\to\coev_Y.
	\]
	Similarly, let $\ev_\phi\colon \psi^\vee \tens \phi \to \id_\mathbf{1}$ be the morphism given by $\ev_X$, $\ev_Y$, and the $2$-morphism
	\[
	\ev_X \stackrel\eta\to \ev_X(\id\tens \psi\phi)=\ev_Y(\psi^\vee\tens \phi).
	\]
	It is easy to check that $\ev_\phi$ and $\coev_\phi$ determine a duality between $\phi$ and $\psi^\vee$ in $\Lax_{(1)}(\scr C)$.
	
	Conversely, suppose that $\phi\colon X\to Y$ admits a dual $\phi'\colon X'\to Y'$ in $\Lax_{(1)}(\scr C)$. The triangle identities imply at once that $X$ and $Y$ are dualizable in $\scr C$ with duals $X'$ and $Y'$. The evaluation $\ev_\phi$ thus consists of the $1$-morphisms $\ev_X$ and $\ev_Y$ and a $2$-morphism $\alpha\colon \ev_X\to \ev_Y(\phi'\tens\phi)$. Let $\eta\colon \id_X \to \phi^{\prime\vee}\phi$ be the $2$-morphism
	\[
	\id_X = (\id_X\tens\ev_X)(\coev_X\tens\id_X) \stackrel\alpha\to (\id_X \tens \ev_Y(\phi'\tens\phi))(\coev_X\tens\id_X) = \phi^{\prime\vee}\phi,
	\]
	and let $\epsilon\colon \phi\phi^{\prime\vee}\to \id_Y$ be defined in a dual manner. It is easy to check that $\eta$ and $\epsilon$ determine an adjunction between $\phi$ and $\phi^{\prime\vee}$ in $\scr C$.
\end{proof}

Let $\scr C$ be an $(\infty,n)$-category. Given $X,Y\in\scr C$, we have an $(\infty,n-1)$-category $\scr C(X,Y)$ defined by
\[
\scr C(X,Y)_{\vec\bullet} = \{X\}\times^h_{\scr C_{0,\vec\bullet}} \scr C_{1,\vec\bullet} \times^h_{\scr C_{0,\vec\bullet}} \{Y\}.
\]
If $\scr C$ has a symmetric monoidal structure, we define
\[\Omega\scr C=\scr C(\mathbf{1},\mathbf{1}).\]
Note that $\Omega\scr C$ is a symmetric monoidal $(\infty,n-1)$-category, with $(\Omega\scr C)[m]=\scr C[m](\mathbf{1},\mathbf{1})$.

If $\Theta$ is an $(n-1)$-computad, there is an ``unreduced suspension'' $n$-computad $\Sigma\Theta$ with the property that, for every $(\infty,n)$-category $\scr C$,
\[
\maps^h(\Sigma\Theta,\scr C)\simeq \maps^h(\Theta,\scr C_{1,\vec\bullet}).
\]
Explicitly, $\Sigma\Theta$ has two vertices $s$ and $t$, which are the source and target of every generating $1$-morphism, and for every $k\geq 0$ there is a bijection $\sigma$ between the generating $k$-morphisms of $\Theta$ and the generating $(k+1)$-morphisms of $\Sigma\Theta$, compatible with sources and targets.
 Note that the underlying CW complex of $\Sigma\Theta$ is the unreduced suspension of that of $\Theta$.

For the proof of the following lemma, the reader will need to have the inductive definition of the $\Theta^{\vec k;\vec l}$'s from \cite[Definition 5.7]{JFS} at hand.

\begin{lemma}\label{lem:Thetas}
	There is a unique family of isomorphisms of computads
	\[
	\Sigma\Theta^{\vec k;\vec l} \simeq *\cup_{\Theta^{\vec k}} \Theta^{\vec k;1,\vec l} \cup_{\Theta^{\vec k}} *,
	\]
	defined for all tuples of natural numbers $\vec k$ and $\vec l$, such that:
	\begin{itemize}
		\item the vertices $s$ and $t$ are sent to the collapsed source and target $\Theta^{\vec k}$'s, respectively;
		\item a generating morphism of type $\sigma(\theta_{i,j})$ is sent to a generating morphism of type $\theta_{i,j+1}$;
		\item these isomorphisms are natural with respect to the structural inclusions $\Theta^{\vec m;\vec n}\into \Theta^{\vec k;\vec l}$.
	\end{itemize}
\end{lemma}

In the last condition, the ``structural inclusions'' are the horizontal source and target inclusions 
\[s_h,t_h\colon \Theta^{(p-1);(q)}\into \Theta^{(p);(q)},\]
 the $n$ horizontal inclusions $\Theta^{(i),\vec k;\vec l}\into \Theta^{(i-1),n,\vec k;\vec l}$, and their vertical analogues. Both assignments $\Theta^{\vec k;\vec l}\mapsto \Sigma\Theta^{\vec k;\vec l}$ and $\Theta^{\vec k;\vec l}\mapsto \Theta^{\vec k;1,\vec l}$ are functorial with respect to these inclusions in an obvious way.

\begin{proof}
	As a reality check, note that the underlying CW complexes are homeomorphic, as an instance of the fact that the unreduced suspension $\Sigma(A\times B)$ is obtained from $A\times\Sigma B$ by collapsing $A\times 0$ and $A\times 1$. We must however verify that this homeomorphism preserves the directionality of each cell.
	
	The three conditions determine the obvious isomorphisms $\Sigma\Theta^{(p)}\simeq \Theta^{(p+1)}$.
	Suppose that $\vec k = (p)$ and $\vec l=(q)$ for some $p,q\geq 1$. We extend the previous isomorphisms to this case by induction on $p+q$.
	The skeletons $\del\Theta^{(p);(q)}$ and $\del\Theta^{(p);(q+1)}$ are obtained by gluing lower-dimensional $\Theta$'s in the exact same way (see \cite[Remark 3.6]{JFS}). Using the induction hypothesis and the third condition, and noting that $\Sigma$ preserves pushouts of computads, we obtain the isomorphism
	\[
	\del\Sigma\Theta^{(p);(q)}=\Sigma\del\Theta^{(p);(q)} \simeq *\cup_{\Theta^{(p)}} \del\Theta^{(p);(q+1)} \cup_{\Theta^{(p)}} * = \del(*\cup_{\Theta^{(p)}} \Theta^{(p);(q+1)} \cup_{\Theta^{(p)}} *).
	\]
	Both $\Sigma\Theta^{(p);(q)}$ and $*\cup_{\Theta^{(p)}} \Theta^{(p);(q+1)} \cup_{\Theta^{(p)}} *$ are obtained from their skeletons by adjoining a single generating $(p+q+1)$-morphism: for $\Sigma\Theta^{(p);(q)}$, one adds the morphism $\sigma(\theta_{p,q})$, which is uniquely determined by having $\sigma(s_h\theta_{p-1,q})$ in its source and $\sigma(t_h\theta_{p-1,q})$ in its target; for $*\cup_{\Theta^{(p)}} \Theta^{(p);(q+1)} \cup_{\Theta^{(p)}} *$, one adds the morphism $\theta_{p,q+1}$, which is uniquely determined by having $s_h\theta_{p-1,q+1}$ in its source and $t_h\theta_{p-1,q+1}$ in its target. By the second and third conditions, the isomorphism between the skeletons extends uniquely to an isomorphism $\Sigma\Theta^{(p);(q)}\simeq *\cup_{\Theta^{(p)}} \Theta^{(p);(q+1)} \cup_{\Theta^{(p)}} *$. Finally, these isomorphisms extend uniquely to the general case by induction, using the third condition.
\end{proof}

\begin{proposition}\label{prop:laxloops}
	Let $\scr C$ be an $(\infty,n)$-category, $\vec k\in(\Delta^\op)^{n-1}$, and $X,Y\in\scr C$. Then
	\[
	\Lax_{\vec k,0}(\scr C)(X,Y) \simeq \Lax_{\vec k}(\scr C(X,Y)),
	\]
	naturally in $\vec k$.
	In particular, if $\scr C$ is a symmetric monoidal $(\infty,n)$-category, then
	\[
	\Omega\Lax_{\vec k,0}(\scr C) \simeq \Lax_{\vec k}(\Omega\scr C).
	\]
\end{proposition}

\begin{proof}
	We have natural equivalences
	\begin{align*}
	\Lax_{\vec k,0}(\scr C)(X,Y)_{\vec l} &= \{X\}\times^h_{\scr C_{\vec k,0}} \Lax_{\vec k,0}(\scr C)_{1,\vec l} \times^h_{\scr C_{\vec k,0}} \{Y\} \\
	&\simeq \{X\}\times^h_{\scr C_{(0)}} (\scr C_{(0)}\times^h_{\scr C_{\vec k,0}} \Lax_{\vec k,0}(\scr C)_{1,\vec l} \times^h_{\scr C_{\vec k,0}}\scr C_{(0)}) \times^h_{\scr C_{(0)}} \{Y\} \\
	&\simeq \{X\} \times^h_{\scr C_{(0)}} \maps^h(*\cup_{\Theta^{\vec k}} \Theta^{\vec k;1,\vec l} \cup_{\Theta^{\vec k}} *,\scr C) \times^h_{\scr C_{(0)}} \{Y\}\\
	&\stackrel{(*)}\simeq \{X\} \times^h_{\scr C_{(0)}} \maps^h(\Sigma\Theta^{\vec k;\vec l},\scr C) \times^h_{\scr C_{(0)}} \{Y\}\\
	&\simeq \{X\} \times^h_{\scr C_{(0)}} \maps^h(\Theta^{\vec k;\vec l},\scr C_{1,\vec\bullet}) \times^h_{\scr C_{(0)}} \{Y\}\\
	&\simeq\maps^h(\Theta^{\vec k;\vec l},\{X\}\times^h_{\scr C_{(0)}}\scr C_{1,\vec\bullet}\times^h_{\scr C_{(0)}}\{Y\}) = \Lax_{\vec k}(\scr C(X,Y))_{\vec l},
	\end{align*}
	where $(*)$ is the isomorphism from Lemma~\ref{lem:Thetas}.
\end{proof}

Let $\scr B$ and $\scr C$ be symmetric monoidal $(\infty,n)$-categories. Then there is a morphism of complete $(n-1)$-fold Segal spaces
\begin{equation}\label{eqn:Omega}
\Omega\colon \iota_{n-1}\Fun^\oplax_\tens(\scr B,\scr C) \to \Fun^\oplax_\tens(\Omega\scr B,\Omega\scr C)
\end{equation}
defined levelwise as follows. Its component at $\vec k\in(\Delta^\op)^{n-1}$ is the composition
\[
 \maps^h_\tens(\scr B,\Lax_{\vec k,0}(\scr C)) \xrightarrow \Omega \maps^h_{\tens}(\Omega\scr B,\Omega\Lax_{\vec k,0}(\scr C)) \simeq \maps^h_\tens(\Omega\scr B,\Lax_{\vec k}(\Omega\scr C)),
\]
where the equivalence is Proposition~\ref{prop:laxloops}.

If $\scr B$ and $\scr C$ are (symmetric monoidal) $(\infty,n)$-categories, there is an evaluation morphism of complete $n$-fold Segal spaces
\begin{equation}\label{eqn:evaluation}
\iota_0\scr B \times \Fun^\oplax_{(\tens)}(\scr B,\scr C) \to \scr C
\end{equation}
defined levelwise as follows. Its component at $\vec k\in(\Delta^\op)^n$ is the evaluation map
\[
\scr B_{(0)} \times \maps^h_{(\tens)}(\scr B,\Lax_{\vec k}(\scr C)) \to \Lax_{\vec k}(\scr C)_{(0)} = \scr C_{\vec k}.
\]

\begin{definition}\label{def:trace}
	Let $\scr C$ be a symmetric monoidal $(\infty,n)$-category. Then the trace functor
	\[\Tr\colon \End(\scr C)\to\Omega\scr C\]
	is the composition
	\[
	\End(\scr C)=\iota_{n-1}\Fun^\oplax_\tens(\Fr^\rig(B\N),\scr C) \xrightarrow\Omega \Fun^\oplax_\tens(\Omega \Fr^\rig(B\N),\Omega\scr C) \to \Omega\scr C,
	\]
	where the last map is evaluation at the trace of the walking endomorphism.
\end{definition}

We can describe this functor more explicitly as follows.
If $p\leq n-1$, recall that a $p$-morphism in $\End(\scr C)$ is a dualizable object of $\Lax_{(p)}(\scr C)$ with an endomorphism. By definition, the functor $\Tr$ sends such a $p$-morphism to the trace of the given endomorphism, which is an element of $(\Omega\Lax_{(p)}(\scr C))_{(0)}\simeq (\Omega\scr C)_{(p)}$. The case $p=1$ is made explicit by Lemma~\ref{lem:dualizability}: a dualizable object of $\Lax_{(1)}(\scr C)$ is a right dualizable morphism $\phi\colon X\to Y$ in $\scr C$, and an endomorphism of such is a triple $(f,g,\alpha\colon \phi f\to g\phi)$. The trace of this endomorphism in $\Lax_{(1)}(\scr C)$ is exactly the endomorphism of $\id_\mathbf{1}$ in $\Lax_{(1)}(\scr C)$ depicted in~\eqref{eqn:2trace}.

Note that the trace functor of Definition~\ref{def:trace} is natural in $\scr C$. It can be upgraded to a symmetric monoidal functor using a standard trick: any symmetric monoidal $(\infty,n)$-category $\scr C$ is canonically a symmetric monoidal object \emph{in} symmetric monoidal $(\infty,n)$-categories.
	More precisely, there is a diagram
	\begin{tikzequation}
		\label{eqn:tildeC}
		\diagram{ & \Cat^\tens_{(\infty,n)} \\ \Fin_* & \Cat_{(\infty,n)} \\ };
		\arrows (12) edge node[right]{$\ev_{[1]}$} (22) (21-) edge node[below]{$\scr C$} (-22) (21) edge[dashed] node[above left]{$\tilde{\scr C}$} (12);
	\end{tikzequation}
	where $\tilde{\scr C}[n][m]=\scr C[nm]$ (see \cite[\S2.5]{TV2}).
If $\scr A$ is an $\infty$-category with finite products and $F\colon \Cat_{(\infty,n)}^\otimes\to \scr A$ is a functor that preserves finite products, then
\[
\Cat_{(\infty,n)}^\otimes \to \CMon(\scr A), \quad \scr C\mapsto F\circ \tilde{\scr C},
\]
is a lift of $F$ to the $\infty$-category of commutative monoids in $\scr A$.
Applying this to the functor
\[
\End\colon \Cat_{(\infty,n)}^\otimes \to\Cat_{(\infty,n-1)},
\]
we obtain a canonical symmetric monoidal structure on $\End(\scr C)$, namely $[n]\mapsto\End(\tilde{\scr C}[n])$.
At the level of objects, this symmetric monoidal product is given by $(X,f)\tens (Y,g)=(X\tens Y,f\tens g)$.

\begin{definition}\label{def:monoidaltrace}
	Let $\scr C$ be a symmetric monoidal $(\infty,n)$-category.
	The symmetric monoidal trace functor
	\[\Tr_\otimes\colon \End(\scr C)\to \Omega\scr C\]
	is the symmetric monoidal functor whose $[n]$th component is \[\Tr\colon \End(\tilde{\scr C}[n]) \to \Omega(\tilde {\scr C}[n])=\scr C[n](\mathbf{1},\mathbf{1}).\]
\end{definition}

\begin{remark}\label{rmk:kfoldTrace}
The trace functor of Definition~\ref{def:monoidaltrace} may be iterated, yielding for any $k\leq n$ a symmetric monoidal $(\infty,n-k)$-functor
\[
\Tr_\otimes^{(k)}\colon \End^k(\scr C) \to \Omega^k\scr C
\]
defined inductively as the composite
\[
\End(\End^{k-1}(\scr C)) \xrightarrow{\End(\Tr_\otimes^{(k-1)})} \End(\Omega^{k-1}\scr C) \xrightarrow{\Tr_\otimes} \Omega\Omega^{k-1}\scr C.
\]
Roughly speaking, an object in $\End^k(\scr C)$ is an object of $\scr C$ equipped with $k$ laxly commuting endomorphisms, with the minimal dualizability conditions that make it possible to take their traces successively. For example, if $n=k=2$, $\End^2(\scr C)$ is the $\infty$-groupoid of lax squares in the subcategory $\scr C^\mathrm{fd}\subset\scr C$ of fully dualizable objects, and $\Tr_\otimes^{(2)}\colon \End^2(\scr C)\to\Omega^2\scr C$ is a homotopy coherent and multiplicative enhancement of the secondary trace considered in \cite{BN2}.
\end{remark}

\subsection{The circle-invariant trace}
\label{sub:S1trace}

We conclude this section with a generalization of the $S^1$-invariant trace from \cite{TV2}. First we introduce the subcategory $\Aut(\scr C)\subset\End(\scr C)$ such that $\Aut(\scr C)_{\vec k}$ is the sub-space of $\End(\scr C)_{\vec k}$ consisting of those pairs $(X\in\Lax_{\vec k}(\scr C),f\colon X\to X)$ where $f$ is an equivalence. Equivalently, if we denote by $S^1=B\Z$ the groupoid completion of $B\N$, then:
\[
\Aut(\scr C) = \Fun^\oplax_\tens(\Fr^\rig(S^1),\scr C).
\]
Using~\eqref{eqn:tildeC}, the assignment $[n]\mapsto \Aut(\tilde{\scr C}[n])$ defines a symmetric monoidal structure on $\Aut(\scr C)$ such that the inclusion $\Aut(\scr C)\subset\End(\scr C)$ is symmetric monoidal.

The groupoid $S^1$ has the structure of an $\infty$-group, \ie, a group object in the $\infty$-category $\scr S$. This group structure can be specified in many equivalent ways: it is induced by complex multiplication on the unit circle in $\bb C$, by the group structure on the simplicial bar construction on $\Z$, or by the homotopy equivalence $S^1\simeq\Omega(\bb C\bb P^\infty)$. 

\begin{definition}
If $X$ is an object in an $\infty$-category $\scr C$ and $G$ is an $\infty$-group, an action of $G$ on $X$ is a functor $BG\to \scr C$ sending the base point to $X$.
\end{definition}

An action of $S^1$ on $X\in\scr C$ is thus a morphism of $E_2$-spaces $\Z\to\Omega^2_X\iota_0\scr C$. Such an action determines a self-homotopy of $\id_X$ (the image of $1\in\Z$), but also some additional data, since $\Z$ is not freely generated by $1$ as an $E_2$-space.

Being an $\infty$-group, $S^1$ acts on itself and hence it acts on $\Aut(\scr C)$. The self-equivalence of the identity functor on $\Aut(\scr C)$ induced by this action is given by $(f,\id_{f^2})\colon (X,f)\to (X,f)$.

\begin{theorem}\label{thm:S1trace}
	The symmetric monoidal trace functor
\[
\Tr_\otimes\colon \Aut(\scr C) \to\Omega\scr C
\]
(Definition~\ref{def:monoidaltrace})
admits a canonical $S^1$-invariant refinement which is natural in the symmetric monoidal $(\infty,n)$-category $\scr C$.
\end{theorem}

\begin{proof}
The morphisms~\eqref{eqn:Omega} and~\eqref{eqn:evaluation} are natural in both $\scr B$ and $\scr C$. 
In particular, they give rise to a morphism of $\infty$-groupoids
\[
\iota_0\Omega\scr B \to \Map(\iota_{n-1}\Fun^\oplax_\tens(\scr B,\ph), \Omega)
\]
natural in $\scr B$,
where the mapping space on the right-hand side is taken in the $(\infty,1)$-category of functors $\Cat_{(\infty,n)}^\tens \to \Cat_{(\infty,n-1)}$. Moreover, the construction~\eqref{eqn:tildeC} gives a natural map
\[
\Map(\iota_{n-1}\Fun^\oplax_\tens(\scr B,\ph), \Omega) \to \Map_{\otimes}(\iota_{n-1}\Fun^\oplax_\tens(\scr B,\ph), \Omega),
\]
where $\Map_{\otimes}$ denotes a mapping space in the $(\infty,1)$-category of functors $\Cat_{(\infty,n)}^\tens \to \Cat_{(\infty,n-1)}^\otimes$.
Taking $\scr B=\Fr^\rig(S^1)$ with its action of $S^1$, we obtain an $S^1$-equivariant morphism of $\infty$-groupoids
\[
\Omega\Fr^\rig(S^1) \to \Map_{\otimes}(\Aut, \Omega),
\]
which sends the trace of the walking automorphism to $\Tr_{\otimes}$.
Taking homotopy $S^1$-fixed points, we obtain a commutative square
\begin{tikzequation}\label{eqn:hS1}
	\diagram{
	(\Omega\Fr^\rig(S^1))^{hS^1} & \Map_{\otimes}(\Aut, \Omega)^{hS^1} \\
	\Omega\Fr^\rig(S^1) & \Map_{\otimes}(\Aut, \Omega)\rlap.\\
	};
	\arrows (11-) edge (-12) (11) edge (21) (21-) edge (-22) (12) edge (22);
\end{tikzequation}
At this point we invoke the $1$-dimensional cobordism hypothesis via \cite[Théorème 2.18]{TV2}: the trace of the walking automorphism lives in a contractible component of the $\infty$-groupoid $\Omega\Fr^\rig(S^1)$, and hence the left vertical map in~\eqref{eqn:hS1} is an equivalence over that component. Thus, the trace of the walking automorphism has a unique $S^1$-invariant refinement, whose image by the top horizontal map of~\eqref{eqn:hS1} is an element of $\Map_{\otimes}(\Aut, \Omega)^{hS^1}$ refining $\Tr_{\otimes}$, as desired.
\end{proof}

\begin{remark}
	By iterating the symmetric monoidal $S^1$-invariant trace, as in Remark~\ref{rmk:kfoldTrace}, we deduce that the $k$-fold trace $\Tr^{(k)}_{\otimes}\colon \Aut^k(\scr C)\to\Omega^k\scr C$ is invariant for the action of the $k$-dimensional torus $(S^1)^k$ on $\Aut^k(\scr C)$.
\end{remark}

\begin{remark}\label{rmk:S1action}
	Concretely, the $S^1$-invariant refinement of the trace provides, for every $(A,f)\in\Aut(\scr C)$, a homotopy between the morphism $\Tr(f)\to \Tr(f)$ induced by $f$ and the identity. When $f=\id_A$, it recovers the action of $S^1$ on the Euler characteristic $\chi(A)$ of $A$ (see \cite[Proposition 4.2.1]{Lurie}). The functoriality of the $S^1$-invariant trace encodes in particular the $S^1$-equivariance of the map $\chi(A)\to \chi(B)$ induced by a right dualizable morphism $A\to B$.
\end{remark}

\section{A localization theorem for traces}
\label{sec:localization}

The goal of this section is to show that, if $\scr C$ is a symmetric monoidal $(\infty,2)$-category $\scr C$ with stable mapping $(\infty,1)$-categories, the trace sends special sequences in $\End(\scr C)$, called \emph{localization sequences}, to cofiber sequences in $\Omega\scr C$. 
As we will see in \S\ref{sub:hochschild}, this result generalizes the localization theorem for topological Hochschild homology from \cite[\S7]{BM}. In fact, the proof is based on the same key idea.

\begin{definition}\label{dfn:linear}
	An $(\infty,2)$-category $\scr C$ is called \emph{linear} if the following conditions hold:
	\begin{itemize}
		\item For every $X,Y\in\scr C$, the $(\infty,1)$-category $\scr C(X,Y)$ is stable.
		\item For every $X,Y,Z\in\scr C$, the composition functor $\scr C(X,Y)\times\scr C(Y,Z)\to \scr C(X,Z)$ is exact in each variable.
	\end{itemize}
	If $\scr C$ is moreover symmetric monoidal, we say that it is \emph{linearly symmetric monoidal} if it is linear and if:
	\begin{itemize}
		\item For every $X,Y,Z\in\scr C$, the functor $(\ph)\tens \id_Z\colon \scr C(X,Y)\to \scr C(X\tens Z,Y\tens Z)$ is exact.
	\end{itemize}
\end{definition}

If $\scr C$ is a linearly symmetric monoidal $(\infty,2)$-category, then the restricted trace functor $\Tr\colon\scr C(X,X) \to \Omega\scr C$ is exact for every dualizable object $X\in\scr C$. Indeed, it can be written as the composite
\[
\scr C(X,X) \xrightarrow{(\ph)\tens\id_{X^\vee}} \scr C(X\tens X^\vee, X\tens X^\vee) \xrightarrow{(\ph)\circ\coev} \scr C(\mathbf{1}, X\tens X^\vee) \xrightarrow{\ev\circ(\ph)} \scr C(\mathbf{1},\mathbf{1})=\Omega\scr C.
\]

Recall that a $1$-morphism $(\phi,\alpha)\colon (X,f)\to (Y,g)$ in $\End(\scr C)$ is a square
\begin{tikzmath}
	\diagram{X & Y \\ X & Y \\ };
	\arrows (11-) edge node[above]{$\phi$} (-12) (11) edge node[left]{$f$} (21) (12) edge node[right]{$g$} (22) (21-) edge node[below]{$\phi$} (-22)
	(11) edge[font=\normalsize,draw=none] node[rotate=45]{$\Longrightarrow$} (22)
	(11) edge[draw=none] node[pos=.3]{$\alpha$} (22);
\end{tikzmath}
where $\phi$ has a right adjoint $\phi^r$.
We say that the morphism $(\phi,\alpha)$ is \emph{right adjointable} if the associated push-pull transformation $\alpha^\flat\colon f\phi^r\to \phi^r g$ is an equivalence. Note that morphisms in $\Aut(\scr C)$ are always right adjointable.

\begin{definition}\label{dfn:exact}
	Let $\scr C$ be a linear $(\infty,2)$-category. A sequence
	\[
	X\xrightarrow\iota Y\xrightarrow\pi Z
	\]
	in $\scr C$ is called a \emph{localization sequence} if the following conditions hold:
	\begin{itemize}
		\item $\iota$ and $\pi$ have right adjoints $\iota^r$ and $\pi^r$;
		\item the composite $\pi\iota$ is a zero object in $\scr C(X,Z)$;
		\item the unit $\eta\colon\id_X\to \iota^r \iota$ and the counit $\epsilon\colon \pi\pi^r\to\id_Z$ are equivalences;
		\item the sequence $\iota\iota^r \to \id_Y \to \pi^r\pi$, with its unique nullhomotopy, is a cofiber sequence in $\scr C(Y,Y)$.
	\end{itemize}
	If $\scr C$ is a linearly symmetric monoidal $(\infty,2)$-category, a sequence
	\[
	(X,f) \xrightarrow{(\iota,\alpha)} (Y,g) \xrightarrow{(\pi,\beta)} (Z,h)
	\]
	in $\End(\scr C)$ is called a \emph{localization sequence}
	if $X\stackrel\iota\to Y\stackrel\pi\to Z$ is a localization sequence and moreover the morphisms $(\iota,\alpha)$ and $(\pi,\beta)$ are right adjointable.
\end{definition}

Given a localization sequence
	\[
	(X,f) \xrightarrow{(\iota,\alpha)} (Y,g) \xrightarrow{(\pi,\beta)} (Z,h)
	\]
in $\End(\scr C)$ and a zero object $0\in \scr C(Y,Y)$, there is a \emph{unique} commutative square
\begin{tikzequation}\label{eqn:nullsquare}
	\diagram{(X,f) & (Y,g) \\ (Y,0) & (Z,h) \\ };
	\arrows (11-) edge node[above]{$(\iota,\alpha)$} (-12) (12) edge node[right]{$(\pi,\beta)$} (22)
	(11) edge node[left]{$(\iota,!)$} (21) (21-) edge node[below]{$(\pi,!)$} (-22);
\end{tikzequation}
in $\End(\scr C)$, since the $\infty$-groupoid of zero objects in $\scr C(X,Z)$ is contractible.
In particular, since $\Tr(0)$ is a zero object in $\Omega\scr C$, the sequence
\begin{equation*}\label{eqn:TrSeq}
	\Tr(f)\to\Tr(g)\to \Tr(h)
\end{equation*}
is equipped with a canonical nullhomotopy.

\begin{theorem}
	\label{thm:localization}
	Let $\scr C$ be a linearly symmetric monoidal $(\infty,2)$-category, and let
	\[
	(X,f) \xrightarrow{(\iota,\alpha)} (Y,g) \xrightarrow{(\pi,\beta)} (Z,h)
	\]
	be a localization sequence in $\End(\scr C)$. Then
	\[
	\Tr(f) \to \Tr(g) \to \Tr(h)
	\]
	is a cofiber sequence in $\Omega\scr C$.
\end{theorem}

\begin{proof}
	Since $\scr C$ is linear, the trace functor $\Tr\colon \scr C(Y,Y)\to \Omega\scr C$ preserves cofiber sequences.
	We will define a cofiber sequence $f'\to g\to h'$ in $\scr C(Y,Y)$ and a diagram
	 \begin{tikzequation}\label{eqn:localization}
	 	\diagram{(X,f) & (Y,g) & (Z,h) \\
		(Y,f') & (Y,g) & (Y,h') \\ };
		\arrows (11-) edge node[above]{$(\iota,\alpha)$} (-12) (12-) edge node[above]{$(\pi,\beta)$} (-13)
		(21-) edge node[below]{$(\id,\alpha')$} (-22) (22-) edge node[below]{$(\id,\beta')$} (-23)
		(11) edge node[left]{$(\iota,\bar\alpha)$} (21) (12) edge[-,vshift=1pt] (22) edge[-,vshift=-1pt] (22) (13) edge[<-] node[right]{$(\pi,\bar\beta)$} (23);
		\end{tikzequation}
		in $\End(\scr C)$, commuting up to homotopy, such that the vertical maps induce equivalences on traces.
		The uniqueness of the square~\eqref{eqn:nullsquare} will then imply that the cofiber sequence $\Tr(f')\to \Tr(g)\to \Tr(h')$ is equivalent to the sequence $\Tr(f)\to\Tr(g)\to\Tr(h)$ with its nullhomotopy, and hence that the latter is a cofiber sequence.
		
		 Let $f'=\iota\iota^r g$, $h'=\pi^r\pi g$, and let
		\begin{gather*}
		\alpha'\colon f'=\iota\iota^r g \stackrel\epsilon\to g,\\
		\beta'\colon g\stackrel\eta\to \pi^r\pi g=h',\\
		\bar\alpha\colon \iota f \stackrel\eta\to \iota f\iota^r\iota \stackrel{\alpha^\flat}\to \iota\iota^r g\iota=f'\iota,\\
		\bar\beta\colon \pi h'=\pi \pi^r\pi g\stackrel\epsilon \to \pi g\stackrel{\beta}\to h\pi.
		\end{gather*}
		The commutativity of the diagram~\eqref{eqn:localization} is then clear. The cofiber sequence $\iota\iota^r\to\id_X\to\pi^r\pi$ shows that $f'\to g\to h'$ is a cofiber sequence in $\scr C(Y,Y)$. It remains to show that $(\iota,\bar\alpha)$ and $(\pi,\bar\beta)$ induce equivalences on traces.

		The morphism induced by $(\iota,\bar\alpha)$ on traces looks as follows:
		\begin{tikzmath}
			\def\rowsep{1em}
			\diagram{
			& X\tens X^\vee & & & X\tens X^\vee & \\
			\mathbf{1} & & & & & \mathbf{1}\rlap. \\
			& Y\tens Y^\vee & & & Y\tens Y^\vee & \\
			};
			\arrows (21) edge (12) edge (32) (12-) edge node[above]{$f\tens \id$} (-15) (32-) edge node[below]{$f'\tens \id$} (-35) (15) edge (26) (35) edge (26) (12) edge node[right]{$\iota\tens\iota^{r\vee}$} (32) (15) edge node[left]{$\iota\tens\iota^{r\vee}$} (35)
			(21) edge[font=\normalsize,draw=none] node[rotate=45]{$\Longleftarrow$} node[pos=.13,rotate=45]{$\Longleftarrow$} node[pos=.87,rotate=45]{$\Longleftarrow$} (26)
			(21) edge[draw=none] node[above left]{$\bar\alpha\tens\id$} (26)
			;
		\end{tikzmath}
		The third $2$-morphism is an equivalence because $\eta\colon \id_X\to \iota^r\iota$ is an equivalence.
		Since moreover $(\iota,\alpha)$ is right adjointable, $\bar\alpha$ is also an equivalence.
		The first $2$-morphism becomes an equivalence when post-composed with $\iota^r\tens \id$ and \textit{a fortiori} when post-composed with $f'\tens\id$, since $f'\simeq \iota f\iota^r$. This shows that $\Tr(\iota,\bar\alpha)$ is an equivalence.
		
		The morphism induced by $(\pi,\bar\beta)$ on traces looks as follows:
		\begin{tikzmath}
			\def\rowsep{1em}
			\diagram{
			& Y\tens Y^\vee & & & Y\tens Y^\vee & \\
			\mathbf{1} & & & & & \mathbf{1}\rlap. \\
			& Z\tens Z^\vee & & & Z\tens Z^\vee & \\
			};
			\arrows (21) edge (12) edge (32) (12-) edge node[above]{$h'\tens \id$} (-15) (32-) edge node[below]{$h\tens \id$} (-35) (15) edge (26) (35) edge (26) (12) edge node[right]{$\pi\tens\pi^{r\vee}$} (32) (15) edge node[left]{$\pi\tens\pi^{r\vee}$} (35)
			(21) edge[font=\normalsize,draw=none] node[rotate=45]{$\Longleftarrow$} node[pos=.13,rotate=45]{$\Longleftarrow$} node[pos=.87,rotate=45]{$\Longleftarrow$} (26)
			(21) edge[draw=none] node[above left]{$\bar\beta\tens\id$} (26)
			;
		\end{tikzmath}
		The first $2$-morphism is an equivalence because $\epsilon\colon \pi\pi^r\to\id_Z$ is an equivalence. The last $2$-morphism is an equivalence when it is pre-composed with $\pi^r\tens\id$, which happens if we decompose $\bar\beta$ as
		\[
		\pi h'\stackrel\eta\to \pi h'\pi^r\pi \stackrel{\bar\beta^\flat}\to \pi\pi^r h\pi \stackrel\epsilon \to h\pi.
		\]
		By assumption, $\epsilon$ is an equivalence. Since moreover $(\pi,\beta)$ is right adjointable, $\bar\beta^\flat$ is also an equivalence.
		It will therefore suffice to show that the $2$-morphism
		\[
		\Tr(h')=\ev_Y (h'\tens \id)\coev_Y \stackrel\eta\to \ev_Y(h'\tens \id) (\pi^r\pi\tens\id)\coev_Y = \Tr(h'\pi^r\pi)
		\]
		is an equivalence. This follows from the fact that $\eta h'\colon h'\to \pi^r\pi h'$ is an equivalence and the cyclicity of the trace.
\end{proof}

\begin{remark}
	It is worth noting that the proof of Theorem~\ref{thm:localization} did not make full use of the linearity of $\scr C$. The same result holds if we only assume that each $\scr C(X,Y)$ is a pointed $\infty$-category admitting cofibers, and that composition and tensoring preserve zero objects and cofiber sequences.
\end{remark}

\section{Preliminaries on $\cE$-linear categories}
\label{prel}
In this section, we study the linearly symmetric monoidal $(\infty, 2)$-category of $\cE$-linear categories, where $\cE$ is a symmetric monoidal $\infty$-category which is small, stable, idempotent complete, and rigid.

\subsection{Categories tensored over $\cE$}
\label{sec:enrich}

\begin{definition}
We denote by
\begin{itemize}
\item $\Cat^\perf$, the $\infty$-category of small, stable, and idempotent complete $\infty$-categories and exact functors. 
\item $\Pr^{\L}_\St$, the $\infty$-category of stable presentable $\infty$-categories and left adjoint functors.  
\end{itemize}
\end{definition}
By \cite[Proposition 5.5.7.10]{HTT}, the ind-completion functor $\Ind\colon \Cat^\perf \to \Pr^\L_\St$ induces an equivalence between $\Cat^\perf$ and the subcategory of $\Pr^\L_\St$ whose objects are the compactly generated stable $\infty$-categories and whose morphisms are the left adjoint functors preserving compact objects.

Given $\infty$-categories (resp.\ stable $\infty$-categories) $\scr A$ and $\scr B$, we denote by $\Fun^\L(\scr A,\scr B)$ (resp.\ by $\Fun^\ex(\scr A,\scr B)$) the full subcategory of $\Fun(\scr A,\scr B)$ spanned by left adjoint functors (resp.\ by exact functors). If $\scr A$ is presentable, a functor $\scr A\to\scr B$ is left adjoint if and only if it preserves small colimits \cite[Corollary 5.5.2.9]{HTT}. If $\scr A$ and $\scr B$ are stable, a functor $\scr A\to\scr B$ is exact if and only if it preserves finite colimits or finite limits \cite[Proposition 1.1.4.1]{HA}.
As a special case of \cite[Proposition 5.3.6.2]{HTT}, the Yoneda embedding $j\colon\scr A\into\Ind(\scr A)$ induces an equivalence of $\infty$-categories
\begin{equation}\label{eqn:PKR}
\Fun^\mathrm{ex}(\scr A,\scr B) \simeq \Fun^\L(\Ind(\scr A),\scr B),
\end{equation}
for any $\scr A\in\Cat^\perf$ and any cocomplete stable $\infty$-category $\scr B$.

Recall from \cite[Section 3.1]{BGT} that $\Cat^\perf$, and $\Pr^\L_\St$ admit symmetric monoidal structures such that
$\Ind\colon \Cat^\perf\to\Pr^\L_\St$ is a symmetric monoidal functor. Their tensor products will be denoted by $\tens$ and $\tens^\L$, respectively. Given $\scr A,\scr B\in\Cat^\perf$, the stable $\infty$-category $\scr A\tens\scr B$ is the recipient of the universal functor $\scr A\times\scr B\to\scr C$ which is exact in each variable. In other words, we have an equivalence of $\infty$-categories
\[
\Fun^\ex(\scr A\tens\scr B,\scr C)\simeq \Fun^\ex(\scr A,\Fun^\ex(\scr B,\scr C))
\]
for all $\scr A,\scr B,\scr C\in\Cat^\perf$. Similarly, given $\scr A,\scr B,\scr C\in\Pr^\L_\St$, the tensor product $\tens^\L$ is characterized by
\[
\Fun^\L(\scr A\tens^\L\scr B,\scr C) \simeq \Fun^\L(\scr A,\Fun^\L(\scr B,\scr C)).
\]

Let $\cC$ be a complete and cocomplete symmetric monoidal $\infty$-category whose tensor product preserves geometric realizations, 
and let $A\in\mathrm{CAlg(\mathrm{\cC})}$. Denote by $\Mod_A(\cC)$ the $\infty$-category of $A$-modules in $\cC$ (see \cite[Section 4.5]{HA}).
By \cite[Theorem 4.5.2.1]{HA}, $\Mod_A(\cC)$ inherits a symmetric monoidal structure whose tensor product we will denote by $\tens_{A}$. Given two $A$-modules $M$ and $N$, $M\tens_AN$ is the colimit of the usual simplicial diagram
\begin{tikzmath}
	\def\colsep{.85em}
	\diagram{
	\dotsb & M\tens A\tens N & M\tens N\rlap. \\
	};
	\arrows
	(11-) edge[-top,vshift=2*\dbl] (-12) edge[-mid] (-12) edge[-bot,vshift=-2*\dbl] (-12)
	(12-) edge[-top,vshift=\dbl] (-13) edge[-bot,vshift=-\dbl] (-13);
\end{tikzmath}
In other words, $M\tens_AN$ is the recipient of the universal $A$-bilinear map $M\tens N\to P$.
If $\cC$ has an internal Hom object $\Hom(\ph,\ph)$, then $\Mod_A(\scr C)$ also has an internal Hom object $\Hom_A(\ph,\ph)$: given $A$-modules $M$ and $N$, $\Hom_A(M,N)$ is the limit of the cosimplicial diagram
\begin{tikzmath}
	\def\colsep{.85em}
	\diagram{
	\Hom(M,N) & \Hom( A,\Hom(M,N)) & \dotsb\rlap. \\
	};
	\arrows
	(12-) edge[-top,vshift=2*\dbl] (-13) edge[-mid] (-13) edge[-bot,vshift=-2*\dbl] (-13)
	(11-) edge[-top,vshift=\dbl] (-12) edge[-bot,vshift=-\dbl] (-12);
\end{tikzmath}

If $\widehat{\Cat}$ denotes the $\infty$-category of (possibly large) $\infty$-categories, a commutative algebra $\scr E$ in $\widehat{\Cat}$ (with the cartesian symmetric monoidal structure) is exactly a symmetric monoidal $\infty$-category; we will refer to objects and morphisms of $\Mod_{\scr E}(\widehat{\Cat})$ as \emph{$\scr E$-module $\infty$-categories} and \emph{$\scr E$-module functors}, respectively.
Given $\scr E$-module $\infty$-categories $\scr A$ and $\scr B$, we denote by $\Fun_{\scr E}(\scr A,\scr B)$ the $\infty$-category of $\scr E$-module functors from $\scr A$ to $\scr B$ \cite[Definition 4.6.2.7]{HA}, which can be described as the limit of the cosimplicial diagram
\begin{tikzmath}
	\def\colsep{.85em}
	\diagram{
	\Fun(\scr A,\scr B) & \Fun(\scr E,\Fun(\scr A,\scr B)) & \dotsb \\
	};
	\arrows
	(12-) edge[-top,vshift=2*\dbl] (-13) edge[-mid] (-13) edge[-bot,vshift=-2*\dbl] (-13)
	(11-) edge[-top,vshift=\dbl] (-12) edge[-bot,vshift=-\dbl] (-12);
\end{tikzmath}
(see \cite[Lemma 4.8.4.12]{HA}). In other words, $\Fun_{\scr E}$ is the internal Hom object in $\Mod_{\cE}(\widehat{\Cat})$.
We denote by $\Fun^\L_{\scr E}(\scr A,\scr B)$ the full subcategory of $\Fun_{\scr E}(\scr A,\scr B)$ defined by the cartesian square
\begin{tikzmath}
	\diagram{
	\Fun_{\scr E}^\L(\scr A,\scr B) & \Fun_{\scr E}(\scr A,\scr B) \\
	\Fun^\L(\scr A,\scr B) & \Fun(\scr A,\scr B)\rlap. \\
	};
	\arrows (11-) edge[c->] (-12) (21-) edge[c->] (-22) (11) edge (21) (12) edge (22);
\end{tikzmath}
If $\scr A$ and $\scr B$ are stable, 
 we define $\Fun^\ex_{\scr E}(\scr A,\scr B)\subset\Fun_{\scr E}(\scr A,\scr B)$ similarly.
 The objects of $\Fun^\ex_{\scr E}(\scr A,\scr B)$ will also be called \emph{$\scr E$-linear functors}.

\begin{definition}
\label{def:Emod}
Let $\cE$ be a commutative algebra in $\Cat^\perf$, \ie, a small, stable, and idempotent complete symmetric monoidal $\infty$-category whose tensor product $\scr E\times\scr E\to\scr E$ is exact in each variable.
We set 
\begin{itemize}
\item $\Cat^{\perf}(\cE):= \Mod_{\cE}(\Cat^{\perf})$.  
\item $\Pr^\L(\cE) :=  \Mod_{\Ind(\cE)}(\Pr^\L_\St)$.
\end{itemize}
\end{definition}

\begin{remark}
	The $\infty$-category $\Sp^\omega$ of finite spectra is the unit for the symmetric monoidal structure on $\Cat^\perf$. Hence, there are identifications $ \Cat^\perf(\Sp^\omega)=\Cat^\perf$ and $\Pr^\L(\Sp^\omega)=  \Pr^\L_\St$. 
\end{remark}

We denote the tensor products in $\ccat{perf}{\cE}$ and 
$\Pr^\L(\cE)$ by $\tens_{\scr E}$ and $\tens^\L_{\scr E}$, respectively.
Thus, $\Fun^\ex_\cE$ is an internal Hom object in $\Cat^\perf(\cE)$, and $\Fun^\L_\cE$ is an internal Hom object in $\Pr^\L(\cE)$.
Since
$\Ind\colon \Cat^\perf\to \Pr^\L_\St$ is a symmetric monoidal functor, it lifts to a symmetric monoidal functor between $\infty$-categories of modules:
\[
	\Ind\colon \Cat^\perf(\cE) \to \Pr^\L(\cE).
\]

\begin{proposition}\label{prop:IdemInd}
	Let $\scr E\in\CAlg(\Cat^\perf)$ and $\scr A\in\Cat^\perf(\scr E)$.
	For any cocomplete stable $\infty$-category $\scr B$ with a colimit-preserving action of $\scr E$, the Yoneda embedding $\scr A\subset\Ind(\scr A)$ induces an equivalence of $\scr E$-modules
		\[
		\Fun^\ex_{\scr E}(\scr A,\scr B)\simeq \Fun^\L_{\scr E}(\Ind(\scr A),\scr B).
		\]
\end{proposition}

\begin{proof}
	This is a straightforward consequence of~\eqref{eqn:PKR}, using the definitions of $\Fun^\ex_{\scr E}$ and $\Fun^\L_{\scr E}$.
\end{proof}

Recall that an $\infty$-category is \emph{compactly generated} if it has the form $\Ind(\scr C)$, where $\scr C$ is small and has finite colimits.

\begin{lemma}\label{lem:coco}
	Let $\scr C$ be a symmetric monoidal $\infty$-category and $\scr M$ a $\scr C$-module. Suppose that $\scr M$ is compactly generated and that, for each $c\in\scr C$, the action $c\tens\ph\colon\scr M\to\scr M$ preserves colimits. Then, for every $A\in\CAlg(\scr C)$, the $\infty$-category $\Mod_A(\scr M)$ is compactly generated.
\end{lemma}

\begin{proof}
	By \cite[Corollary 4.2.3.7]{HA}, $\Mod_A(\scr M)$ is presentable. Let $F\colon \Ind(\Mod_A(\scr M)^\omega)\to\Mod_A(\scr M)$ be the functor induced by the inclusion $\Mod_A(\scr M)^\omega\into \Mod_A(\scr M)$. Since compact objects are stable under finite colimits, it remains to show that $F$ is an equivalence.
By \cite[Proposition 5.3.5.11 (1)]{HTT}, $F$ is fully faithful, and by \cite[Proposition 5.5.1.9]{HTT}, $F$ preserves colimits.
Hence, to conclude the proof, it suffices to show that $\Mod_A(\scr M)$ is generated under colimits by its compact objects.
Any $A$-module in $\scr M$ is canonically the colimit of a (split) simplicial diagram whose terms are free $A$-modules. Free $A$-modules are in turn filtered colimits of free $A$-modules of the form $A\tens X$ with $X\in\scr M^\omega$. Since the forgetful functor $\Mod_A(\scr M)\to\scr M$ preserves filtered colimits, such $A$-modules are compact in $\Mod_A(\scr M)$.
\end{proof}

\begin{proposition}\label{prop:coco}
	For every $\cE\in\CAlg(\Cat^\perf)$, the $\infty$-category $\ccat{perf}{\cE}$ is compactly generated.
\end{proposition}
\begin{proof}
By \cite[Corollary 4.25]{BGT}, the $\infty$-category $ \Cat^{\perf}$ is compactly generated.
The result now follows from Lemma~\ref{lem:coco} applied to $\scr C=\scr M=\Cat^\perf$.
\end{proof}

\subsection{The enriched Yoneda embedding} 
\label{sub:mapping}
If $\scr E$ is a commutative algebra in $\Cat^\perf$ and $\scr A\in\Cat^\perf(\scr E)$, then $\scr A$ is naturally enriched over $\Ind(\scr E)$.
Indeed, given $a\in\scr A$, the functor $\scr E\to\scr A$ sending $e$ to $e\tens a$ preserves finite colimits and hence admits an ind-right adjoint $\scr A^{\scr E}(a,\ph)\colon \scr A\to\Ind(\scr E)$. One can hence define a functor
\begin{equation}\label{eqn:yoneda}
\scr A \to \Fun^\ex(\scr A^\op,\Ind(\scr E)),
\end{equation}
given informally by $a\mapsto \scr A^{\scr E}(\ph,a)$. When $\scr E=\Sp^\omega$, the functor \eqref{eqn:yoneda} is fully faithful, but this is not true for more general $\scr E$. Indeed, given $a\in\scr A$, the functor $\scr A^{\scr E}(\ph,a)$ comes with extra structure, namely that of an \emph{$\Ind(\scr E)$-enriched} functor.
Below we will consider an assumption on $\scr E$ that greatly simplifies the situation: we will assume that $\scr E$ is \emph{rigid}, i.e., that every object of $\scr E$ is dualizable. This will suffice for our later applications. In that case, we will show that \eqref{eqn:yoneda} factors through a fully faithful $\scr E$-linear embedding $\scr A\into \Fun_{\scr E}^\ex(\scr A^\op,\Ind(\scr E))$. 

We denote by $\CAlg^\rig(\Cat^\perf)\subset \CAlg(\Cat^\perf)$ the full subcategory spanned by the rigid symmetric monoidal $\infty$-categories.\footnote{One of the reasons rigidity plays a special role is that objects of $\CAlg^\rig(\Cat^\perf)$ are \emph{smooth Frobenius algebras} in $\Cat^\perf$, see \cite[\S4.6.5]{HA}.} We gather some immediate consequences of rigidity in the next proposition:

\begin{proposition}\label{prop:rigid}
Let $\scr E\in\CAlg^\rig(\Cat^\perf)$.
Then:
\begin{enumerate}
	\item There is a canonical equivalence of symmetric monoidal $\infty$-categories
	\[
	\scr E\simeq \scr E^\op,\quad e\mapsto e^\vee = \Hom(e,\mathbf{1}).
	\]
	\item For any $\scr A\in\Pr^\L(\cE)$, the action of $\scr E$ on $\scr A$ restricts to the full subcategory $\scr A^\omega\subset\scr A$ of compact objects. In particular, if $\scr A$ is compactly generated, then it belongs to the essential image of the functor $\Ind\colon \Cat^\perf(\scr E)\to\Pr^\L(\scr E)$.
\item For any $\scr E$-module $\infty$-categories $\scr A$ and $\scr B$ and any $F\in\Fun_{\scr E}^\L(\scr A,\scr B)$, the right adjoint $G\colon \scr B\to\scr A$ of $F$ has a canonical structure of $\scr E$-module functor, and the unit $\id_{\scr A}\to G\circ F$ and counit $F\circ G\to\id_{\scr B}$ have canonical structures of $\scr E$-module natural transformations.
	\item Let $\scr A$ and $\scr B$ be arbitrary $\scr E$-module $\infty$-categories. Then there is a canonical equivalence of $\scr E$-module $\infty$-categories
	\[
	\Fun_{\cE}^\L(\scr A,\scr B)\simeq \Fun_{\cE}^\R(\scr B,\scr A)^\op,
	\]
	sending a left adjoint $\scr E$-module functor to its right adjoint.
\end{enumerate}
\end{proposition}

\begin{proof}
	(1) The functor $e\mapsto e^\vee$ is adjoint to itself, and the unit of the adjunction $e\to (e^\vee)^\vee$ is an equivalence since $e$ is dualizable.
	
	(2) Since $\scr E$ is rigid, $e\otimes\ph\colon\scr A\to\scr A$ is left adjoint to $e^\vee\otimes\ph$, for every $e\in\scr E$. In particular, $e^\vee \otimes\ph$ commutes with filtered colimits. Hence, $e\otimes\ph$ preserves compact objects.

	(3) 
	In the language of \cite[\S7.3.2]{HA}, the $\scr E$-modules $\scr A$ and $\scr B$ are encoded by cocartesian fibrations over the $\infty$-operad $\scr L\scr M^\tens$, and $F$ is an $\scr L\scr M^\otimes$-monoidal functor between them. By \cite[Corollary 7.3.2.7]{HA},\footnote{This corollary is missing the assumption that $F$ should preserve cocartesian edges.} $F$ admits a right adjoint relative to $\scr L\scr M^\otimes$, which is moreover a morphism of $\infty$-operads.
	In other words, $G$ has a structure of right-lax $\scr E$-module functor and the unit and counit of the adjunction are $\scr E$-module transformations. To prove our assertion, it remains to show that, when $\scr E$ is rigid, the right-lax $\scr E$-module structure on $G$ is strict, \ie, for every $e\in\scr E$ and $b\in\scr B$, the map $e\tens G(b)\to G(e\tens b)$ is an equivalence. 
	For any $a\in\scr A$, the induced map $\scr A(a,e\tens G(b))\to\scr A(a,G(e\tens b))$ is the composition
\begin{align*}
 \scr A (a,e\otimes  G(b)) &\simeq \scr A ( e^\vee\otimes a, G(b)) \\& \simeq  \scr B ( F(e^\vee\otimes a), b) \\ &\simeq \scr B ( e^\vee\otimes F(a), b) \\
 &\simeq \scr B ( F(a), e \otimes b) \\ &\simeq \scr A( a, G(e \otimes b)).\end{align*}

(4) As $\scr E$ is rigid, it acts on $\scr A$ and $\scr B$ via left adjoint functors. Hence, $\Fun_{\scr E}^\L(\scr A,\scr B)$ is the limit of the cosimplicial diagram
	\begin{tikzmath}
		\def\colsep{.85em}
		\diagram{
		\Fun^\L(\scr A,\scr B) & \Fun(\scr E,\Fun^\L(\scr A,\scr B)) & \dotsb\rlap. \\
		};
		\arrows
		(12-) edge[-top,vshift=2*\dbl] (-13) edge[-mid] (-13) edge[-bot,vshift=-2*\dbl] (-13)
		(11-) edge[-top,vshift=\dbl] (-12) edge[-bot,vshift=-\dbl] (-12);
	\end{tikzmath}
	Dually, $\Fun_{\scr E}^\R(\scr B,\scr A)^\op$ is the limit of the cosimplicial diagram
	\begin{tikzmath}
		\def\colsep{.85em}
		\diagram{
		\Fun^\R(\scr A,\scr B)^\op & \Fun(\scr E^\op,\Fun^\R(\scr A,\scr B)^\op) & \dotsb\rlap. \\
		};
		\arrows
		(12-) edge[-top,vshift=2*\dbl] (-13) edge[-mid] (-13) edge[-bot,vshift=-2*\dbl] (-13)
		(11-) edge[-top,vshift=\dbl] (-12) edge[-bot,vshift=-\dbl] (-12);
	\end{tikzmath}
	Using the equivalence $\scr E\simeq\scr E^\op$ from (1) and the equivalence $\Fun^\L(\scr A,\scr B)\simeq \Fun^\R(\scr B,\scr A)^\op$ from \cite[Proposition 5.2.6.2]{HTT}, we can identify these two cosimplicial diagrams, and the result follows.
\end{proof}

\begin{proposition}\label{prop:dual}
	Let $\scr E\in\CAlg^\rig(\Cat^\perf)$.
	If $\scr A \in\Pr^\L(\scr E)$ is compactly generated, then it is dualizable with dual $\scr A^\vee=\Ind(\scr A^{\omega,\op})$.
\end{proposition}

\begin{proof}
	Note that $\scr A^\vee$ is a meaningful object of $\Pr^\L(\scr E)$, by Proposition~\ref{prop:rigid} (1,2).
	We must construct an equivalence $\scr A^\vee\tens_{\scr E}^\L \scr B\simeq \Fun^\L_{\scr E}(\scr A,\scr B)$, natural in $\scr B$.
	By Proposition~\ref{prop:rigid} (4), for any $\scr E$-module $\infty$-categories $\scr C$ and $\scr D$, we have
	\[
	\Fun^\L_{\scr E}(\scr C,\scr D)
	\simeq \Fun^\L_{\scr E}(\scr D^\op,\scr C^\op).
	\]
	The rest of the argument is identical to \cite[Proposition 4.8.1.16]{HA}.
	If $\scr B,\scr C\in\Pr^\L(\scr E)$, we have natural equivalences
	\begin{align*}
		\Fun^\L_{\scr E}(\scr A^\vee\tens_{\scr E}^\L\scr B,\scr C) &\simeq \Fun^\L_{\scr E}(\scr A^\vee,\Fun^\L_{\scr E}(\scr B,\scr C)) \\
		& \simeq \Fun^\L_{\scr E}(\scr A^\vee,\Fun^\L_{\scr E}(\scr C^\op,\scr B^\op)) \\
		& \simeq \Fun^\L_{\scr E}(\scr C^\op,\Fun^\L_{\scr E}(\scr A^\vee,\scr B^\op)) \\
		& \simeq \Fun^\L_{\scr E}(\Fun^\L_{\scr E}(\scr A^\vee,\scr B^\op)^\op,\scr C), \\
	\end{align*}
	so that $\scr A^\vee\tens_{\scr E}^\L\scr B\simeq \Fun^\L_{\scr E}(\scr A^\vee,\scr B^\op)^\op$.
	Using Proposition~\ref{prop:IdemInd} twice, we get 
	\[
	\Fun^\L_{\scr E}(\scr A^\vee,\scr B^\op)^\op \simeq \Fun^{\ex}_{\scr E}(\scr A^{\omega,\op},\scr B^\op)^\op
	\simeq \Fun^{\ex}_{\scr E}(\scr A^{\omega},\scr B) \simeq \Fun^\L_{\scr E}(\scr A,\scr B),
	\]
	as desired.
\end{proof}

\begin{corollary}\label{cor:bimod}
	Let $\scr E\in\CAlg^\rig(\Cat^\perf)$. For any $\scr A,\scr B\in\Cat^\perf(\scr E)$, there is a canonical equivalence
	\[
	\Ind(\scr A^\op\tens_{\scr E}\scr B)\simeq \Fun^\ex_{\scr E}(\scr A,\Ind(\scr B))
	\]
	in $\Pr^\L(\scr E)$. In particular, $\Ind(\scr A)\simeq \Fun^\ex_{\scr E}(\scr A^\op,\Ind(\scr E))$.
\end{corollary}

\begin{proof}
	By Proposition~\ref{prop:dual}, we have 
	\[
	\Ind(\scr A^\op)\tens_{\scr E}^\L\Ind(\scr B)\simeq \Fun_{\scr E}^\L(\Ind(\scr A),\Ind(\scr B)).
	\]
	We conclude using Proposition~\ref{prop:IdemInd}.
\end{proof}

By Corollary~\ref{cor:bimod}, if $\scr E$ is rigid and $\scr A\in\Cat^\perf(\scr E)$, we have a fully faithful $\scr E$-linear functor
\[
j_{\scr E}\colon\scr A \into \Fun^\ex_{\scr E}(\scr A^\op,\Ind(\scr E)),
\]
called the \emph{$\scr E$-linear Yoneda embedding}, exhibiting the right-hand side as the ind-completion of $\scr A$. 
An $\scr E$-linear functor $\scr A^{\op}\to\Ind(\scr E)$ (resp.\ $\scr A\to\Ind(\scr E)$) is also called a \emph{right $\scr A$-module} (resp.\ a \emph{left $\scr A$-module}). Thus, $\Ind(\scr A)$ and $\Ind(\scr A)^\vee$ are canonically identified with the $\infty$-categories of right and left $\scr A$-modules, respectively.
Similarly, for $\scr A,\scr B\in\Cat^\perf(\cE)$, $\Ind(\scr A)^\vee\tens_\cE^\L\Ind(\scr B)$ is identified with the $\infty$-category of $\scr E$-bilinear functors $\scr B^\op\times\scr A\to\Ind(\scr E)$, called \emph{$\scr A$-$\scr B$-bimodules}.

\begin{definition}\label{def:CatMor}
	Let $\scr E\in\CAlg^\rig(\Cat^\perf)$. We denote by
	\[
	 \Cat^\Mor(\scr E)\subset \Pr^\L(\scr E)
	\]
	the full subcategory spanned by the compactly generated $\infty$-categories.
\end{definition}

 By Proposition~\ref{prop:rigid} (2), $\Cat^\Mor(\scr E)$ is exactly the essential image of the functor $\Ind\colon\Cat^\perf(\scr E)\to \Pr^\L(\scr E)$. We may therefore think of $\Cat^\Mor(\scr E)$ as having the same objects as $\Cat^\perf(\scr E)$, but the morphisms from $\scr A$ to $\scr B$ are now $\scr A$-$\scr B$-bimodules. Since $\Ind$ is symmetric monoidal, we also deduce that $\Cat^\Mor(\scr E)$ is stable under the tensor product $\tens^\L_\cE$.
Moreover, by Proposition~\ref{prop:IdemInd}, ind-completion induces a fully faithful functor $\Fun_{\scr E}^\ex(\scr A,\scr B)\into \Fun_\cE^\L(\Ind(\cA),\Ind(\cB))$ for every $\scr A,\scr B\in\Cat^\perf(\scr E)$. In particular, we can identify $\Cat^\perf(\scr E)$ with a \emph{wide} subcategory of $\Cat^\Mor(\scr E)$, \ie, a subcategory obtained by discarding some noninvertible morphisms.

\subsection{Dualizable $\cE$-categories}
\label{app:dualizable}
In this subsection, we give a characterization of the dualizable objects in $\Cat^{\perf}(\cE)$, where $\cE\in\CAlg^\rig(\Cat^\perf)$. Our results are straightforward generalizations of facts that are well known in the case when $\cE$ is the $\infty$-category of perfect modules over a commutative ring. The theory in that case is due to To\"en;  
a convenient reference is \cite[Section 3]{BGT}.

\begin{definition}
Let $\scr E\in\CAlg^\rig(\Cat^\perf)$ and let $\cA\in\Cat^{\perf}(\cE)$. We say that $\cA$ is:
\begin{itemize}
\item \emph{smooth} if the object $\Delta\in\mathrm{Ind}(\cA^\op \otimes_\cE  \cA)$ corresponding to the Yoneda embedding $\cA \into\Ind(\cA)$ 
 is compact (see Corollary \ref{cor:bimod});
\item \emph{proper} if, for all $a, a' \in \cA$, the mapping object $\cA^\cE(a, a') \in \mathrm{Ind} (\cE)$ is compact;
 \item \emph{saturated} if it is both smooth and proper.
\end{itemize}
We denote by $\Cat^\sat(\scr E)\subset\Cat^\perf(\scr E)$ the full subcategory of saturated $\scr E$-module $\infty$-categories.
\end{definition}

\begin{lemma}
\label{lem:corest}
Let $\scr E\in\CAlg^\rig(\Cat^\perf)$ and let $\cA\in\Cat^{\perf}(\cE)$. Then $\cA$ is dualizable if and only if the two maps
\begin{gather*}
	\ev_{\Ind(\cA)}\colon \Ind(\cA^\op \otimes_\cE \cA)\simeq \mathrm{Ind}(\cA^\op) \otimes_\cE^\L \mathrm{Ind}(\cA) \rightarrow \mathrm{Ind}(\cE), \\
	\coev_{\Ind(\cA)}\colon \mathrm{Ind}(\cE) \rightarrow \mathrm{Ind}(\cA) \otimes_\cE^\L \mathrm{Ind}(\cA^{\op})\simeq\Ind(\scr A\tens_{\scr E}\scr A^\op)
\end{gather*}
preserve compact objects. In that case, their restrictions to compact objects are the evaluation and coevaluation of a duality between $\cA$ and $\cA^\op$.
\end{lemma}

\begin{proof}
	 Assume that $\cA$ is dualizable and let $\ev_\cA\colon \cA^\vee\tens_\cE\cA\to\cE$ exhibit $\scr A^\vee$ as dual to $\scr A$. As $\Ind$ is symmetric monoidal, $\Ind(\ev_{\cA})$ exhibits $\Ind(\cA^\vee)$ as dual to $\Ind(\cA)$. By uniqueness of duals \cite[Lemma 4.6.1.10]{HA}, $\Ind(\ev_\cA)$ can be identified with $\ev_{\Ind(\cA)}$, which therefore preserves compact objects. A dual argument shows that $\coev_{\Ind(\cA)}$ preserves compact objects.
Conversely, if $\ev_{\Ind(\cA)}$ and $\coev_{\Ind(\cA)}$ preserve compact objects, then their restrictions to compact objects satisfy the triangle identities and hence exhibit $\cA^\op$ as a dual of $\cA$.
\end{proof}

\begin{proposition}
\label{prop:sat}
Let $\scr E\in\CAlg^\rig(\Cat^\perf)$ and let $\cA\in\Cat^{\perf}(\cE)$. Then $\cA$ is dualizable if and only if it is saturated. In that case, the dual of $\scr A$ is $\scr A^\op$.
\end{proposition}
\begin{proof}
The proposition is a simple corollary of Lemma~\ref{lem:corest}. Indeed, $\ev_{\Ind(\cA)}$ preserves compact objects if and only if $\cA$ is  proper. Similarly, $\coev_{\Ind(\cA)}$ preserves compact objects if and only if $\cA$ is smooth.
\end{proof}

\begin{corollary}\label{cor:Brown}
	Let $\scr E\in\CAlg^\rig(\Cat^\perf)$ and let $\cA\in\Cat^{\sat}(\cE)$. Then the $\scr E$-linear Yoneda embedding $j_{\scr E}$ induces an equivalence
	\[
	\cA \simeq \Fun_\cE^\ex(\cA^{\op}, \cE).
	\]
\end{corollary}

\begin{remark}
	In fact, the following finer statements hold. If $\cA\in\Cat^\perf(\cE)$, both $\cA$ and $\Fun_\cE^\ex(\cA^\op,\cE)$ can be identified with full subcategories of $\Fun_{\cE}^\ex(\cA^\op,\Ind(\cE))$. Then:
	\begin{itemize}
		\item If $\cA$ is proper, $\cA\subset \Fun_\cE^\ex(\cA^\op,\cE)$.
		\item If $\cA$ is smooth, $\Fun_\cE^\ex(\cA^\op,\cE)\subset\cA$.
	\end{itemize}
\end{remark}

\begin{corollary}\label{cor:split}
	Let $\scr E\in\CAlg^\rig(\Cat^\perf)$ and let $\cA,\cB\in\Cat^{\perf}(\cE)$.
If $\cA$ is smooth and $\cB$ is proper, then every $\cE$-linear functor $F\colon \cA \rightarrow \cB$ has an $\cE$-linear right adjoint. 
\end{corollary}
\begin{proof}
The functor $\Ind(F)$ has a right adjoint $G\colon \Ind(\cB) \to \Ind(\cA)$ in $\Cat^\Mor(\scr E)$, so it suffices to show that $G$ preserves compact objects. But $G$ can be written as the composite
\[
\Ind(\cB) \xrightarrow{\coev_{\Ind(\cA)}} \Ind(\cA)\tens^\L_\cE\Ind(\cA^\op)\tens^\L_\cE\Ind(\cB) \xrightarrow{G^\vee} 
\Ind(\cA)\tens^\L_\cE\Ind(\cB^\op)\tens^\L_\cE\Ind(\cB) \xrightarrow{\ev_{\Ind(\cB)}} \Ind(\cA).
\]
Since $\cA$ is smooth and $\cB$ is proper, $\coev_{\Ind(\cA)}$ and $\ev_{\Ind(\cB)}$ preserve compact objects. Finally, $G^\vee=\Ind(F^\op)$ also preserves compact objects.
\end{proof}

\begin{proposition}\label{prop:sat=>compact}
	Let $\scr E\in\CAlg^\rig(\Cat^\perf)$.
If $\cA\in\Cat^{\perf}(\cE)$ is saturated, then $\cA$ is compact.
\end{proposition}

\begin{proof}
		By Proposition~\ref{prop:sat}, there is an equivalence of mapping 
	$\infty$-groupoids
	\[
	\Cat^{\perf}(\cE)(\cA,\ph) \simeq \Cat^{\perf}(\cE)(\cE, \cA^\op \otimes_{\cE} \ph).
	\] 
	As $\cA^\op \otimes_\cE \ph$ preserves colimits, it suffices to show that the unit $\cE$ is compact. The functor $\Cat^{\perf}(\cE)(\cE,\ph)$ is equivalent to the composition
	\[
	\Cat^\perf(\scr E) \to \Cat^\perf \to \Cat_{(\infty,1)} \xrightarrow{\iota_0} \scr S.
	\]
	The first functor preserves all colimits \cite[Corollary 4.2.3.7]{HA}, and $\iota_0$ clearly preserves filtered colimits.
	The forgetful functor $\Cat^\perf\to\Cat_{(\infty,1)}$ also preserves filtered colimits, by \cite[Proposition 1.1.4.6 and Lemma 7.3.5.10]{HA}.
\end{proof}

\begin{corollary}
\label{cor:small}
Let $\scr E\in\CAlg^\rig(\Cat^\perf)$.
Then the $\infty$-category $ \Cat^\sat(\scr E)$ is small.
\end{corollary}

\begin{proof}
	By Proposition~\ref{prop:sat=>compact}, $\Cat^\sat(\scr E)$ is a subcategory of $\Cat^\perf(\scr E)^\omega$, which is small by Proposition~\ref{prop:coco}.
\end{proof}

\subsection{The \texorpdfstring{$(\infty,2)$}{(∞,2)}-categorical structure}
\label{sub:2cat}

To apply the results of \S\ref{sec:traces}–\ref{sec:localization}, we need to upgrade the symmetric monoidal $\infty$-categories
\[
\Cat^\sat(\scr E)\subset \Cat^\perf(\scr E) \subset \Cat^\Mor(\scr E)\subset \Pr^\L(\scr E)
\]
to symmetric monoidal $(\infty,2)$-categories. 
We refer to \cite[Chapter I.1, \S6.1.8]{GR} for the construction of the symmetric monoidal $(\infty,2)$-category of stable cocomplete $\infty$-categories.\footnote{In \emph{loc. cit.}, the authors use an axiomatic approach to $(\infty,1)$-categories. However, if we choose to use complete Segal spaces as a model for $(\infty,1)$-categories, then their definitions of $(\infty,2)$-categories and symmetric monoidal structures coincide with those used in \cite{JFS} and recalled in \S\ref{sub:trace}.}
We let $\PR^\L_\St$ denote the full subcategory of the latter spanned by the \emph{presentable} stable $\infty$-categories.
For $n\in\Delta$, the $(\infty,1)$-category $(\PR^\L_\St)_{n,\bullet}$ is the subcategory of $\widehat{\Cat}_{/(\Delta^n)^\op}$ whose objects are the presentable fibrations with stable fibers and whose morphisms are the fiberwise equivalences.
By construction, we have $\iota_1\PR^\L_\St\simeq \Pr^\L_\St$ as symmetric monoidal $(\infty,1)$-categories, and the $(\infty,1)$-category of morphisms from $\scr A$ to $\scr B$ in $\PR^\L_\St$ is $\Fun^\L(\scr A,\scr B)$. Moreover, the notion of adjunction internal to the $(\infty,2)$-category $\PR^\L_\St$ matches the usual notion of adjunction between functors \cite[Chapter I.1, Lemma 5.3.2]{GR}.

Let $\scr C\in \Cat_{(\infty,1)}^\otimes$ be a symmetric monoidal $\infty$-category compatible with geometric realizations \cite[Definition 3.1.1.18]{HA}. Then, for every commutative algebra $A\in\CAlg(\scr C)$, the $\infty$-category of $A$-modules $\Mod_A(\scr C)$ has a canonical symmetric monoidal structure. In fact, by \cite[Theorem 4.5.3.1]{HA}, there exists a functor
\begin{equation}\label{eqn:ModA1}
\CAlg(\scr C) \to \Cat_{(\infty,1)}^\otimes, \quad A\mapsto \Mod_A(\scr C).
\end{equation}
We will need a $2$-categorical enhancement of this construction. 

More generally, suppose that $\scr C\in \Cat_{(\infty,n)}^\otimes$ is a symmetric monoidal $(\infty,n)$-category whose underlying symmetric monoidal $(\infty,1)$-category $\iota_1\scr C$ is compatible with geometric realizations.
	We construct a functor
	\begin{equation}\label{eqn:ModA2}
	\CAlg(\iota_1\scr C) \to \Cat_{(\infty,n)}^\otimes,\quad A\mapsto \Mod_A(\scr C),
	\end{equation}
	such that $\iota_1\Mod_A(\scr C)\simeq \Mod_A(\iota_1\scr C)$.
	The functor~\eqref{eqn:ModA1}
	is obtained by straightening an explicit cocartesian fibration $\Mod(\scr C)^\otimes\to \CAlg(\scr C)\times \Fin_*$, which is natural in $\scr C$ at the point-set level. In particular, if we plug in the $n$-fold simplicial symmetric monoidal $(\infty,1)$-category
	\[
	\vec k\mapsto \iota_1\Fun(\Theta^{\vec k},\scr C),
	\]
	 and pull back the resulting cocartesian fibrations to the initial object of $(\Delta^\op)^n$,
	 we obtain an $n$-fold simplicial cocartesian fibration over $\CAlg(\iota_1\scr C)\times \Fin_*$. By straightening, this gives rise to a functor
	\[
	\CAlg(\iota_1\scr C) \times (\Delta^\op)^n \to \Cat_{(\infty,1)}^\otimes,\quad (A,\vec k)\mapsto \Mod_A(\iota_1\Fun(\Theta^{\vec k},\scr C)).
	\]
	For fixed $A\in\CAlg(\iota_1\scr C)$, we claim that this is a complete $n$-fold Segal object in symmetric monoidal $(\infty,1)$-categories.
	Since $\Fun(\Theta^{\vec \bullet},\scr C)$ is a complete $n$-fold Segal object, it suffices to show that $\Mod_A(\ph)$ preserves limits of $\iota_1\scr C$-modules, but this follows directly from the definition \cite[Definition 4.2.1.13]{HA}.
	Applying $\iota_0$, we obtain the functor~\eqref{eqn:ModA2}. The identification $\iota_1\Mod_A(\scr C)\simeq \Mod_A(\iota_1\scr C)$ results from 
	\[
	\Mod_A(\iota_1\Fun(\Delta^n,\scr C))=\Mod_A(\Fun(\Delta^n,\iota_1\scr C))\simeq \Fun(\Delta^n,\Mod_A(\iota_1\scr C)).
	\]

We apply this construction with $\scr C=\PR^\L_\St$: given $\scr E\in\CAlg(\Cat^\perf)$, we denote by $\PR^\L(\scr E)$ the symmetric monoidal $(\infty,2)$-category $\Mod_{\Ind(\scr E)}(\PR^\L_\St)$. Thus, $\iota_1\PR^\L(\scr E)\simeq \Pr^\L(\scr E)$.
An equivalent construction of this symmetric monoidal $(\infty,2)$-category can be found in \cite[Chapter I.1, \S8.3]{GR}.
By unraveling the construction, we see that for $\scr A,\scr B\in\PR^\L(\scr E)$, there is an equivalence of $(\infty,1)$-categories
\[
\PR^\L(\scr E)(\scr A,\scr B) \simeq \Fun_{\scr E}^\L(\scr A,\scr B)
\]
compatible with binary composition and tensor product. In particular, $\PR^\L(\scr E)$ is linearly symmetric monoidal in the sense of Definition~\ref{dfn:linear}.
By~\eqref{eqn:ModA2}, we can moreover regard the assignment $\scr E\mapsto\PR^\L(\scr E)$ as a functor $\CAlg(\Cat^\perf)\to\Cat_{(\infty,2)}^\otimes$.

Assume now that $\scr E$ is rigid. Since $\Cat^\sat(\scr E)$, $\Cat^\perf(\scr E)$, and $\Cat^\Mor(\scr E)$ are (or can be identified with) symmetric monoidal subcategories of $\Pr^\L(\scr E)$, they can be upgraded to symmetric monoidal $(\infty,2)$-categories $\CAT^\sat(\scr E)$, $\CAT^\perf(\scr E)$, and $\CAT^\Mor(\scr E)$, namely the corresponding subcategories of $\PR^\L(\scr E)$. We thus have a sequence of linearly symmetric monoidal $(\infty,2)$-categories
\begin{equation*}\label{eqn:CAT}
\CAT^\sat(\scr E)\subset \CAT^\perf(\scr E) \subset \CAT^\Mor(\scr E)\subset \PR^\L(\scr E).
\end{equation*}
For any morphism $\scr E\to\scr F$ in $\CAlg^\rig(\Cat^\perf)$, these subcategories are preserved by the change of scalars functor $\PR^\L(\scr E)\to\PR^\L(\scr F)$, and hence they vary functorially with $\scr E\in\CAlg^\rig(\Cat^\perf)$.

Recall that any symmetric monoidal $(\infty,n)$-category $\scr C$ has a (not necessarily full) subcategory $\scr C^\mathrm{fd}\subset \scr C$ of \emph{fully dualizable} objects \cite[\S2.3]{Lurie}; this is the largest subcategory in which every object is dualizable and every $p$-morphism, for $0<p<n$, has left and right adjoints. The following proposition is a rephrasing of previous results:

\begin{proposition}\label{prop:2cat}
	Let $\scr E\in\CAlg^\rig(\Cat^\perf)$.
	\begin{enumerate}
		\item Every object in $\CAT^\Mor(\scr E)$ is dualizable.
		\item $\CAT^\perf(\scr E)$ is the wide subcategory of $\CAT^\Mor(\scr E)$ on the right dualizable $1$-morphisms.
		\item $\CAT^\sat(\scr E)$ is the full subcategory of dualizable objects in $\CAT^\perf(\scr E)$.
		\item $\CAT^\sat(\scr E)$ is the subcategory of fully dualizable objects in $\CAT^\Mor(\scr E)$.
	\end{enumerate}
\end{proposition}

\begin{proof}
	(1) This is Proposition~\ref{prop:dual}.
	
	(2) We must show that, for every $\scr A,\scr B\in\Cat^\perf(\scr E)$, the functor $\Fun_\cE^\ex(\scr A,\scr B)\to \Fun_\cE^\L(\Ind(\scr A),\Ind(\scr B))$ is fully faithful, and that its image is the subcategory of right dualizable $1$-morphisms.
	By Proposition~\ref{prop:IdemInd}, it is indeed a full embedding whose image consists of those functors that preserve compact objects. The right adjoint of such a functor preserves colimits, and by Proposition~\ref{prop:rigid} (3) it can be promoted to a right adjoint in the $(\infty,2)$-category $\CAT^\Mor(\scr E)$.
	
	(3) This is Proposition~\ref{prop:sat}.
	
	(4) This follows from (1)–(3) and \cite[Proposition 4.2.3]{Lurie}.
\end{proof}

\subsection{Hochschild homology as a trace}
\label{sub:hochschild}

Let $\scr E\in\CAlg^\rig(\Cat^\perf)$.
We recall how the trace
\[
\Tr\colon \End(\CAT^\Mor(\scr E)) \to \Omega\CAT^\Mor(\scr E)\simeq \Ind(\scr E)
\]
can be identified with Hochschild homology (relative to $\Ind(\scr E)$). By Corollary~\ref{cor:bimod}, endomorphisms of $\Ind(\scr A)$ in $\CAT^\Mor(\scr E)$ are $\scr A$-bimodules. This leads to the following informal description of the $(\infty,1)$-category $\End(\CAT^\Mor(\scr E))$:
\begin{itemize}
	\item An object of $\End(\CAT^\Mor(\scr E))$ is a pair $(\scr A,\scr M)$ where $\scr A\in\Cat^\perf(\scr E)$ and $\scr M\colon\scr A^\op\times\scr A\to\Ind(\scr E)$ is an $\scr A$-bimodule.
	\item A morphism $(\scr A,\scr M)\to(\scr B,\scr N)$ in $\End(\CAT^\Mor(\scr E))$ is an $\scr E$-linear functor $\phi\colon\scr A\to\scr B$ together with a morphism of $\scr A$-bimodules $\scr M\to \phi^*(\scr N)$.
\end{itemize}

Let us recall the standard definition of the Hochschild homology of the pair $(\scr A,\scr M)$. To do so we must choose a set $S$ of objects of $\scr A$. Define a simplicial object $C_\bullet(S,\scr M)$ in $\Ind(\scr E)$ by
\[
C_n(S,\scr M) = \bigvee_{a_0,\dots,a_n\in S} \scr A^\cE(a_n,a_{n-1})\tens\dotsb\tens \scr A^\cE(a_1,a_0)\tens \scr M(a_0,a_n) \in\Ind(\scr E).
\]

\begin{proposition}\label{prop:hochschild}
	Let $\scr E\in\CAlg^\rig(\Cat^\perf)$.
	Let $(\scr A,\scr M)\in\End(\CAT^\Mor(\scr E))$ and let $S$ be a set of objects of $\scr A$ meeting all equivalence classes.
	Then there is an equivalence
	\[
	\Tr(\scr A,\scr M) \simeq \colim C_\bullet(S,\scr M)
	\]
	in $\Ind(\scr E)$, natural in $(\scr A, S,\scr M)$.
\end{proposition}

\begin{proof}
	The trace $\Tr(\scr A,\scr M)$ is the image of the $\scr A$-bimodule $\scr M$ by the evaluation map $\Ind(\scr A)^\vee\tens_{\scr E}\Ind(\scr A)\to\Ind(\scr E)$.
	We must therefore identify $\scr M\mapsto\colim C_\bullet(S,\scr M)$ with the evaluation. By duality, it suffices to show that the composition
	\[
	\Ind(\scr A) \xrightarrow{\coev\tens\id} \Ind(\scr A)\tens^\L_{\scr E}\Ind(\scr A)^\vee\tens^\L_{\scr E}\Ind(\scr A) \xrightarrow{\id\tens\colim C_\bullet(S,\ph)} \Ind(\scr A)
	\]
	is naturally equivalent to the identity.
	The coevaluation map $\Ind(\scr E)\to \Ind(\scr A)\tens_{\scr E}\Ind(\scr A)^\vee$ sends $\mathbf{1}$ to the $\scr A$-bimodule $(x,y)\mapsto \scr A^\cE(x,y)\in\Ind(\scr E)$.
	Hence, the above composition sends a right $\scr A$-module $\scr N$ to the right $\scr A$-module
	\[
	x\mapsto \colim C_\bullet(S,\scr A^\cE(x,\ph)\tens\scr N(\ph)).
	\]
	There is an augmented simplicial object
	\[
	C_\bullet(S,\scr A^\cE(x,\ph)\tens\scr N(\ph)) \to \scr N(x),
	\]
	natural in $\scr N$ and $x$. If $x\in S$, this augmented simplicial object has an extra degeneracy sending $C_{n-1}$ to the summand of $C_n$ where $a_n=x$. Since $S$ meets every equivalence class in $\scr A$, this completes the proof.
\end{proof}

\begin{corollary}\label{cor:filteredcolimits}
	The functor
	\[
	\Fun(S^1,\Cat^\perf(\scr E)) \simeq \Aut(\CAT^\Mor(\scr E)) \xrightarrow{\Tr} \Ind(\scr E)
	\]
	preserves filtered colimits.
\end{corollary}

\begin{proof}
	Let $(\scr A_\alpha,f_\alpha)_{\alpha\in \scr I}$ be a filtered diagram in $\Fun(S^1,\Cat^\perf(\scr E))$. Without loss of generality, we can assume that $\scr I$ is a filtered poset. Then we can find a compatible diagram of sets of objects $S_\alpha\subset\scr A_\alpha$ meeting all equivalence classes. Using the formula of Proposition~\ref{prop:hochschild}, we deduce that $\Tr$ preserves the colimit of $(\scr A_\alpha,f_\alpha)_{\alpha\in\scr I}$.
\end{proof}

\section{Categories of $\cE$-motives}
\label{E-motives}

In this section, we extend the main results from \cite[Sections 6--9]{BGT} to $\cE$-linear $\infty$-categories. 
We assume throughout that $\scr E\in\CAlg(\Cat^\perf)$ is rigid.
We define the $\infty$-categories of additive and localizing motives for $\Cat^\perf(\cE)$, which are the recipients of the universal additive and localizing invariants on $\Cat^\perf(\scr E)$, respectively. Their construction is reminiscent of that of the Grothendieck group of an exact category. We then show that connective and nonconnective $K$-theory are corepresentable in these $\infty$-categories.
We follow \cite{BGT} closely but we include for the reader's convenience complete arguments or precise references. The constructions in \cite{BGT} are recovered when $\cE = \Sp^{\omega}$ is the symmetric monoidal $\infty$-category of compact spectra. 

\subsection{Exact sequences}

\begin{definition}[{\cite[Definition 5.12 and Proposition 5.13]{BGT}}] 
 A sequence 
$
\cA \stackrel{f}{\rightarrow} \cB \stackrel{g}{\rightarrow} \cC
$
in $\Cat^{\perf}$ is \emph{exact} if it is a cofiber sequence and $f$ is fully faithful.
\end{definition}

Note that ``being a cofiber sequence'' is a meaningful property, since the $\infty$-groupoid of equivalences $g\circ f\simeq 0$ is either empty or contractible. Whether a given sequence $\scr A\to\scr B\to\scr C$ is exact can be checked at the level of triangulated homotopy categories: it is exact if and only if $\h\scr A\to\h\scr B$ is fully faithful and $\h\scr C$ is the idempotent completion of the Verdier quotient $\h\scr B/\h\scr A$ \cite[Proposition 5.15]{BGT}.

\begin{definition}[{\cite[Definition 5.18]{BGT}}]  
A sequence 
$\cA \stackrel{f}{\rightarrow} \cB \stackrel{g}{\rightarrow} \cC$
in $\Cat^\perf$ is \emph{split exact} if it is exact and if $f$ and $g$ admit right adjoints.
\end{definition}

\begin{definition}\label{dfn:exact2}
A sequence
\[
\cA \stackrel f\to \cB \stackrel g\to \cC
\]
in $\Cat^\perf(\scr E)$ is called \emph{exact} (resp.\ \emph{split exact}) if its image by the forgetful functor $\Cat^\perf(\scr E)\to\Cat^\perf$ is exact (resp.\ split exact). 
\end{definition}

\begin{proposition}\label{prop:exact}
	A sequence in $\Cat^\perf(\scr E)$ is exact (resp.\ split exact) in the sense of Definition~\ref{dfn:exact2} if and only if it is a localization sequence in the $(\infty,2)$-category $\CAT^\Mor(\scr E)$ (resp.\ in the $(\infty,2)$-category $\CAT^\perf(\scr E)$), in the sense of Definition~\ref{dfn:exact}.
\end{proposition}

\begin{proof}
	By Proposition~\ref{prop:rigid} (3), the forgetful functors $\CAT^\Mor(\scr E)\to \CAT^\Mor$ and $\CAT^\perf(\scr E)\to\CAT^\perf$ reflect localization sequences. Hence, we are reduced to the case $\scr E=\Sp^\omega$.
	The statement for split exact sequences follows immediately from the statement for exact sequence.
	For the latter, since $\Ind\colon\Cat^\perf\to \Pr^\L_\St$ preserves colimits, it suffices to prove the following: given a commutative square
\begin{tikzmath}
	\diagram{\cA & \cB \\ 0 & \cC \\};
	\arrows (11-) edge node[above]{$f$} (-12) (12) edge node[right]{$g$} (22) (11) edge (21) (21-) edge (-22);
\end{tikzmath}
in $\Pr^\L_\St$ with $f$ fully faithful, it is a pushout square if and only if $g^r$ is fully faithful and the null sequence
\[
ff^r \to \id_{\scr B} \to g^rg
\]
is a cofiber sequence. By \cite[Theorem 5.5.3.18]{HTT}, the above square is a pushout if and only if $g^r$ is fully faithful with essential image $(f^r)^{-1}(0)$. The result is now straightforward.
\end{proof}

\begin{proposition}\label{prop:exactsequences}
	\leavevmode
	\begin{enumerate}
		\item Every exact sequence $\scr A\to \scr B\to\scr C$ in $\Cat^\perf(\scr E)$ is a filtered colimit of exact sequences $\scr A_\alpha\to\scr B_\alpha\to\scr C_\alpha$ where each $\scr B_\alpha$ is compact.
		\item Every split exact sequence $\scr A\to \scr B\to\scr C$ in $\Cat^\perf(\scr E)$ is a filtered colimit of split exact sequences $\scr A_\alpha\to\scr B_\alpha\to\scr C_\alpha$ where $\scr A_\alpha$, $\scr B_\alpha$, and $\scr C_\alpha$ are compact.
	\end{enumerate}
\end{proposition}

\begin{proof}
	(1) Recall that every object in $\ccat{perf}{\cE}$ is a filtered colimit of compact objects. Let $\cI$ be a filtered 
$\infty$-category and let $\cB_\alpha$, $\alpha \in \cI$, be a filtered diagram of compact objects having colimit $\cB$. We set $\cA_{\alpha}:= \cA \times_{\cB} \cB_\alpha$. Then the projection
$\cA_{\alpha} \rightarrow \cB_\alpha$ is fully faithful, and we let $\cC_\alpha$ be its cofiber. Then $\cA_{\alpha} \rightarrow \cB_\alpha \rightarrow \cC_\alpha$ is a 
filtered diagram of exact sequences whose colimit is $\cA \rightarrow \cB \rightarrow \cC$, and each $\scr B_\alpha$ is compact, as desired.

(2) Let $\scr G$ be the subcategory of $\Fun(\Delta^2,\Cat^\perf(\scr E))$ whose objects are split exact sequences and whose morphisms are morphisms of sequences forming right adjointable squares \cite[Definition 7.3.1.2]{HTT}. The three evaluation functors
\[
\ev_0,\ev_1,\ev_2\colon \scr G\to \Cat^\perf(\scr E)
\]
have right adjoints sending $\scr A$ to $\scr A\xrightarrow\id\scr A\to 0$,
\begin{tikzmath}
	\diagram{\scr A & \Fun(\Delta^1,\scr A) & \scr A\rlap,\\};
	\arrows (11-) edge[vshift=\dbl] (-12) edge[<-,vshift=-\dbl] node[below=\dbl]{$\ev_1$} (-12)
	(12-) edge[vshift=\dbl] node[above=\dbl]{$\ev_0$} (-13) edge[<-,vshift=-\dbl] (-13);
\end{tikzmath}
and $0\to\scr A\xrightarrow\id\scr A$, respectively. As these right adjoints preserve filtered colimits, the above evaluation functors preserve compact objects. It will therefore suffice to show that $\scr G$ is compactly generated. 

Consider the $\infty$-category $\scr H$ of triples $(\scr A,\scr C,h)$, where $\scr A,\scr C\in\Cat^\perf(\scr E)$ and $h\colon\Ind(\scr A)\to\Ind(\scr C)$ is a colimit-preserving $\scr E$-linear functor. 
In other words, $\scr H$ is the $\infty$-category of cartesian fibrations $\scr X\to\Delta^1$ with fiber-preserving action of $\scr E$, such that $\scr X_0$ and $\scr X_1$ are compactly generated and the pullback functor $e^*\colon\scr X_1\to\scr X_0$ preserves colimits: such a fibration encodes the triple $(\scr X_1^\omega,\scr X_0^\omega,e^*)$.
Since $\Cat^\perf(\scr E)$ and $\Fun^\L_\scr E(\Ind(\scr A),\Ind(\scr C))\simeq\Ind(\scr A^\op\tens_{\scr E}\scr C)$ are compactly generated, $\scr H$ is also compactly generated: an object $(\scr A,\scr B,h)\in\scr H$ is compact if and only if $\scr A$, $\scr C$, and $h$ are compact in their respective $\infty$-categories.
Consider the functor $\phi\colon\scr G\to\scr H$ sending the sequence \[\scr A\stackrel f\to \scr B\stackrel g\to \scr C\] to the triple $(\scr A,\scr C,g^{rr}\circ f)$. We claim that $\phi$ is an equivalence, which will conclude the proof.
When $\scr E=\Sp^\omega$, $\phi$ is a $\Cat^\perf$-module functor, and the general case is obtained from this case by passing to the $\infty$-categories of $\scr E$-modules. Hence, we may assume that $\scr E=\Sp^\omega$.
Given any split exact sequence as above, we observe that $\Ind(\scr B)$ is a \emph{recollement} of the subcategories $g^r(\Ind(\scr C))$ and $f^{rr}(\Ind(\scr A))$ in the sense of \cite[Definition A.8.1]{HA}. In fact, using the notation of \cite[Remark A.8.19]{HA}, a split exact sequence in $\Cat^\perf$ is the same thing as a stable compactly generated $\infty$-category $\scr D$ equipped with stable subcategories $i_*\colon\scr D_0\into\scr D$ and $j_*\colon\scr D_1\into\scr D$ forming a recollement of $\scr D$, with the additional condition that $i_*$ and $j^*$ preserve compact objects. Since the pair of localization functors $(i^*,j^*)$ is conservative, this additional condition is equivalent to $i^!\circ j_!$ being colimit-preserving.
The fact that $\phi$ is an equivalence now follows from \cite[Remark A.8.18]{HA}. Explicitly,
the inverse functor $\scr H\to\scr G$ sends a cartesian fibration $p\colon\scr X\to\Delta^1$ to the split exact sequence
$\scr X_1^\omega\to \Gamma(p)^\omega \to \scr X_0^\omega$.
\end{proof}

\subsection{Additive $\cE$-motives}\label{sub:add}
\begin{definition}
Let $\scr C$ be a small $(\infty,1)$-category.
We denote by
\begin{itemize} 
	\item $\Pre(\scr C) := \Fun(\scr C^\op,\cS)$, the $\infty$-category of presheaves on $\scr C$.
	\item $\Pre_{\Sp}(\scr C) := \Fun(\scr C^\op,\Sp)$, the $\infty$-category of presheaves of spectra on $\scr C$.
	\item $\Sigma^\infty_+\colon \Pre(\scr C) \to \Pre_{\Sp}(\scr C)$ the stabilization functor, given objectwise by $\Sigma^\infty_+\colon \cS \to \Sp$.
\end{itemize}
\end{definition}

\begin{definition}
We denote by 
\[\psi\colon \Cat^\perf(\cE)\to \Pre_{\Sp}(\Cat^\perf(\cE)^{\omega})\]
the filtered-colimit-preserving extension of the composition
\[
 \Cat^\perf(\cE)^\omega \xrightarrow{j} \Pre(\Cat^\perf(\cE)^\omega) \xrightarrow{\Sigma^\infty_+}
\Pre_{\Sp}(\Cat^\perf(\cE)^\omega),
\] 
where $j$ is the Yoneda embedding.
\end{definition}

Recall that an exact sequence in $\Cat^\perf(\scr E)$ is equipped with a canonical nullhomotopy, since $\CAT^\perf(\scr E)$ is a linear $(\infty,2)$-category.
Let $S_{\add}$ be the class of morphisms in $\Pre_{\Sp}(\Cat^\perf(\cE)^{\omega})$ of the form 
\begin{gather*}
0\to \Sigma^n\psi(0), \\
\Sigma^{n}(\psi(\cB)/\psi(\cA)) \rightarrow \Sigma^{n}\psi(\cC),
\end{gather*}
where $\cA \rightarrow \cB \rightarrow \cC$ is a split exact sequence in $\ccat{perf}{\cE}$ and $n\leq 0$.

\begin{definition}
\label{def:mot-add} 
The \emph{$\infty$-category of additive $\cE$-motives} is the full subcategory of 
$
\Pre_{\Sp}(\Cat^\perf(\cE)^{\omega})
$ 
spanned by the $S_\add$-local objects, in the sense of \cite[Definition 5.5.4.1]{HTT}.
We denote it by $\Mot(\cE)$.
\end{definition}

A priori, the definition of  $\Mot(\cE)$ involves localizing with respect to a proper class of morphisms. However, we now show that 
there exists a small set of morphisms $S'_{\add}$ that generates the same strongly saturated class as $S_{\add}$. In particular, $\Mot(\cE)$ is the full subcategory of $\Pre_{\Sp}(\Cat^\perf(\cE)^{\omega})$ consisting of the $S'_{\add}$-local objects.

\begin{proposition}
\label{prop:motstab}
The $\infty$-category
$\Mot(\cE)$ is an exact $\omega$-accessible localization of $\Pre_{\Sp}(\Cat^\perf(\cE)^{\omega})$. In particular, $\Mot(\cE)$ is a stable compactly generated $\infty$-category.
\end{proposition} 
\begin{proof}
Let $S_\add'\subset S_\add$ be the small subset consisting of the maps $0\to\Sigma^n\psi(0)$ and $\Sigma^n(\psi(\cA)/\psi(\cB))\to\Sigma^n\psi(\cC)$, where $\cA\to\cB\to\cC$ is a split exact sequence in $\Cat^\perf(\scr E)^\omega$ and $n\leq 0$. By Proposition~\ref{prop:exactsequences} (2),
every element of $S_\add$ is a filtered colimit of elements of $S'_\add$ in $\Fun(\Delta^1,\Cat^\perf(\scr E))$. In particular $S'_\add$ and $S_\add$ generate the same strongly saturated class of morphisms. Applying \cite[Proposition 5.5.4.15]{HTT}, we deduce that $\Mot(\cE)$ is an accessible localization of $\Pre_{\Sp}(\Cat^\perf(\cE)^{\omega})$. By definition of $S_\add$, it is clear that $\Mot(\scr E)$ is closed under suspension, and hence it is stable by \cite[Proposition 1.4.2.11]{HA}. 
Finally, note that $S_\add'$-local presheaves are stable under filtered colimits. Since $\Pre_{\Sp}(\Cat^\perf(\cE)^{\omega})$ is compactly generated, it follows that $\Mot(\scr E)$ is compactly generated.
\end{proof}

\begin{remark}\label{rmk:Sadd}
	By definition of $S_\add'$, a presheaf of spectra $F\colon \Cat^\perf(\cE)^{\omega,\op}\to \Sp$ belongs to $\Mot(\cE)$ if and only if it preserves zero objects and carries split exact sequences in $\Cat^\perf(\cE)^\omega$ to fiber sequences of spectra.
\end{remark}

Thus, the inclusion $\Mot(\cE)\subset \Pre_{\Sp}(\Cat^\perf(\cE)^{\omega})$ admits an exact left adjoint, and we denote by $\cU_{\add}$ the composition
\[
\cU_{\add}\colon \Cat^\perf(\cE) \stackrel{\psi}{\rightarrow} \Pre_{\Sp}(\Cat^\perf(\cE)^{\omega}) \rightarrow \Mot(\cE). 
\]
Note that $\cU_\add$ preserves compact objects.

\begin{definition}
\label{def:add}
Let $\cD$ be a stable presentable  $\infty$-category and let $F\colon \Cat^\perf(\cE) \rightarrow \cD$ be a functor. We say that $F$ is an \emph{additive invariant} if the following conditions are satisfied:
\begin{enumerate}
\item $F$ preserves filtered colimits.
\item $F$ preserves zero objects.
\item $F$ sends split exact sequences in $\ccat{perf}{\cE}$ to cofiber sequences in 
$\cD$.\footnote{Note that, as a consequence, $F$ sends split exact sequences to split cofiber sequences.}
\end{enumerate}
We denote by
$\mathrm{Fun_{\add}}(\Cat^\perf(\cE), \cD)$ the 
$\infty$-category of additive invariants with values in $\cD$.
\end{definition}

\begin{theorem}\label{thm:add}
 The functor $\cU_{\add}\colon \Cat^\perf(\cE) \rightarrow \Mot(\cE)$ is the \emph{universal additive invariant}. More precisely, for any presentable stable $\infty$-category $\cD$, $\cU_\add$ induces an equivalence of $\infty$-categories
\[
\mathrm{Fun^L}(\Mot(\cE), \cD) \simeq \mathrm{Fun_{\add}}(\Cat^\perf(\cE), \cD).
\] 
\end{theorem}

\begin{proof}
Note first that $\cU_{\add}$ is an additive invariant: condition 
(1) is satisfied because $\psi$ preserves filtered colimits, and conditions (2) and (3) are satisfied by definition of $S_\add$.
 Next, observe that if $\cD$ is a stable presentable $\infty$-category, the functor $\psi$ induces an equivalence
\[
\mathrm{Fun^L}(\Pre_{\Sp}(\Cat^\perf(\cE)^{\omega}), \cD) \simeq \Fun^\mathrm{flt}(\Cat^\perf(\cE), \cD),
\]
where an object in the right-hand side is a functor that preserves filtered colimits.
The claim now follows from the universal property of the localization $\Mot(\cE)$. 
\end{proof}

We now briefly discuss symmetric monoidal structures (see also \cite{CT} for a different treatment in the language of derivators).
It follows immediately from the definition of $\tens_{\scr E}$ that, if $\scr A,\scr B\in \Cat^\perf(\scr E)$ are compact, so is $\scr A\tens_{\scr E}\scr B$. The $\infty$-category of presheaves $\Pre(\Cat^\perf(\scr E)^\omega)$ therefore acquires a presentably symmetric monoidal structure given by Day convolution, such that the Yoneda embedding 
\[j\colon \Cat^\perf(\scr E)^\omega\into \Pre(\Cat^\perf(\scr E)^\omega)\]
becomes universal among symmetric monoidal functors to presentably symmetric monoidal $\infty$-categories \cite[Proposition 4.8.1.10]{HA}.
The stabilization functor
\[
\Sigma^\infty_+\colon \Pre(\Cat^\perf(\scr E)^\omega) \to \Pre_{\Sp}(\Cat^\perf(\scr E)^\omega)
\]
can also be promoted to a symmetric monoidal functor with an obvious universal property \cite[Remark 2.25]{Robalo}. Note moreover that $\ph\tens_{\scr E}\scr A$ preserves split exact sequences, for any $\scr A\in\Cat^\perf(\scr E)$ (see \cite[Lemma 5.5]{BGT2}). This implies, by \cite[Proposition 2.2.1.9]{HA}, that $\Mot(\scr E)$ acquires a symmetric monoidal structure such that the localization functor $\Pre_{\Sp}(\Cat^\perf(\scr E)^\omega)\to \Mot(\scr E)$ is symmetric monoidal and has a universal property as such.
Combining these universal properties, we obtain:

\begin{theorem}\label{thm:monoidaladd}
 The symmetric monoidal functor $\cU_{\add}\colon \Cat^\perf(\cE) \rightarrow \Mot(\cE)$ is the \emph{universal symmetric monoidal additive invariant}. More precisely, for any presentably symmetric monoidal stable $\infty$-category $\cD$, $\cU_\add$ induces an equivalence of $\infty$-categories
\[
\Fun^{\L,\tens}(\Mot(\cE), \cD) \simeq \Fun_{\add}^\tens(\Cat^\perf(\cE), \cD).
\] 
\end{theorem}

Finally, we discuss the functoriality in $\scr E$ of $\Mot(\cE)$. Suppose that $\scr E,\scr F\in\CAlg(\Cat^\perf)$ are rigid and that $f\colon\scr E\to \scr F$ is an exact symmetric monoidal functor. Then $f$ induces a symmetric monoidal base change functor $f^*\colon \Cat^\perf(\scr E)\to\Cat^\perf(\scr F)$ between $\infty$-categories of modules \cite[\S4.5.3]{HA}, with a colimit-preserving right adjoint $f_*$. We claim that the composite functor
\[
\Cat^\perf(\scr E)\xrightarrow{f^*}\Cat^\perf(\scr F) \xrightarrow{\cU_\add} \Mot(\scr F)
\]
is an additive invariant. Conditions (1) and (2) are clear, and condition (3) follows from the fact that $f^*$ preserves (split) exact sequences (see \cite[Lemma 5.5]{BGT2}). Thus, it induces a symmetric monoidal colimit-preserving functor
\[
f^*\colon \Mot(\scr E)\to \Mot(\scr F).
\]
In fact, the construction $\cE\mapsto \Mot(\cE)$ can be promoted to a functor
\[
\Mot(\ph)\colon \CAlg^\rig(\Cat^\perf) \to \CAlg(\Pr^\L_\St).
\]
Indeed, by the universal property of the symmetric monoidal functor $\Cat^\perf(\cE)^\omega \to \Mot(\cE)$ from Theorem~\ref{thm:monoidaladd}, this follows from the functoriality of $\cE\mapsto\Cat^\perf(\cE)$ (see \eqref{eqn:ModA1}), and the fact that objectwise solutions to a universal problem automatically determine a functor \cite[Proposition 5.2.4.2]{HTT}.

\subsection{Localizing $\cE$-motives} 
\label{sub:loc}
Let $S_{\loc}$ be the class of morphisms in $\Pre_{\Sp}(\Cat^\perf(\scr E)^{\omega})$ of the form 
\begin{gather*}
0\to\Sigma^n\psi(0),\\
\Sigma^n(\psi(\cB)/\psi(\cA)) \rightarrow \Sigma^n\psi(\cC),
\end{gather*}
where $\cA \rightarrow \cB \rightarrow \cC$ is an exact sequence in $\ccat{perf}{\cE}$ and $n\leq 0$.

\begin{definition}
\label{def:mot-loc} 
The \emph{$\infty$-category of localizing $\cE$-motives} is the full subcategory of
$\Pre_{\Sp}(\Cat^\perf(\scr E)^{\omega})$
spanned by the $S_\loc$-local objects. We denote it by $\mathbb M\mathrm{ot}(\cE)$.
\end{definition}

Note that $\mathbb M\mathrm{ot}(\cE)\subset \Mot(\cE)$.
As for $S_\add$, we show that there exists a (small) set of morphisms 
$S'_{\loc}$ that generates $S_{\loc}$ under filtered colimits.

\begin{proposition}\label{prop:Motstab}
	The $\infty$-category
	$\MOT(\cE)$ is an exact accessible localization of $\Pre_{\Sp}(\Cat^\perf(\cE)^{\omega})$. In particular, $\MOT(\cE)$ is a stable presentable $\infty$-category.
\end{proposition}
\begin{proof}
	Let $S_\loc'\subset S_\loc$ denote the subclass consisting of the maps $0\to\Sigma^n\psi(0)$ and $\Sigma^n(\psi(\cA)/\psi(\cB))\to\Sigma^n\psi(\cC)$, where $\cA\to\cB\to\cC$ is an exact sequence with $\scr B\in \Cat^\perf(\scr E)^\omega$ and $n\leq 0$. By Proposition~\ref{prop:exactsequences} (1),
every element of $S_\loc$ is a filtered colimit of elements of $S'_\loc$ in $\Fun(\Delta^1,\Cat^\perf(\scr E))$. In particular $S'_\loc$ and $S_\loc$ generate the same strongly saturated class of morphisms.
Note that $S_\loc'$ is essentially small, since $\Cat^\perf(\scr E)^\omega$ is small and the collection of full subcategories of a given small $\infty$-category is small.
 Applying \cite[Proposition 5.5.4.15]{HTT}, we deduce that $\MOT(\cE)$ is an accessible localization of $\Pre_{\Sp}(\Cat^\perf(\cE)^{\omega})$. By definition of $S_\loc$, $\MOT(\scr E)$ is closed under suspension, and hence it is stable by \cite[Proposition 1.4.2.11]{HA}.
\end{proof}

Thus, the inclusion $\MOT(\cE)\subset \Pre_{\Sp}(\Cat^\perf(\cE)^{\omega})$ admits an exact left adjoint, and we denote by $\cU_{\loc}$ the composition
\[
\cU_{\loc}\colon \Cat^\perf(\cE) \stackrel{\psi}{\rightarrow} \Pre_{\Sp}(\Cat^\perf(\cE)^{\omega}) \rightarrow\MOT(\cE). 
\]

\begin{definition}
\label{def:loc}
Let $\cD$ be a stable presentable  $\infty$-category and let $F\colon \Cat^\perf(\cE) \rightarrow \cD$ be a functor. We say that $F$ is a \emph{localizing invariant} if the following conditions are satisfied:
\begin{enumerate}
\item $F$ preserves filtered colimits.
\item $F$ preserves zero objects.
\item $F$ sends exact sequences in $\Cat^\perf(\cE)$ to cofiber sequences in 
$\cD$.
\end{enumerate}
We denote by
$\Fun_{\mathrm{loc}}(\Cat^\perf(\cE), \cD)$ the 
$\infty$-category of localizing invariants with values in $\cD$.
\end{definition}

\begin{theorem}
\label{thm:localizing}
 The functor $\cU_{\loc}\colon \Cat^\perf(\cE) \rightarrow \MOT(\cE)$ is the \emph{universal localizing invariant}. More precisely, for any presentable stable $\infty$-category $\cD$, $\cU_\loc$ induces an equivalence of $\infty$-categories
\[
\mathrm{Fun^L}(\mathbb M\mathrm{ot}(\cE), \cD) \simeq \mathrm{Fun_{\loc}}(\Cat^\perf(\cE), \cD).
\] 
\end{theorem}
\begin{proof}
Since $\MOT(\cE)$ is the full subcategory of $\cU_\add(S_\loc)$-local objects in $\Mot(\cE)$,
the claim follows from Theorem~\ref{thm:add} and the universal property of localization. 
\end{proof}

Noting that $\ph\tens_{\scr E}\scr A$ preserves exact sequences, we deduce from \cite[Proposition 2.2.1.9]{HA} and Theorem~\ref{thm:monoidaladd} that $\MOT(\cE)$ acquires a symmetric monoidal structure with the following universal property:

\begin{theorem}\label{thm:monoidalloc}
 The symmetric monoidal functor $\cU_{\loc}\colon \Cat^\perf(\cE) \rightarrow \MOT(\cE)$ is the \emph{universal symmetric monoidal localizing invariant}. More precisely, for any presentably symmetric monoidal stable $\infty$-category $\cD$, $\cU_\loc$ induces an equivalence of $\infty$-categories
\[
\Fun^{\L,\tens}(\MOT(\cE), \cD) \simeq \Fun_{\loc}^\tens(\Cat^\perf(\cE), \cD).
\] 
\end{theorem}

If $\scr E,\scr F\in\CAlg(\Cat^\perf)$ are rigid and $f\colon \scr E\to\scr F$ is an exact symmetric monoidal functor, the composition
\[
\Cat^\perf(\scr E)\xrightarrow{f^*}\Cat^\perf(\scr F) \xrightarrow{\cU_\loc} \MOT(\scr F)
\]
is a symmetric monoidal localizing invariant and hence induces a symmetric monoidal colimit-preserving functor
\[
f^*\colon \MOT(\scr E)\to \MOT(\scr F).
\]
Its right adjoint $f_*$ is simply the restriction of $f_*\colon \Mot(\scr F)\to \Mot(\scr E)$ to the full subcategory $\MOT(\scr F)$.
As in \S\ref{sub:add}, Theorem~\ref{thm:monoidalloc} implies that this construction can be promoted to a functor
\[
\MOT(\ph)\colon \CAlg^\rig(\Cat^\perf) \to \CAlg(\Pr^\L_\St).
\]

\subsection{Corepresentability of $K$-theory}
\label{sec:corepresentability}
In this subsection, we prove that the connective (resp.\ nonconnective) $K$-theory of objects in $\ccat{perf}{\cE}$ is corepresented by the unit in the symmetric monoidal stable $\infty$-category $\Mot(\cE)$ (resp.\ $\MOT(\cE)$).
These are direct generalizations of \cite[Theorems 7.13 and 9.36]{BGT}.
In fact, for nonconnective $K$-theory, we can easily deduce the corepresentability result from the case $\cE=\Sp^\omega$.
On the other hand, to get the full corepresentability result for connective $K$-theory, we have to repeat the arguments from \cite{BGT}, but we introduce a simplification based on \cite{B}. We start by recalling the definition of the $K$-theory of $\infty$-categories. 

The $\infty$-categorical version of the $S_\bullet$-construction was introduced in \cite[Definition 1.2.2.2]{HA}.  
Let $\Cat^\mathrm{rex}_*$ be the $(\infty, 1)$-category of small pointed $\infty$-categories admitting finite colimits and right exact functors between them.

For all $n\in\Delta$, let $\Ar[n]$ be the category such that: 
\begin{itemize}
\item its objects are pairs $(i,j)$ where $0 \le i \le j \le n$,
\item there is exactly one morphism 
$(i,j) \to (k,l)$ if $i \le k \le j\le l$, 
and none otherwise.
\end{itemize}
In other words, $\Ar[n]$ is the arrow category of the poset $[n]$.
For $\cA\in \Cat^\mathrm{rex}_*$, denote by $S_n(\cA)$ the 
full subcategory of $\Fun(\Ar[n],\cA)$ spanned by the functors $F\colon \Ar[n] \rightarrow \cA$ such that:
\begin{itemize} 
\item for all $i$, $F(i,i)$ is a zero object in $\cA$,
\item if $i \le j \le k$ then $F(i,j) \to F(i,k) \to F(j,k)$ is a cofiber sequence in $\cA$.
\end{itemize}
For all $n$, $S_n(\cA)$ is again an object in $\Cat^\mathrm{rex}_*$. 
Moreover, the $\infty$-categories $S_n(\cA)$ assemble into a simplicial $\infty$-category $S_\bullet(\cA) \in \Fun(\Delta^\op,\Cat^\mathrm{rex}_*)$.

Let $\iota_0S_\bullet(\cA)$ be the simplicial pointed space obtained by 
taking the maximal 
subgroupoids of $S_\bullet(\cA)$ levelwise, and letting the initial objects be the base points. Denote by $\lvert\iota_0S_\bullet(\cA)\rvert\in\cS$  its colimit. 

\begin{definition}
	Let $\cA\in\Cat^\mathrm{rex}_*$. The space $\Omega\lvert\iota_0S_\bullet(\cA)\rvert$ is the \emph{$K$-theory space} of $\cA$. 
\end{definition}

As in the case of ordinary Waldhausen categories, the \emph{$K$-theory spectrum}  of $\cA$, denoted by $K(\cA)$, can be defined by iterating the $S_\bullet$-construction: more precisely, the $n$th space of the spectrum is $K(\cA)_n = \lvert\iota_0S_\bullet^n(\cA)\rvert$ for $n\geq 1$, see \cite[Section 7.1]{BGT}.

If $\cA\in\ccat{perf}{\cE}$, then $S_n \cA$ is stable, idempotent complete, and tensored over $\cE$, and $S_\bullet$ lifts to a functor
\[
S_\bullet\colon \Cat^\perf(\cE) \to \Fun(\Delta^\op,\Cat^\perf(\cE)).
\]

\begin{lemma} 
\label{lem:S_1}
Let $\cA,\cB\in\ccat{perf}{\cE}$. There is a natural equivalence 
\[ 
S_\bullet (\Fun_\cE^\ex(\cA,\cB)) \simeq \Fun_\cE^\ex(\cA,S_\bullet(\cB))
\]
in $\Fun(\Delta^\op,\Cat^\perf(\cE))$.
\end{lemma}

\begin{proof}
For all $m \in \Delta$, we have a natural equivalence of simplicial $\infty$-categories
\[S_\bullet(\Fun^\ex(\cA\tens\cE^{\tens m},\cB)) \simeq \mathrm{Fun}^\ex( \cA\tens\cE^{\tens m}, S_\bullet(\cB)),\]
by  \cite[Lemma 7.16]{BGT}.
Taking the limit over $m\in\Delta$ and noting that $S_n$ preserves limits concludes the proof.
 \end{proof}

For $\cA\in\ccat{perf}{\cE}$, denote by $K_\cA$ the presheaf of spectra on 
$\Cat^\perf(\cE)^\omega$ defined by
\[K_\cA(\cB) = K(\Fun_\cE^\ex(\cB,\cA)).\]

\begin{lemma}\label{lem:KAlocal}
	Let $\cA\in\Cat^\perf(\cE)$. Then $K_\cA$ belongs to $\Mot(\cE)$.
\end{lemma}

\begin{proof}
	By Remark~\ref{rmk:Sadd}, we must prove that $K_\cA(0)\simeq 0$ and that $K_\cA$ takes split exact sequences in $\Cat^\perf(\cE)^\omega$ to fiber sequences.
	The former is clear, since $\Fun_{\cE}^\ex(0,\cA)\simeq 0$. Let $\cB_1\to\cB_2\to\cB_3$ be a split exact sequence in $\Cat^\perf(\cE)^\omega$. Then the sequence 
	\[\Fun_\cE^\ex(\cB_3, \cA) \to \Fun_\cE^\ex(\cB_2, \cA) \to \Fun_\cE^\ex(\cB_1, \cA)\] 
	is split exact. By \cite[Proposition 10.12]{B},
	applying $K$ to a split exact sequence in $\Cat^\perf$ yields a fiber sequence of $K$-theory spectra, as desired.
\end{proof}

\begin{lemma}
\label{lem:S_2}
For any $\cA\in\Cat^\perf(\cE)$, there is a natural equivalence $K_\cA\simeq \cU_\add(\cA)$ in $\Mot(\cE)$.
 \end{lemma}
\begin{proof} 
As in the proof of \cite[Proposition 7.17]{BGT}, there is a canonical levelwise split exact sequence in $\Fun(\Delta^\op, \Cat^\perf(\cE))$, 
\[
\cA_{\bullet} \to P S_\bullet \cA \to S_\bullet \cA,
\]
where $\cA_\bullet$ is the constant simplicial object with value $\cA$ and $P(\ph)$ denotes the simplicial path object. By applying $\cU_{\add}$ and taking realizations, we obtain a cofiber sequence in $\Mot(\cE)$, 
\[ 
\lvert\cU_\add(\cA_{\bullet})\rvert\simeq\cU_{\add}(\cA) \rightarrow \lvert\cU_{\add}(P S_\bullet \cA)\rvert \simeq \cU_\add(0) \rightarrow \lvert\cU_{\add}(S_\bullet \cA)\rvert.
\]
Thus, since $\cU_\add(0)\simeq 0$, $\Sigma \cU_{\add}(\cA) \simeq \lvert\cU_{\add}(S_\bullet \cA)\rvert$ in $\Mot(\cE)$.
Iterating, we find
\begin{equation}\label{eqn:SigmaSdot}
	\Sigma^n\cU_{\add}(\cA) \simeq \lvert\cU_{\add}(S_\bullet^n \cA)\rvert,
\end{equation}
for any $n\geq 1$.
Next, consider the following equivalences of pointed presheaves on $\Cat^\perf(\cE)^\omega$, where $n\geq 1$ and $j(S^n_\bullet \cA)$ is pointed by $j(0)$:
\begin{align*}
\Omega^{\infty-n} K_\cA & \simeq  
\lvert \iota_0S^n_\bullet(\Fun_\cE^\ex (\ph,\cA)) \rvert
\\ & \simeq   \lvert\iota_0\Fun_\cE^\ex(\ph, S^n_\bullet \cA)\rvert 
 \\ & =  \lvert j(S^n_\bullet \cA)\rvert.
\end{align*}
The first equivalence follows from the definition of $K$-theory, and the second one from Lemma \ref{lem:S_1}.
Writing $K_\cA$ as a colimit of desuspensions of its constituent spaces, we get
\[
K_{\cA} \simeq  \colim_n\Sigma^{-n}\Sigma^\infty\lvert j(S^n_\bullet \cA)\rvert\simeq \colim_n\Sigma^{-n}\lvert \psi(S^n_\bullet \cA)/\psi(0)\rvert.
\]
Hence, by Lemma \ref{lem:KAlocal} and~\eqref{eqn:SigmaSdot}, we have the following equivalences in $\Mot(\cE)$:
\begin{align*}
	K_{\cA} & \simeq\colim_n\Sigma^{-n}\lvert \cU_\add(S^n_\bullet \cA)/\cU_\add(0)\rvert\\
	&\simeq \colim_n\Sigma^{-n}\lvert\cU_\add(S^n_\bullet \cA)\rvert\\
	&\simeq \colim_n\Sigma^{-n}\Sigma^n\cU_\add(\cA)\\
	&\simeq \colim_n\cU_\add(\cA).
\end{align*}
Comparing the construction of~\eqref{eqn:SigmaSdot} with that of the structure maps of the $K$-theory spectrum, we see that this last colimit is constant and hence that $K_\cA\simeq \cU_\add(\cA)$, as desired.
 \end{proof}

\begin{theorem}[Corepresentability of connective $K$-theory]
\label{thm:representability}
Let $\cA,\cB\in\Cat^\perf(\scr E)$ and assume that $\cB$ is compact. Then there is a natural equivalence of spectra
\[ 
\Mot(\cE)( \cU_{\add}(\cB), \cU_{\add}(\cA)) \simeq K(\Fun_{\scr E}^\ex(\cB,\cA)).
\]
\end{theorem}

\begin{proof}
Since $\cB$ is compact, the presheaf $\psi(\cB)$ is representable.
We then have a sequence of natural equivalences of spectra:
\begin{align*}
	\Mot(\cE)(\cU_{\add}(\cB), \cU_{\add}(\cA)) &\simeq \Pre_{\Sp}(\Cat^\perf(\cE)^\omega)(\psi(\cB), \cU_\add(\cA))\\
	&\simeq \Pre_{\Sp}(\Cat^\perf(\cE)^\omega)(\psi(\cB), K_\cA)\\
	&\simeq K_\cA(\cB)=K(\Fun_\cE^\ex(\cB,\cA)).
\end{align*}
The first holds by adjunction, the second by Lemma~\ref{lem:S_2}, and the third by the spectral Yoneda lemma.
\end{proof}

\begin{theorem}[Corepresentability of nonconnective $K$-theory]
	\label{thm:nc-representability}
Let $\cA,\cB\in\Cat^\perf(\cE)$ and assume that $\cB$ is saturated. Then there is a natural equivalence of spectra
\[ \mathbb M\mathrm{ot}(\cE)(\cU_{\loc}(\cB), \cU_{\loc}( \cA) ) \simeq \mathbb{K}(\cB^\op\otimes_\cE \cA).\]
\end{theorem} 

\begin{proof}
By Proposition~\ref{prop:sat}, $\scr B$ is dualizable in $\Cat^\perf(\scr E)$ with dual $\cB^\op$.
Let $f\colon\scr S_\infty^\omega\to\scr E$ be the unique exact symmetric monoidal functor.
Since $\cU_\loc$ is symmetric monoidal, we have
\begin{align*}
	\mathbb M\mathrm{ot}(\cE) (\cU_{\loc} (\cB)  , \cU_{\loc}(\cA)) & \simeq \MOT(\cE)(\cU_\loc(\cE), \cU_\loc(\cB^\op)\tens\cU_\loc(\cA))\\
	&\simeq \MOT(\cE)(\cU_\loc(\cE), \cU_\loc(\cB^\op\tens_{\scr E}\cA))\\
	&\simeq \MOT(\cE)(f^*\cU_\loc(\Sp^\omega), \cU_\loc(\cB^\op\tens_{\scr E}\cA))\\
	&\simeq \MOT(\Sp^\omega)(\cU_\loc(\Sp^\omega), f_*\cU_\loc(\cB^\op\tens_{\scr E}\cA))\\
	&\simeq \MOT(\Sp^\omega)(\cU_\loc(\Sp^\omega), \cU_\loc(\cB^\op\tens_{\scr E}\cA))\\
	&\simeq \K(\cB^\op\otimes_\cE \cA),
\end{align*}
where the last equivalence is \cite[Theorem 9.8]{BGT}.
\end{proof}

\section{The categorified Chern character}
\label{sec:chern}

\subsection{Tannakian geometry}
\label{sub:Tannakian}

Our construction of the categorified Chern character naturally takes place in a slightly more general context than that of spectral geometry, which we call \emph{Tannakian geometry}. It is an $\infty$-categorical version of Balmer's tensor triangular geometry \cite{Balmer}. This generality is merely a convenient way to streamline some of the proofs, and the reader should feel free to replace all occurrences of ``Tannakian'' by ``spectral'' or ``derived'', or to ignore the word altogether.

 We define the $\infty$-category $\Aff^\Tan$ of \emph{Tannakian affine schemes} to be the opposite of the $\infty$-category $\CAlg^\rig(\Cat^\perf)$ of symmetric monoidal $\infty$-categories that are small, stable, idempotent complete, and rigid. Given such a symmetric monoidal $\infty$-category $\scr E$, we denote by $\Spec\scr E\in \Aff^\Tan$ the corresponding Tannakian affine scheme, and we denote by $\Aff^\Tan_{\scr E}$ the overcategory $(\Aff^\Tan)_{/\Spec\scr E}$.
This homotopical algebro-geometric context relates to the usual ones via forgetful functors
\[
\Aff^\mathrm{der} \to \Aff^\mathrm{sp} \into \Aff^\mathrm{nc} \into \Aff^\Tan,
\]
where $(\Aff^\mathrm{der})^\op$, $(\Aff^\mathrm{sp})^\op$, and $(\Aff^\mathrm{nc})^\op$ are the $\infty$-categories of simplicial commutative rings, connective $E_\infty$-ring spectra, and arbitrary $E_\infty$-ring spectra, respectively. The last functor sends an $E_\infty$-ring spectrum $R$ to the symmetric monoidal $\infty$-category $\Perf(R)$ of perfect $R$-modules, and it is fully faithful by \cite[Proposition 3.2.9]{DAG8}.

Given $X=\Spec\scr E\in\Aff^\Tan$, we write $\scr O(X)$ for the $E_\infty$-ring spectrum of endomorphisms of the unit in $\scr E$, $\Perf(X)$ for the symmetric monoidal $\infty$-category $\scr E$ itself, and $\QCoh(X)$ for $\Ind(\Perf(X))$. We also write $\Cat^{\sat,\perf,\Mor}(X)$ for $\Cat^{\sat,\perf,\Mor}(\scr E)$ and similarly for their $(\infty,2)$-categorical enhancements.

A \emph{Tannakian prestack} is a presheaf of $\infty$-groupoids on $\Aff^\Tan$ which is a small colimit of representables. We denote by $\PrSt^\Tan$ the $\infty$-category of Tannakian prestacks and by $\PrSt^\Tan_{\scr E}$ the overcategory $(\PrSt^\Tan)_{/\Spec\scr E}$. The above functors on $\Aff^\Tan$ extend uniquely to limit-preserving functors
 \begin{gather*}
	\scr O\colon \PrSt^{\Tan,\op} \to \CAlg(\Sp),\\
	\Perf\subset\QCoh\colon \PrSt^{\Tan,\op} \to \Cat_{(\infty,1)}^\tens,\\
	\CAT^\mathrm{sat}\subset \CAT^{\perf} \subset \CAT^\Mor\colon \PrSt^{\Tan,\op} \to \Cat_{(\infty,2)}^\tens.
 \end{gather*}
For any $X\in\PrSt^\Tan$, $\Perf(X)$ is the full subcategory of dualizable objects in $\QCoh(X)$, since an object in a limit of symmetric monoidal $\infty$-categories is dualizable if and only if each of its components is dualizable. Similarly, $\CAT^\sat(X)$ is the full subcategory of dualizable objects in $\CAT^\perf(X)$, which is in turn the wide subcategory of right dualizable morphisms in $\CAT^\Mor(X)$, which is itself rigid. In particular, as in Proposition~\ref{prop:2cat} (4), we have $\CAT^\sat(X)\simeq\CAT^\Mor(X)^\mathrm{fd}$.
Note also that $\Perf(X)$ and $\CAT^\sat(X)$ are small, being small limits of small $(\infty,2)$-categories.
Furthermore, we have the following relations between these functors, where $\Omega_\Sp$ denotes the \emph{spectrum} of endomorphisms of the unit in a stable symmetric monoidal $\infty$-category:
 \begin{gather*}
 	\Omega_\Sp\circ\Perf=\Omega_\Sp\circ\QCoh\simeq\scr O,\\
	\Omega\circ\CAT^\sat=\Omega\circ\CAT^\perf\simeq\Perf,\quad \Omega\circ\CAT^\Mor\simeq \QCoh.
 \end{gather*}
Indeed, since $\Omega$ and $\Omega_\Sp$ preserve limits, all of these functors are right Kan extensions of their restrictions to Tannakian affine schemes, where the given equivalences are clear.

The inclusion $\Aff^\mathrm{nc}\into \Aff^\Tan$ induces, by left Kan extension, a fully faithful embedding of (nonconnective) spectral prestacks over an $E_\infty$-ring spectrum $R$ into Tannakian prestacks over $\Perf(R)$. If $X$ is a spectral prestack viewed as a Tannakian prestack in this way, then $\scr O(X)$, $\Perf(X)$, and $\QCoh(X)$ have their usual meanings, and $\CAT^\Mor(X)$ is the full subcategory of the $(\infty,2)$-category of quasi-coherent sheaves of $\infty$-categories on $X$, in the sense of \cite{Ga1}, consisting of those sheaves whose $\infty$-categories of sections over any spectral affine scheme are compactly generated.

\begin{definition}
	Let $X$ be a Tannakian prestack. A sequence $\cA\to\cB\to\cC$ in $\Cat^\perf(X)$ is called \emph{exact} (resp.\ \emph{split exact}) if it is a localization sequence in $\CAT^\Mor(X)$ (resp.\ in $\CAT^\perf(X)$) in the sense of Definition~\ref{dfn:exact}.
\end{definition}

When $X$ is affine, this definition agrees with Definition~\ref{dfn:exact2}, by Proposition~\ref{prop:exact}. In general, a sequence in $\Cat^\perf(X)$ is exact (resp.\ split exact) if and only if its pullback to any affine is exact (resp.\ split exact).

\begin{definition}
	Let $X$ be a Tannakian prestack. The nonconnective $K$-theory of $X$, denoted by $\K(X)$, is the nonconnective $K$-theory of the symmetric monoidal stable $\infty$-category $\Perf(X)$.
\end{definition}

Thus, $\K(X)$ is an $E_\infty$-ring spectrum. Note that the affinization map $X\to \Spec\Perf(X)$ induces an equivalence on $\Perf$ and hence on $\K$.

\subsection{Chern characters}
\label{sub:chern}

From now on we fix a Tannakian affine base scheme $\Spec\scr E\in\Aff^\Tan$.

\begin{definition}
	Let $X$ be a Tannakian prestack over $\scr E$.
	The \emph{free loop space} $\scr LX$ of $X$ over $\scr E$ is the Tannakian prestack over $\scr E$ defined by
	\[
	\scr LX = X^{S^1}\simeq\lim_{S^1}X\simeq X\times_{X\times_{\scr E}X}X.
	\]
\end{definition}

Note that $\scr LX$ is a relative construction and depends on the base $\scr E$, although we choose not to indicate it in the notation. If $X=\Spec\scr C$ is affine, then $\scr LX=\Spec(S^1\tens_{\scr E}\scr C)$, where $\tens_{\scr E}$ denotes the canonical action of $\scr S$ on $\Aff^{\Tan,\op}_{\scr E}$:
\[
S^1\tens_{\scr E}\scr C = \colim_{S^1}\scr C\simeq \scr C\tens_{\scr C\tens_{\scr E}\scr C}\scr C.
\]
When $\scr E=\Perf(R)$ for some $E_\infty$-ring spectrum $R$, the free loop space construction does not commute with the inclusion of spectral prestacks over $R$ into Tannakian prestacks over $\scr E$. However, there is always a canonical map from the free loop space of a spectral prestack $X$ to its free loop space as a Tannakian prestack, and it is an equivalence when $X$ is affine.

Let $X$ be a Tannakian prestack over $\scr E$.
We briefly recall the construction of the Chern pre-character $\ch^\pre\colon \iota_0\Perf(X)\to \scr O(\scr LX)^{hS^1}$ in this context, following Toën and Vezzosi \cite[\S4.2]{TV2}. Given a perfect complex over $X$, its pullback to the free loop space $\scr LX$ is equipped with a canonical monodromy automorphism $m$, whose trace $\Tr(m)\in \Omega^\infty\scr O(\scr LX)$ has a canonical $S^1$-invariant refinement in $\Omega^\infty\scr O(\scr LX)^{hS^1}$. This construction defines a map of $E_\infty$-semirings
\begin{equation}\label{eqn:TV1}
\ch^{\pre}\colon \iota_0\Perf(X) \to \Omega^\infty\scr O(\scr LX)^{hS^1},
\end{equation}
natural in $X$, called the \emph{Chern pre-character} of the Tannakian prestack $X$.

\begin{theorem}\label{thm:additivity}
	Let $X$ be a Tannakian prestack over $\scr E$. The Toën–Vezzosi Chern pre-character $\ch^\pre$ descends to $K$-theory and deloops to a morphism of $E_\infty$-ring spectra
	\[
	\ch\colon\K(X) \to \scr O(\scr LX)^{hS^1},
	\]
	natural in $X$.
\end{theorem}

We will deduce this theorem from its categorified version, Theorem~\ref{thm:main}, below.

\begin{remark}
	Theorem~\ref{thm:additivity} is already interesting when $X=\Spec\scr E$: in that case, it states that the Euler characteristic $\chi\colon \iota_0\scr E\to\Omega\scr E$ induces a morphism of $E_\infty$-ring spectra from the nonconnective $K$-theory of $\scr E$ to the spectrum $\Omega_\Sp\scr E$ of endomorphisms of $\mathbf{1}\in\scr E$, and even to the mapping spectrum $\Hom(\Sigma^\infty_+ BS^1,\Omega_\Sp\scr E)$. This is a refinement of May's additivity theorem for Euler characteristics \cite[Theorem 0.1]{May}.
\end{remark}

\begin{remark}
	We emphasize that Theorem~\ref{thm:additivity} is not a formal consequence of the universal property of $K$-theory \cite{BGT,B}. Indeed, the Chern character is only defined on \emph{rigid symmetric monoidal} stable $\infty$-categories and not on arbitrary stable $\infty$-categories, as would be required to invoke the universal property of $K$-theory. We will see in Remark~\ref{rmk:Dennistrace} that the Chern character of Theorem~\ref{thm:additivity} is nevertheless an instance of the Dennis trace map, but this is a corollary of our main result.
\end{remark}

The above construction of $\ch^{\pre}$ is clearly very general: in \cite{TV2}, a Chern pre-character is constructed for any sheaf of rigid symmetric monoidal $\infty$-categories on an $\infty$-topos, the ``classical'' case recalled above corresponding to the tautological presheaf $\Perf$ on $\Aff^\Tan$. Toën and Vezzosi also consider in \cite[\S4.3]{TV2} a categorified version of the classical case, where $\Perf$ is replaced by $\Cat^\Mor$. Their construction then yields a morphism of $E_\infty$-semirings
\begin{equation}\label{eqn:TV2}
\ch^\pre\colon\iota_0\Cat^\perf(X) \to \iota_0\QCoh^{S^1}(\scr LX),
\end{equation}
natural in $X$, where $\QCoh^{S^1}(\scr LX)=\QCoh(\scr LX)^{hS^1}$ is the $\infty$-category of $S^1$-equivariant quasi-coherent sheaves on the free loop space of $X$.
Using our higher-categorical generalization of the $S^1$-invariant trace from \S\ref{sub:S1trace}, we can easily upgrade this morphism of $E_\infty$-semirings to a symmetric monoidal $(\infty,1)$-functor:

\begin{theorem}\label{thm:main}
	Let $X$ be a Tannakian prestack over $\scr E$. The Toën–Vezzosi categorified Chern pre-character lifts to a symmetric monoidal $(\infty,1)$-functor
	\[
	\ch^\pre\colon \Cat^\perf(X) \to \QCoh^{S^1}(\scr LX),
	\]
	natural in $X$, which preserves zero objects and sends exact sequences to cofiber sequences.
\end{theorem}

\begin{proof}
	To construct $\ch^\pre$, we repeat the construction of Toën and Vezzosi, using the symmetric monoidal $S^1$-invariant trace functor constructed in \S\ref{sub:S1trace}.
	The obvious functor
	\[\colim_{S^1}\Cat^\perf(X)\to \Cat^\perf(\scr LX)\]
	induces by adjunction a functor
	\[
	\Cat^\perf(X) \to \Fun(S^1,\Cat^\perf(\scr LX)),
	\]
	sending a quasi-coherent sheaf of $\infty$-categories on $X$ to its pullback to $\cL X$ equipped with a canonical monodromy automorphism.
	This functor is manifestly $S^1$-equivariant for the trivial action on the source and the diagonal action on the target.
	We then consider the symmetric monoidal composition
	\begin{equation}\label{eqn:ch}
	\Cat^\perf(X) \to \Fun(S^1,\Cat^\perf(\scr LX)) \simeq \Aut(\CAT^\Mor(\scr LX)) \xrightarrow{\Tr_\otimes} \Omega\CAT^\Mor(\scr LX)\simeq \QCoh(\scr LX).
	\end{equation}
	Since the trace $\Tr_\otimes$ is a natural transformation $\Aut \to \Omega$ which is $S^1$-invariant for the action of $S^1$ on $\Aut$ (Theorem~\ref{thm:S1trace}), we deduce that the whole composition is $S^1$-equivariant, and hence yields the desired symmetric monoidal functor $\ch^\pre$.
	The first part of~\eqref{eqn:ch} sends exact sequences in $\Cat^\perf(X)$ to localization sequences in $\Aut(\CAT^\Mor(\scr LX))$, since any morphism in $\Aut(\CAT^\Mor(\scr LX))$ is right adjointable. The fact that $\ch^\pre$ sends exact sequences to cofiber sequences then follows from Theorem~\ref{thm:localization} and the fact that the forgetful functor $\QCoh^{S^1}(\scr LX)\to\QCoh(\scr LX)$ reflects colimits.
\end{proof}

\begin{corollary}\label{cor:affine}
	Let $\Spec\scr C$ be a Tannakian affine scheme over $\scr E$. Then the symmetric monoidal functor $\ch^\pre\colon \Cat^\perf(\scr C) \to \Ind(S^1\tens_{\scr E}\scr C)^{hS^1}$ is a localizing invariant, and hence it factors uniquely through the symmetric monoidal $\infty$-category of localizing $\scr C$-motives:
	\begin{tikzmath}
		\diagram{
		\Cat^\perf(\scr C) & \Ind(S^1\tens_{\scr E}\scr C)^{hS^1}\rlap. \\
		\MOT(\scr C) & \\
		};
		\arrows (11-) edge node[above]{$\ch^\pre$} (-12) (11) edge (21) (21) edge[dashed] node[below right]{$\ch$} (12);
	\end{tikzmath}
\end{corollary}

\begin{proof}
By Theorem~\ref{thm:main}, the only thing to check is that $\ch^\pre$ preserves filtered colimits. This follows at once from its definition~\eqref{eqn:ch} and Corollary~\ref{cor:filteredcolimits}.
\end{proof}

\begin{proof}[Proof of Theorem~\ref{thm:additivity}]
	Since the affinization map $X\to \Spec\Perf(X)$ is an equivalence on $\K$, we can assume that $X=\Spec\scr C$.
	Applying $\Omega_\Sp$ to the symmetric monoidal functor $\ch\colon \MOT(\scr C) \to \Ind(S^1\tens_{\scr E}\scr C)^{hS^1}$ of Corollary~\ref{cor:affine} and using Theorem~\ref{thm:nc-representability}, we get a morphism of $E_\infty$-ring spectra $\ch\colon \K(\scr C) \to \scr O(\scr LX)^{hS^1}$. It is easy to check that this morphism refines the Chern pre-character in the desired way.
\end{proof}

\begin{remark}\label{rmk:Dennistrace}
	When $\scr C=\scr E$, the categorified Chern character of Corollary~\ref{cor:affine} is a colimit-preserving symmetric monoidal functor $\ch\colon \MOT(\scr E)\to \Fun(BS^1,\Ind(\scr E))$. By construction, it sends $\scr A\in\Cat^\perf(\scr E)$ to the Hochschild homology of $\scr A$ over $\scr E$ with the $S^1$-action of Remark~\ref{rmk:S1action},
	hence gives rise to morphisms of spectra
	\[\K(\scr B^\op\tens_{\scr E}\scr A) \to {\Ind(\scr E)}(\HH(\scr B/\scr E),\HH(\scr A/\scr E))^{hS^1}\]
	for all $\scr A\in\Cat^\perf(\scr E)$ and $\scr B\in\Cat^\sat(\scr E)$.
	According to \cite[Example 4.2.2]{Lurie}, this $S^1$-action on $\HH(\scr A/\scr E)$ coincides with the usual $S^1$-action coming from the extension of the simplicial bar construction $C_\bullet(S,\scr A)$ of \S\ref{sub:hochschild} to a cyclic object.
	When $\scr B=\scr E=\Sp^\omega$, it follows from \cite[Theorem 10.6]{BGT} that the map $\K(\scr A)\to \HH(\scr A)$ is the classical Dennis trace map. Hence, Corollary~\ref{cor:affine} recovers the classical factorization of the Dennis trace map through the homotopy $S^1$-fixed points of topological Hochschild homology.
\end{remark}

\subsection{Secondary $K$-theory and the secondary Chern character}
\label{sub:K2}

In the remainder of this section, we will use the categorified Chern character to construct a Chern character for Toën's secondary $K$-theory.

Let $\scr C$ be a \emph{small} linear $(\infty,2)$-category (see Definition~\ref{dfn:linear}).
We denote by $\Pre_{\Sp}^\mathrm{loc}(\scr C)$ the $\infty$-category of presheaves of spectra on $\iota_1\scr C$ that send initial objects\footnote{In a linear $(\infty,2)$-category $\scr C$, an object $X$ is initial iff it is final iff $\scr C(X,X)=0$.} to $0$ and localization sequences to fiber sequences. As the inclusion $\Pre_{\Sp}^\mathrm{loc}(\scr C)\subset \Pre_{\Sp}(\scr C)$ preserves filtered colimits, $\Pre_{\Sp}^\mathrm{loc}(\scr C)$ is compactly generated. If $\scr C$ is moreover linearly symmetric monoidal, tensoring with a fixed object preserves localization sequences. It follows from \cite[Proposition 2.2.1.9]{HA} that the Day convolution symmetric monoidal structure on $\Pre_{\Sp}(\scr C)$ (see \cite{Saul}) descends to a symmetric monoidal structure on $\Pre_{\Sp}^\mathrm{loc}(\scr C)$ and that the localization functor is symmetric monoidal. In particular, we get a symmetric monoidal functor
\[
\iota_1\scr C \to \Pre_{\Sp}^\mathrm{loc}(\scr C)^\omega.
\]
For example, the definition of the symmetric monoidal $\infty$-category $\Mot(\cE)$ can be summarized as
\[
\Mot(\cE)=\Pre_{\Sp}^\mathrm{loc}(\CAT^\perf(\cE)^\omega).
\]

For $X$ a Tannakian prestack, we define
\[
\Mot^\sat(X) = \Pre_{\Sp}^\mathrm{loc}(\CAT^\sat(X))^\omega.
\]
Thus, $\Mot^\sat(X)$ is a small, stable, and idempotent complete symmetric monoidal $\infty$-category. It is moreover rigid since it is generated under finite colimits and retracts by the image of $\Cat^\sat(X)$.

\begin{remark}\label{rmk:split}
By definition, a localization sequence in $\CAT^\sat(X)$ is a split exact sequence. However, since $\CAT^\sat(X)=\CAT^\Mor(X)^\mathrm{fd}$, any exact sequence in $\CAT^\sat(X)$ is automatically split exact.
\end{remark}

\begin{remark}\label{rmk:kontsevich}
	The proof of Lemma~\ref{lem:S_2} can be repeated to show that the image in $\Mot^\sat(\scr E)$ of a saturated $\cE$-category $\cA$ is the presheaf of spectra $\cB\mapsto K(\Fun_\cE^\ex(\cB,\cA))$: the only thing to note is that all the terms in the split exact sequence $\cA\to PS_n\cA\to S_n\cA$ are saturated.
	Hence, for any $\scr A,\scr B\in\Cat^\sat(\scr E)$, the spectrum of maps from $\scr B$ to $\scr A$ in $\Mot^\sat(\scr E)$ is the $K$-theory spectrum $K(\scr B^\op\tens_{\scr E}\scr A)$. It follows that the symmetric monoidal functor $\Mot^\sat(\scr E)\to\Mot(\scr E)^\omega$ induced by the inclusion $\Cat^\sat(\scr E)\subset\Cat^\perf(\scr E)^\omega$ is fully faithful. When $\scr E=\Perf(k)$, this shows that $\Mot^\sat(\scr E)$ is an $\infty$-categorical enhancement of Kontsevich's triangulated category of noncommutative mixed motives over $k$ (see \cite{Kontsevich} or \cite[\S8.2]{CT}). 
	In general, we may therefore think of $\Mot^\sat(X)$ as an $\infty$-category of noncommutative mixed motives parametrized by the Tannakian prestack $X$.
\end{remark}

\begin{remark}
	Unlike for $K$-theory, the affinization morphism $X\to\Spec\Perf(X)$ may not induce an equivalence on $\Mot^\sat$. It does so however whenever $X$ is $1$-affine in the sense of Gaitsgory \cite{Ga1}, which includes many nonaffine cases, such as quasi-compact quasi-separated schemes or classifying stacks $BG$, where $G$ is a linearly reductive linear algebraic group.
\end{remark}

\begin{corollary}\label{cor:main}
	Let $X$ be a Tannakian prestack over $\scr E$. The restriction of $\ch^\pre$ to $\Cat^\sat(X)$ induces an exact symmetric monoidal functor $\ch\colon \Mot^\sat(X)\to \Perf^{S^1}(\scr LX)$:
	\begin{tikzmath}
		\diagram{
		\Cat^\sat(X) & \Perf^{S^1}(\scr LX)\rlap. \\
		\Mot^\sat(X) & \\
		};
		\arrows (11-) edge node[above]{$\ch^\pre$} (-12) (11) edge (21) (21) edge[dashed] node[below right]{$\ch$} (12);
	\end{tikzmath}
\end{corollary}

\begin{proof}
	Since $\ch^\pre\colon \Cat^\perf(X) \to \QCoh^{S^1}(\scr LX)$ is symmetric monoidal, it preserves dualizable objects and hence restricts to a functor $\Cat^\sat(X) \to \Perf^{S^1}(\scr LX)$. Since this functor has a stable target and sends exact sequences to cofiber sequences, it factors uniquely through $\Mot^\sat(X)$, by the definition of the latter and the universal property of the Day convolution.
\end{proof}

\begin{definition}\label{dfn:K2}
	Let $X$ be a Tannakian prestack. The \emph{nonconnective secondary $K$-theory} of $X$, denoted by $\K^{(2)}(X)$, is the nonconnective $K$-theory of the symmetric monoidal stable $\infty$-category $\Mot^\sat(X)$.
\end{definition}

\begin{remark}
	When $X$ is the spectrum of a commutative ring, Definition \ref{dfn:K2} is closely related to Toën's definition of secondary $K$-theory in \cite[\S5.4]{ToenLectures}. Toën considers the Waldhausen $\infty$-category structure on $\Cat^\sat(X)$ where the cofibrations are the fully faithful $1$-morphisms; let us denote its $K$-theory spectrum by $K^{(2)}_\text{To\"en}(X)$ (it is what we called $K^{(2)}(X)$ in \S\ref{sub:K2intro}). The stable $\infty$-category $\Mot^\sat(X)$ defined above is in a precise sense the closest approximation of this Waldhausen $\infty$-category by a stable idempotent complete $\infty$-category. In particular, we have a symmetric monoidal Waldhausen functor $\Cat^\sat(X)\to\Mot^\sat(X)$,  and hence a morphism of $E_\infty$-ring spectra \[K^{(2)}_\text{To\"en}(X)\to \K^{(2)}(X).\]
	It seems plausible that $K^{(2)}_\text{To\"en}(X)$ is in fact the connective cover of $\K^{(2)}(X)$.
	There are several reasons for using $\Mot^\sat(X)$ instead of $\Cat^\sat(X)$ in our definition of nonconnective secondary $K$-theory: first, it gives a more natural source for the secondary Chern character, and it makes Theorem~\ref{thrm:TV2} below stronger; second, it allows us to relate secondary $K$-theory to iterated $K$-theory, see Remark~\ref{rmk:KK}; and finally, there is at this time no construction of the nonconnective $K$-theory of a Waldhausen $\infty$-category in the literature.
\end{remark}

Combining Theorem~\ref{thm:additivity} and Corollary~\ref{cor:main}, we obtain the following commutative diagram of $E_\infty$-semirings and $E_\infty$-ring spectra, natural in $X\in\PrSt_{\scr E}^\Tan$, where a dotted arrow means a map to the infinite loop space of the target:
\begin{tikzequation}\label{eqn:ch2}
	\diagram{
	\iota_0\Cat^\sat(X) & \iota_0\Perf^{S^1}(\scr LX) & \scr O(\scr L^2X)^{h(S^1\times S^1)}\rlap. \\
	\iota_0\Mot^\sat(X) & \K^{S^1}(\scr LX) & \\
	\K^{(2)}(X) & & \\
	};
	\arrows (11-) edge node[above]{$\iota_0\ch^\pre$} (-12) (12-) edge[dotted] node[above]{$\ch^\pre$} (-13)
	(11) edge (21) (21) edge node[below right]{$\iota_0\ch$} (12)
	(21) edge[dotted] (31) (22) edge node[below right]{$\ch$} (13)
	(12) edge[dotted] (22)
	(31) edge node[below right]{$\K(\ch)$} (22)
	;
\end{tikzequation}
Here, $ \K^{S^1}(\scr LX) =\K(\Perf^{S^1}(\scr LX))$ is the $S^1$-equivariant $K$-theory of the free loop stack $\scr LX$.
The diagonal composition is the \emph{secondary Chern character}
\[
\ch^{(2)}\colon \K^{(2)}(X) \to \scr O(\scr L^2X)^{h(S^1\times S^1)}.
\]
It is thus a morphism of $E_\infty$-rings, natural in $X\in\PrSt_{\scr E}^\Tan$.
The top row of~\eqref{eqn:ch2} is exactly the secondary Chern pre-character constructed by Toën and Vezzosi in \cite[\S4.4]{TV2}. In particular, we have proved the following:

\begin{theorem}
\label{thrm:TV2}
The Toën–Vezzosi secondary Chern pre-character $\iota_0\Cat^\sat(X)\to \cO(\cL^{2} X)^{h(S^1 \times S^1)}$ is refined by a morphism of $E_\infty$-ring spectra
\[
\ch^{(2)}\colon\K^{(2)}(X) \to \cO(\cL^{2} X)^{h(S^1 \times S^1)},
\]
natural in $X\in\PrSt_{\scr E}^\Tan$.
\end{theorem}

\begin{remark}
	The top row of~\eqref{eqn:ch2} sends $\scr A\in\iota_0\Cat^\sat(X)$ to the secondary trace $\Tr^{(2)}$ of the canonical pair of commuting automorphisms of the pullback of $\scr A$ to the double free loop space $\scr L^2X$. It follows from the $2$-dimensional cobordism hypothesis that this map is in fact invariant for the action of the framed diffeomorphism group of the torus on $\scr O(\scr L^2X)$.
	It is natural to ask whether the secondary Chern character is also invariant for this action. Unfortunately, this seems difficult to answer from our construction, which crucially depends on decomposing the torus as a product of circles. On the other hand, $\ch^{(2)}$ admits an asymmetrical refinement, namely the composition
	\[
	\K^{(2)}(X) \xrightarrow{\K(\ch)} \K^{S^1}(\scr LX)=\K(\Perf^{S^1}(\scr LX)) \xrightarrow{\ch} \Omega_{\Sp}(S^1\tens_{\scr E}\Perf^{S^1}(\scr LX))^{hS^1}.
	\]
\end{remark}

\begin{remark}
	If $\scr F(X)$ is a symmetric monoidal $(\infty,n)$-category varying functorially with $X$, the $n$-fold trace $\Tr^{(n)}$ similarly defines an $n$-ary Chern pre-character
	\[
	\iota_0\scr F(X)^\mathrm{fd} \to \Omega^n\scr F(\scr L^nX).
	\]
	By the $n$-dimensional cobordism hypothesis, this map is invariant under the action of the framed diffeomorphism group of the $n$-dimensional torus on $\Omega^n\scr F(\scr L^nX)$.
	For $n\geq 3$, however, we do not know what $n$-ary $K$-theory is.
\end{remark}

\begin{remark}\label{rmk:KK}
	Let $R$ be an $E_\infty$-ring spectrum. The computation of the mapping spaces in $\Mot^\sat(R)$ from Remark~\ref{rmk:kontsevich} shows that the $\infty$-category $\Mod^\mathrm{free}(K(R))$ of finite free $K(R)$-modules is a full symmetric monoidal subcategory of $\Mot^\sat(R)$. Since $\Mot^\sat(R)$ is idempotent complete, we get a symmetric monoidal exact functor $\Mod^\mathrm{proj}(K(R))\to \Mot^\sat(R)$, where $\Mod^\mathrm{proj}(K(R))$ is the idempotent completion of $\Mod^\mathrm{free}(K(R))$.
	 Since the $K$-theory of a connective ring spectrum $A$ is equivalent to the group completion of the $E_\infty$-space $\iota_0\Mod^\mathrm{proj}(A)$, we obtain a canonical map of $E_\infty$-ring spectra $K(K(R))\to K(\Mot^\sat(R))$. 
	 It is not difficult to show that the composition
	 \[
	 K(K(R)) \to K(\Mot^\sat(R)) \xrightarrow{K(\ch)} K^{S^1}(\mathrm{HH}(R))
	 \]
	 is $K$ of the Dennis trace map $K(R)\to \mathrm{HH}(R)$. This can be used to detect nonzero elements in the homotopy groups $\pi_n\K^{(2)}(R)$ for $n\geq 0$. For example, if $R$ is a field of characteristic not $2$, $\{\pm 1\}=K_1(K(R))\to \pi_1\K^{(2)}(R)$ is injective.
\end{remark}

\providecommand{\bysame}{\leavevmode\hbox to3em{\hrulefill}\thinspace}

\end{document}